\documentclass[12pt]{amsart}

\usepackage{amssymb,latexsym,enumerate}    
\usepackage{fullpage}
\usepackage{graphicx}
\usepackage{xcolor}

\input{xypic}
\newcommand{\Mdef}[2]{\newcommand{#1}{\relax \ifmmode #2 \else $#2$\fi}}

\newcommand{\codim}{\mathrm{codim}}

\newcommand{\im}{\mathrm{im}}

\newcommand{\sm }{\wedge}

\newcommand{\tensor}{\otimes}

\newcommand{\Hom}{\mathrm{Hom}}

\newcommand{\Ext}{\mathrm{Ext}}
\Mdef{\bhom}{\mathbf{\hat{H}om}}

\Mdef{\Mod}{\mathrm{mod}}

\newcommand{\st}{\; | \;}

\newtheorem{thm}{Theorem}[section]
\newtheorem{lemma}[thm]{Lemma}
\newtheorem{prop}[thm]{Proposition}
\newtheorem{cor}[thm]{Corollary}

\theoremstyle{definition}

\newtheorem{defn}[thm]{Definition}

\newtheorem{example}[thm]{Example}

\newtheorem{remark}[thm]{Remark}

\numberwithin{equation}{section}

\newcommand{\qqed}{\qed \\[1ex]}
\renewenvironment{proof}[1][\hspace*{-.8ex}]{\noindent {\bf Proof #1:\;}}{\qqed}

\Mdef{\PH} {\Phi^H}
\Mdef{\PK} {\Phi^K}
\Mdef{\PL} {\Phi^L}
\Mdef{\PT} {\Phi^{\T}}

\Mdef{\ef}{E{\cF}_+}
\Mdef{\etf}{\widetilde{E}{\cF}}
\Mdef{\eg}{E{G}_+}
\Mdef{\etg}{\widetilde{E}{G}}

\newcommand{\etp}{\widetilde{E}\cP}

\newcommand{\piA}{\pi^{\cA}}

\Mdef{\infl}{\mathrm{inf}}
\Mdef{\defl}{\mathrm{def}}
\Mdef{\res}{\mathrm{res}}
\Mdef{\ind}{\mathrm{ind}}
\Mdef{\coind}{\mathrm{coind}}

\Mdef{\univ}{\mathcal{U}}

\Mdef{\Fp}{\mathbb{F}_p}
\Mdef{\Zpinfty}{\Z /p^{\infty}}
\Mdef{\Zpadic}{\Z_p^{\wedge}}

\newcommand{\bi}{\begin{itemize}}
\newcommand{\be}{\begin{enumerate}}
\newcommand{\bc}{\begin{center}}
\newcommand{\bd}{\begin{description}}
\newcommand{\ei}{\end{itemize}}
\newcommand{\ee}{\end{enumerate}}
\newcommand{\ec}{\end{center}}
\newcommand{\ed}{\end{description}}

\newcommand{\adjunction}[4]{
\diagram
#1:#2 \rrto<0.7ex> &&
#3  \llto<0.7ex> :#4 
\enddiagram}
\newcommand{\RLadjunction}[4]{
\diagram
#1:#2 \rrto<-0.7ex> &&
#3  \llto<-0.7ex> :#4 
\enddiagram}
\newcommand{\lra}{\longrightarrow}
\newcommand{\lla}{\longleftarrow}

\newcommand{\iso}{\cong}

\newcommand{\Gspectra}{\mbox{$G$-{\bf spectra}}}

\newcommand{\rings}{\mathbf{Rings}}

\Mdef{\we}{\mathbf{we}}
\Mdef{\fib}{\mathbf{fib}}
\Mdef{\cof}{\mathbf{cof}}
\Mdef{\BI}{\mathcal{BI}}
\newcommand{\Qeq}{\simeq_Q}

\newcommand{\fibre}{\mathrm{fibre}}

\newcommand{\colim}{\mathop{  \mathop{\mathrm {lim}} \limits_\rightarrow} \nolimits}
\newcommand{\holim}{\mathop{ \mathop{\mathrm {holim}} \limits_\leftarrow} \nolimits}

\Mdef{\A}{\mathbb{A}}
\Mdef{\B}{\mathbb{B}}
\Mdef{\C}{\mathbb{C}}
\Mdef{\D}{\mathbb{D}}
\Mdef{\E}{\mathbb{E}}
\Mdef{\T}{\mathbb{T}}
\Mdef{\F}{\mathbb{F}}
\Mdef{\G}{\mathbb{G}}
\Mdef{\I}{\mathbb{I}}
\Mdef{\N}{\mathbb{N}}
\Mdef{\Q}{\mathbb{Q}}
\Mdef{\R}{\mathbb{R}}
\Mdef{\bbS}{\mathbb{S}}
\Mdef{\Z}{\mathbb{Z}}

\Mdef{\bA}{\mathbb{A}}
\Mdef{\bB}{\mathbb{B}}
\Mdef{\bC}{\mathbb{C}}
\Mdef{\bD}{\mathbb{D}}
\Mdef{\bE}{\mathbb{E}}
\Mdef{\bF}{\mathbb{F}}
\Mdef{\bG}{\mathbb{G}}
\Mdef{\bH}{\mathbb{H}}
\Mdef{\bI}{\mathbb{I}}
\Mdef{\bJ}{\mathbb{J}}
\Mdef{\bK}{\mathbb{K}}
\Mdef{\bL}{\mathbb{L}}
\Mdef{\bM}{\mathbb{M}}
\Mdef{\bN}{\mathbb{N}}
\Mdef{\bO}{\mathbb{O}}
\Mdef{\bP}{\mathbb{P}}
\Mdef{\bQ}{\mathbb{Q}}
\Mdef{\bR}{\mathbb{R}}
\Mdef{\bS}{\mathbb{S}}
\Mdef{\bT}{\mathbb{T}}
\Mdef{\bU}{\mathbb{U}}
\Mdef{\bV}{\mathbb{V}}
\Mdef{\bW}{\mathbb{W}}
\Mdef{\bX}{\mathbb{X}}
\Mdef{\bY}{\mathbb{Y}}
\Mdef{\bZ}{\mathbb{Z}}

\newcommand{\cA}{\mathcal{A}}
\Mdef{\cB}{\mathcal{B}}
\Mdef{\cC}{\mathcal{C}}
\Mdef{\mcD}{\mathcal{D}} 
\Mdef{\cE}{\mathcal{E}}
\Mdef{\cF}{\mathcal{F}}
\Mdef{\cG}{\mathcal{G}}
\Mdef{\mcH}{\mathcal{H}} 
\Mdef{\cI}{\mathcal{I}}
\Mdef{\cJ}{\mathcal{J}}
\Mdef{\cK}{\mathcal{K}}
\Mdef{\mcL}{\mathcal{L}}

\Mdef{\cM}{\mathcal{M}}
\Mdef{\cN}{\mathcal{N}}
\Mdef{\cO}{\mathcal{O}}
\Mdef{\cP}{\mathcal{P}}
\Mdef{\cQ}{\mathcal{Q}}
\Mdef{\mcR}{\mathcal{R}}
\Mdef{\cS}{\mathcal{S}}
\Mdef{\cT}{\mathcal{T}}
\Mdef{\cU}{\mathcal{U}}
\Mdef{\cV}{\mathcal{V}}
\Mdef{\cW}{\mathcal{W}}
\Mdef{\cX}{\mathcal{X}}
\Mdef{\cY}{\mathcal{Y}}
\Mdef{\cZ}{\mathcal{Z}}

\Mdef{\At}{\tilde{A}}
\Mdef{\Bt}{\tilde{B}}
\Mdef{\Ct}{\tilde{C}}
\Mdef{\Et}{\tilde{E}}
\Mdef{\Ht}{\tilde{H}}
\Mdef{\Kt}{\tilde{K}}
\Mdef{\Lt}{\tilde{L}}
\Mdef{\Mt}{\tilde{M}}
\Mdef{\Nt}{\tilde{N}}
\Mdef{\Pt}{\tilde{P}}

\Mdef{\tA}{\tilde{A}}
\Mdef{\tB}{\tilde{B}}
\Mdef{\tC}{\tilde{C}}
\Mdef{\tE}{\tilde{E}}
\Mdef{\tH}{\tilde{H}}
\Mdef{\tK}{\tilde{K}}
\Mdef{\tL}{\tilde{L}}
\Mdef{\tM}{\tilde{M}}
\Mdef{\tN}{\tilde{N}}
\Mdef{\tP}{\tilde{P}}

\Mdef{\ft}{\tilde{f}}
\Mdef{\xt}{\tilde{x}}
\Mdef{\yt}{\tilde{y}}

\Mdef{\Ab}{\overline{A}}
\Mdef{\Bb}{\overline{B}}
\Mdef{\Cb}{\overline{C}}
\Mdef{\Db}{\overline{D}}
\Mdef{\Eb}{\overline{E}}
\Mdef{\Fb}{\overline{F}}
\Mdef{\Gb}{\overline{G}}
\Mdef{\Hb}{\overline{H}}
\Mdef{\Ib}{\overline{I}}
\Mdef{\Jb}{\overline{J}}
\Mdef{\Kb}{\overline{K}}
\Mdef{\Lb}{\overline{L}}
\Mdef{\Mb}{\overline{M}}
\Mdef{\Nb}{\overline{N}}
\Mdef{\Ob}{\overline{O}}
\Mdef{\Pb}{\overline{P}}
\Mdef{\Qb}{\overline{Q}}
\Mdef{\Rb}{\overline{R}}
\Mdef{\Sb}{\overline{S}}
\Mdef{\Tb}{\overline{T}}
\Mdef{\Ub}{\overline{U}}
\Mdef{\Vb}{\overline{V}}
\Mdef{\Wb}{\overline{W}}
\Mdef{\Xb}{\overline{X}}
\Mdef{\Yb}{\overline{Y}}
\Mdef{\Zb}{\overline{Z}}

\Mdef{\db}{\overline{d}}
\Mdef{\hb}{\overline{h}}
\Mdef{\qb}{\overline{q}}
\Mdef{\rb}{\overline{r}}
\Mdef{\tb}{\overline{t}}
\Mdef{\ub}{\overline{u}}
\Mdef{\vb}{\overline{v}}

\Mdef{\hc}{\hat{c}}
\Mdef{\he}{\hat{e}}
\Mdef{\hf}{\hat{f}}
\Mdef{\hA}{\hat{A}}
\Mdef{\hH}{\hat{H}}
\Mdef{\hJ}{\hat{J}}
\Mdef{\hM}{\hat{M}}
\Mdef{\hP}{\hat{P}}
\Mdef{\hQ}{\hat{Q}}

\Mdef{\thetab}{\overline{\theta}}
\Mdef{\phib}{\overline{\phi}}

\Mdef{\uA}{\underline{A}}
\Mdef{\uB}{\underline{B}}
\Mdef{\uC}{\underline{C}}
\Mdef{\uD}{\underline{D}}

\Mdef{\bolda}{\mathbf{a}}
\Mdef{\boldb}{\mathbf{b}}
\Mdef{\boldD}{\mathbf{D}}

\Mdef{\fm}{\frak{m}}
\Mdef{\fp}{\frak{p}}

\Mdef{\eps}{\epsilon}

\newcommand{\mccM}{\M}

\newcommand{\Rt}{R_t}

\newcommand{\Rtop}{R_{top}}
\newcommand{\Rttop}{\widetilde{R}_{top}}

\newcommand{\Ra}{R_a}

\renewcommand{\Et}{\cE_t}

\newcommand{\cOtcF}{\tilde{\cO}_{\cF}}

\newcommand{\cOcF}{\cO_{\cF}}
\newcommand{\cOcFG}{\cO_{\cF /G}}
\newcommand{\cOcFH}{\cO_{\cF /H}}
\newcommand{\cOcFK}{\cO_{\cF /K}}
\newcommand{\cOcFL}{\cO_{\cF /L}}

\newcommand{\M}{\bM}

\newcommand{\cell}{\mbox{cell-}}
\newcommand{\modcat}[1]{\mbox{$#1$-$\mathrm{mod}$}}
\newcommand{\torsmodcat}[1]{\mbox{tors-$#1$-$\mathrm{mod}$}}

\newcommand{\qcemodcat}[1]{\mbox{qce-$#1$-mod}}
\newcommand{\pqcemodcat}[1]{\mbox{pqce-$#1$-mod}}
\newcommand{\cellmodcat}[1]{\mbox{cell-$#1$-mod}}

\newcommand{\modcatG}[1]{\mbox{$#1$-mod-$G$-spectra}}
\newcommand{\cellmodcatG}[1]{\mbox{cell-$#1$-mod-$G$-spectra}}

\newcommand{\modcatGK}[1]{\mbox{$#1$-mod-$G/K$-spectra}}

\newcommand{\modcatGG}[1]{\mbox{$#1$-mod-spectra}}
\newcommand{\cellmodcatGG}[1]{\mbox{cell-$#1$-mod-spectra}}

\newcommand{\RtimodGspectra}{\mbox{$\Rti$-mod-$G$-spectra}}

\newcommand{\Rmodspectra}{\mbox{$R$-mod-spectra}}

\newcommand{\AticellAtimodGspectra}{\mbox{$\Ati$-cell-$\Ati$-mod-$G$-spectra}}
\newcommand{\AtimodGspectra}{\mbox{$\Ati$-mod-$G$-spectra}}
\newcommand{\Amodspectra}{\mbox{$A$-mod-spectra}}

\newcommand{\Amod}{\modcatG{A}}

\newcommand{\AKmod}{\modcatGK{A^K}}

\newcommand{\Ramod}{\modcat{\Ra}}

\newcommand{\Rtmod}{\modcat{\Rt}}

\newcommand{\cellRamod}{\cellmodcat{\Ra}}

\newcommand{\mccR}{R}

\newcommand{\cellHBGmodp}{\cellmodcat{H^*(BG)}_p}
\newcommand{\cellHBGmodi}{\cellmodcat{H^*(BG)}_i}
\newcommand{\torsHBGmod}{\torsmodcat{H^*(BG)}}
\newcommand{\HBGmodi}{\modcat{H^*(BG)}_i}
\newcommand{\HBGmodp}{\modcat{H^*(BG)}_p}

\newcommand{\Rttopmod}{\modcatG{\Rttop}}

\newcommand{\cellRttopmod}{\cellmodcatG{\Rttop}}

\newcommand{\DEGmod}{\modcatG{\DH EG_+}}
\newcommand{\DBGmod}{\modcatGG{\DH BG_+}}

\newcommand{\Rtopmod}{\modcat{\Rtop}}
\newcommand{\cellRtopmod}{\cellmodcatGG{\Rtop}}

\newcommand{\SpOG}{G\mbox{Sp}^O}
\newcommand{\SpLOG}{G\mbox{Sp}_{\mcL}^O}

\newcommand{\SpLOGK}{G/K\mbox{Sp}_{\mcL}^O}

\newcommand{\SpLOone}{\mbox{Sp}_{\mcL}^O}

\newcommand{\SpSigma}{\mbox{Sp}^{\Sigma}}

\newcommand{\efp}{E\cF_+}

\newcommand{\efhp}{E\cF/H_+}
\newcommand{\efkp}{E\cF/K_+}
\newcommand{\eflp}{E\cF/L_+}

\newcommand{\siftyV}[1]{S^{\infty V(#1)}}

\newcommand{\lr}[1]{\langle #1\rangle}

\newcommand{\sifty}[1]{S^{\infty V(#1)}}

\newcommand{\connsubG}{\mathbf{ConnSub(G)}}

\newcommand{\Kos}{\mathrm{Kos}}

\newcommand{\cEi}{\cE^{-1}}

\renewcommand{\DH}{D}

\newcommand{\cBI}{\mathcal{BI}}

\newcommand{\ccI}{\bfD}

\newcommand{\bbarmccR}{\overline{\mccR}}

\newcommand{\bbU}{\mathbb U}
\newcommand{\FF}{\mathbb F}

\newcommand{\Ati}{\tilde{A}}
\newcommand{\Rti}{\tilde{R}}

\newcommand{\bfD}{\mathbf{D}}

\newcommand{\GI}{\mathcal{GI}}

\newcommand{\alphat}{\tilde{\alpha}}
\newcommand{\betat}{\tilde{\beta}}

\newcommand{\Ho}{\mathrm{Ho}}

\newcommand{\fbbS}[1]{\bbS_{\geq #1}}

\newcommand{\Cf}{C_f}
\newcommand{\Ci}{C_i}
\newcommand{\Cif}{C_{if}}
\newcommand{\PCf}{PC_f}
\newcommand{\PCi}{PC_i}
\newcommand{\PCif}{PC_{if}}

\newcommand{\lnz}{lnz}
\newcommand{\Rtilde}{\tilde{R}}
\newcommand{\DGAs}{\mathrm{DGAs}}
\newcommand{\ist}{i_{\sigma}^{\tau}}
\newcommand{\Rhat}{\hat{R}}

\newcommand{\cAs}{\cA^s}

\newcommand{\smL}{\sm_{\mcL}}
\newcommand{\RRc}{\mathbb{R}_c}
\newcommand{\RRd}{\mathbb{R}_d}

\newcommand{\einfl}{\widetilde{\mathrm{inf}}}
\newcommand{\smb}{\overline{\sm}}

\begin{document}
\title{ An algebraic model for rational torus-equivariant spectra} 
\author{J.~P.~C.~Greenlees}
\address{Department of Pure Mathematics, The Hicks Building, 
Sheffield S3 7RH. UK.}
\email{j.greenlees@sheffield.ac.uk}

\author{B.~Shipley}
\thanks{The authors are grateful to D.~Barnes, A.~Blumberg, M.~Hill,
  M.~Mandell, J.P.~May and M.~Kedziorek for numerous conversations
  about this work, to the Isaac Newton Institute (Cambridge) during
  the early stages of the work, and to MSRI (Berkeley) during recent
  revisions, for providing excellent working environments. We are also
  grateful to the referee for an attentive and thoughtful reading. 
The first author is grateful for support under EPSRC grant  number EP/H040692/1. 
The second author was supported in part by 
NSF Grants DMS-0417206, DMS-0706877, DMS-1104396, and DMS-1406468 
and by a Sloan Research Fellowship.}
\address{Department of Mathematics, Statistics and Computer Science, University of Illinois at
Chicago, 508 SEO m/c 249,
851 S. Morgan Street,
Chicago, IL, 60607-7045, USA}
\email{bshipley@math.uic.edu}

\date{}
\maketitle

\begin{abstract}
We provide a universal de Rham model for rational $G$-equivariant cohomology
theories for an arbitrary torus $G$. More precisely, we show that the 
representing category, of rational $G$-spectra,   is Quillen
equivalent to the explicit small and calculable algebraic model $d\cA
(G)$ of differential graded objects in the category $\cA (G)$
introduced in \cite{tnq1}. 
\end{abstract}

\tableofcontents
\part{Introduction}
\section{Overview}
\subsection{Preamble}
Cohomology theories are contravariant homotopy functors  on topological
spaces satisfying the Eilenberg-Steenrod axioms (except for the
dimension axiom), and any cohomology theory $E^*(\cdot )$ is represented by a
homotopy theoretic spectrum $E$ in the sense that $E^*(X)=[X,E]^*$. 
Accordingly, the category of spectra gives an
embodiment of the category of cohomology theories in which one can do homotopy theory. 
The complexity of the homotopy theory of spectra is visible even in 
the homotopy endomorphisms of the unit object: this is the ring of 
stable homotopy groups of spheres, which is  so intricate that we  cannot expect 
a complete analysis of the category of spectra in general.  However, most of the complication comes from
$\Z$-torsion so we can  simplify things by rationalizing. The
resulting category of rational spectra represents cohomology theories
with values in rational vector spaces. The simplicity of this
rationalized category is apparent by Serre's theorem: the rationalization of the stable 
homotopy groups of spheres simply consists of $\Q$ in degree 0, and 
it is a small step to see that there is nothing more to the topology
of rational
cohomology  theories than their graded rational vector space of
coefficients.  On the other hand, de Rham cohomology shows that a large amount of 
useful geometry remains even when we rationalize. Accordingly, 
the study of rational cohomology
 theories and rational spectra is both accessible and useful. 

These facts are  well-known, and it is natural to ask what happens 
when we consider spaces with an action of a compact Lie group $G$. 
Once again, a $G$-equivariant cohomology theory is a contravariant 
homotopy functor on $G$-spaces satisfying suitable conditions, and 
each such $G$-equivariant cohomology theory is represented by a
$G$-spectrum \cite{LMS(M)}. In the equivariant case, when we rationalize a 
$G$-spectrum, considerably more structure remains than in the
non-equivariant case.  It is natural to expect rational 
representation theory to play a role in understanding rational
equivariant cohomology theories, and when $G$ is finite this is the 
only ingredient. However in general, the other significant piece of 
structure is exemplified by the Localization Theorem: for a torus $G$
this  states that (for finite complexes) there is no difference between
the Borel cohomology of a $G$-space and its
$G$-fixed points once the Euler classes are inverted. These ingredients
can be used to build the algebraic model \cite{tnq1} for rational $G$-spectra
described in Section \ref{sec:standard} below. 

The archetype for giving an algebraic model for the homotopy theory of
topological origin is Quillen's analysis of simply connected rational 
spaces \cite{Quillen}. To prove the result, he introduced the axiomatic
framework of model categories which underly the homotopy category, 
and the notion of a Quillen equivalence between model categories
preserving the homotopy theories. The use of these ideas is now 
widespread, and we refer to \cite{hovey-model} and \cite{hh} for
details.   

Our main result is a Quillen equivalence between the category of
rational $G$-spectra for a torus $G$ and an explicit and calculable algebraic
model. In the course of our proof, we introduce a number of techniques of
broader interest, in equivariant homotopy theory and in the theory of model
categories. In the rest of the introduction, we give a little history,
and then describe our results, methods and conventions.

\subsection{Equivariant cohomology theories}
Non-equivariantly, rational stable homotopy theory is very simple: 
the homotopy category of rational spectra is equivalent to 
the category of graded rational vector spaces, and all cohomology 
theories are ordinary in the sense that they are naturally equivalent
to ordinary cohomology with coefficients in a graded vector space.
The first author has conjectured \cite{gqsurvey} that for each compact Lie group $G$, there
is an abelian category $\cA (G)$, so that the homotopy category of 
rational $G$-spectra is equivalent to its derived category $D(\cA
(G))$, i.e., to the  homotopy category of the category $d\cA(G)$
consisting of differential graded objects of $\cA (G)$:
$$\mathrm{Ho}(\mbox{$G$-spectra}/\Q ) \simeq \mathrm{Ho}(d\cA
(G))=D(\cA (G)).$$
In general terms, the objects of $\cA (G)$ are sheaves of graded
modules with additional structure over the space
of closed subgroups of $G$, with the fibre over $H$ giving information 
about the geometric $H$-fixed points. The conjecture describes various
properties of $\cA (G)$, and in particular
asserts that its injective dimension is equal to the rank of $G$. 
According to the conjecture one may therefore expect  to  make complete calculations 
in rational equivariant stable homotopy theory, and to classify cohomology 
theories. Indeed, one can construct a cohomology theory
by writing down a differential graded object in $\cA (G)$: this is how $SO(2)$-equivariant 
elliptic cohomology was constructed in \cite{ellT}, and it is hoped 
to construct cohomology theories associated to generic curves of higher genus
in a similar way using the results of this paper. 

The conjecture is elementary for finite groups, where 
$\cA (G)=\prod_{(H)}\mbox{$\Q W_G(H)$-mod}$ \cite[Theorems A.7, A.8, A.9]{Tate}, where
the product is over conjugacy classes of subgroups $H$ and  
$W_G(H)=N_G(H)/H$. This means that  any cohomology theory is again
ordinary in the sense that it is a sum over conjugacy classes $(H)$ of
ordinary cohomology of the $H$-fixed points with coefficients in a
graded $\Q W_G(H)$-module. The conjecture has been 
proved for the rank 1 groups $G=SO(2), O(2), SO(3)$ in \cite{s1q, o2q, so3q}, 
where $\cA (G)$ is more complicated. 
It is natural to go on to conjecture that the equivalence comes from a 
Quillen equivalence
$$\mbox{$G$-spectra}/\Q  \simeq d\cA (G), $$
for suitable model structures. 
{The second author proved that for $G=SO(2)$ the Quillen equivalence
would follow from  a triangulated equivalence on the derived
categories \cite{qs1q}. It was claimed in \cite{s1q} that the
equivalence of homotopy categories was in fact a triangulated
equivalence, but the proof is incomplete, and subsequent work of
Patchkoria \cite{Patchkoria} shows that the method of \cite{s1q} is
insufficient. In any case, there is no  prospect of extending the methods of 
\cite{s1q} or \cite{qs1q} to higher rank. Even if one only wants an equivalence of triangulated categories, 
it appears essential to establish the Quillen equivalence when $r \geq 2$.
Building on the present work,  
Barnes \cite{Barnes1, Barnes2} has shown how to deduce the Quillen
equivalence for $G=O(2)$
from a suitable proof for $G=SO(2)$ (such as the one we use here), and
Kedziorek \cite{Kedziorek} has done so for $G=SO(3)$. 

Recently, Barnes, Kedziorek and the present authors have given a
separate account of a Quillen equivalence for the case $G=SO(2)$
\cite{s1qmon}. This has the merit of avoiding the massive complication
 due to the complexity of the space of connected subgroups for a general torus,
 and also  gives a stronger conclusion than the specialization of our result here,
since the equivalence is monoidal.  }

\subsection{The classification theorem}
The present paper completes the programme begun in \cite{tnq1,tnq2}
and supported by \cite{cellprin,modulefps,diagrammodcats,AGs}. 
The purpose of the series is to provide a small and calculable algebraic model for rational 
$G$-equivariant cohomology theories for  a torus $G$ 
of rank $r \geq 0$. Such cohomology theories are represented by rational 
$G$-spectra, and in this paper we show that the category of 
rational $G$-spectra is Quillen equivalent to the small and concrete abelian 
category $\cA (G)$ introduced in \cite{tnq1}
(its definition and properties are summarized  in Section \ref{sec:standard}).
 The category 
$\cA (G)$ is designed as a natural 
target of a homology theory
$$\piA_*: \Gspectra \lra \cA (G); $$
the idea is that  $\cA (G)$ is a category of sheaves of modules, with 
the stalk over a closed subgroup $H$ being the Borel cohomology 
of the geometric $H$-fixed point set with suitable coefficients. 
A main theorem of  \cite{tnq1} shows that $\cA (G)$ is
of finite injective dimension (shown in \cite{tnq2} to be $r$).

The main  theorem of the present paper and the culmination of the 
series is as follows. Model structures will be described 
in Sections \ref{sec:modelsofspectra} and \ref{sec:algebraicmodels}  below. 

\begin{thm}
\label{thm:culmination}
For any torus $G$, there is a Quillen equivalence
$$\Gspectra/\Q \Qeq d\cA (G) $$
of model categories. In particular their homotopy categories are 
equivalent
$$Ho(\Gspectra/ \Q) \simeq Ho(d\cA (G) ) =D(\cA(G))$$
as triangulated categories. 

\end{thm}

\begin{remark}
The functors involved in these Quillen equivalences are monoidal,
but their interaction with the model structures is not straightforward.
For this reason,  the extension of this result to Quillen equivalences
on the associated categories of monoids will be discussed  elsewhere
(as done in \cite{s1qmon} in the rank 1 case). 

\end{remark}

Because of the nature of Theorem \ref{thm:culmination},  it is easy to impose restrictions 
on the isotropy groups occurring in topology and algebra,  
and one may deduce versions of this theorem for categories 
of spectra with restricted isotropy groups. For example we recover
a special case of the result of \cite{gfreeq}, which states that if $G$ is any connected compact Lie
group there is a Quillen equivalence
$$ \mbox{free-$\Gspectra/\Q$}\Qeq \mbox{DG-torsion-$H^*(BG)$-modules},  $$
with a quite different proof. The methods of the present paper are
used to  extend the result on free $G$-spectra  to disconnected groups
$G$ in \cite{gfreeq2}. 

\subsection{Applications}
Beyond the obvious structural insight, the type of applications we 
anticipate may be seen from those already given for the circle group $\T$ 
(i.e., the case $r=1$). For example 
\cite{s1q} gives a classification of 
rational $\T$-equivariant cohomology theories, a precise formulation 
and proof of the rational $\T$-equivariant Segal conjecture, and an algebraic
analysis of existing theories,  such as $K$-theory. More 
significant is the construction in \cite{ellT} of a rational equivariant
cohomology theory associated to an elliptic curve $C$ over a $\Q$-algebra, 
and the identification of a part of $\T$-equivariant stable homotopy theory 
modelled on the derived category of sheaves over $C$. The philosophy
in which equivariant cohomology theories correspond to algebraic groups
is expounded in \cite{GEGA}, and there are encouraging signs suggesting
that one may use the model described in the present paper to construct  
torus-equivariant cohomology theories associated to  generic complex
curves of higher genus.

\subsection{Outline of strategy}
The general strategy makes use of the existence of a good symmetric 
monoidal model category of spectra, allowing 
us to talk about commutative ring spectra and modules over
them. However, it is convenient that the commutative monoids in our category 
should include spectra that are  only $E_{\infty}$-rings 
 (i.e., algebras over a non-equivariant $E_{\infty}$-operad, which is
 to say  $E_{\infty}'$  in the sense of McClure \cite{McClure}). 
Accordingly we work in the Blumberg-Hill category of
$\mcL$-spectra in orthogonal spectra. The particular properties we need
are summarized in Proposition \ref{prop:axiom}. 
The rest of this introduction will outline the strategy without
mentioning detailed requirements of the model. 

There are two ingredients at the core of our argument, that led us to 
believe we could establish a Quillen equivalence. The first author's
\cite{tnq1}  construction of the algebraic
category $\cA (G)$ giving the basis of an effcient algebraic 
Adams spectral sequence suggested differential objects in $\cA (G)$ as
an algebraic model. However it is the second author's results
\cite{s-alg}  (giving Quillen equivalences between commutative
algebras over the Eilenberg-Mac~Lane spectrum $H\Q$ and differential 
graded commutative $\Q$-algebras, and between the  module categories
of the corresponding algebras) that gives the bridge  allowing us to pass from topology to
algebra. 

In outline, what we have to achieve is to move from the category of 
rational $G$-spectra to the category of DG objects of the abelian category $\cA (G)$.
There are five main stages to this, which we first describe and then
illustrate on a chain of Quillen equivalences.
\begin{description}
\item[(1) Isotropy separation (Sections \ref{sec:ringsandmodules} to \ref{sec:formalcube})] 
Rational $G$-spectra are modules over the rational sphere spectrum
$\bbS$. The rational sphere spectrum is the homotopy pullback of a
diagram $\Rttop$ of 
isotropically simpler commutative ring $G$-spectra. Accordingly, by 
the methods of \cite{diagrammodcats}, the category of
$\bbS$-modules is equivalent to a category of diagrams of  modules  over the
pullback diagram of ring $G$-spectra. 

The diagram $\Rttop$ has the  shape of  a punctured $(r+1)$-cube, which we call the `formal' punctured
cube $\PCf$. The module category of each individual ring spectrum captures isotropical information about subgroups with a specified
dimension and the diagram shows how to reassemble this isotropically local 
information into a global spectrum.  

\item[(2) Removal of equivariance (Section \ref{sec:removeequiv})] 
At each point in the diagram, we replace the
 commutative ring $G$-spectrum by a commutative non-equivariant ring spectrum by passage to $G$-fixed
points, and show that the module categories are equivalent  using the general methods described in \cite{modulefps}. 
\item[(3) Transition to algebra  (Section \ref{sec:spectratoDGAs})] 
At each point in the diagram, we
apply the machinery of \cite{s-alg} to replace all the commutative 
ring spectra in the diagram by commutative DGAs, and the category
of module spectra by the corresponding category of DG-modules over
the DGAs. 
\item[(4) Rigidity (Section \ref{sec:formality})] 
The diagram of commutative DGAs is intrinsically formal in the sense
that it is determined up to equivalence by its homology. Accordingly
the diagram of commutative DGAs may be replaced by a diagram of 
commutative algebras. 
\item[(5) Simplification (Sections \ref{sec:algebraicmodels}  
and \ref{sec:algcells})] 
  At each stage so far, we have used cellularization to pick out the
  relevant homotopy category as  the localizing subcategory built
from certain specified `cells'. The final step is to
replace this cellularization of the category of DG-modules over the diagram 
of commutative rings by a much smaller category of modules with 
special properties, so that no cellularization is necessary; using
apparatus from \cite{AGs}, this category turns out to be  $\cA (G)$. 
\end{description}

These steps correspond to the following
sequence of Quillen equivalences, several of which are themselves
zig-zags of simple Quillen equivalences. The cellularizations are all
with respect to the set of images of the cells $G/H_+$ as $H$ runs
through closed subgroups, and the diagrams of rings
are all punctured $(r+1)$-cubes. 
\begin{multline*}
\Gspectra 
\stackrel{(1)} \simeq \cellmodcatG{\Rttop}
\stackrel{(2)} \simeq \cellmodcatGG{\Rtop}
\stackrel{(3)} \simeq \cellmodcat{\Rt}\\
\stackrel{(4)} \simeq \cellmodcat{\Ra}
\stackrel{(5)} \simeq \pqcemodcat{\Ra} 
\stackrel{(5)}\simeq \cA(G)
\end{multline*}

It is worth highlighting some of the techniques of more general applicability. 

First, we constantly use the Cellularization Principle \cite{cellprin}. The idea is
that a Quillen adjunction induces a Quillen equivalence between
cellularized model categories, provided we cellularize with respect to
cells that are small and correspond under the adjunction. The
hypotheses are mild, and  it may appear like a tautology, but it has been useful innumerable times in the 
present paper and deserves emphasis. It can be directly compared to another extremely
powerful formality: a natural transformation of cohomology
theories that is an isomorphism on spheres is an equivalence. 

Second, we make extensive use of categories of modules over diagrams
of rings \cite{diagrammodcats}, and prove that up to Quillen equivalence and
cellularization, we can replace a category of modules over a  diagram
of rings by the category of modules over its pullback.

Third, the fact that if $A$ is a ring $G$-spectrum,  passage to categorical
$K$-fixed points establishes a
close relationship between the category of $A$-module $G$-spectra and 
the  category of $A^K$-module $G/K$-spectra \cite{modulefps}. More precisely, we consider 
a Quillen adjunction
$$
\adjunction{A\tensor_{A^K}(\cdot)}
{\AKmod}{\Amod}{(\cdot )^K}.  $$
This is especially effective in conjunction with  the Cellularization Principle.
 
Finally, we note that at the centre of the proof is rigidity: any 
two model categories with suitable specified homotopy level properties 
are equivalent. The equivariant sphere ring spectrum should be
viewed as the sheaf of functions on a  non-affine variety; we find a
cover by affine varieties which are individually rigid, and the
configuration of the cover is also rigid. 

In effect, we have used only one basic rigidity result: any two commutative DGAs
which have the same polynomial cohomology are quasi-isomorphic. This elementary 
result has far reaching consequences.  Our main use of it here is to 
patch together local rigidity results (each based on polynomial rings)
to give a global rigidity result. In \cite{gfreeq} we applied it to 
prove rigidity of Koszul duals. We also need a rigidity result for
modules, that by an Adams spectral sequence argument, the standard 
cells are determined by their homology \cite[12.1]{tnq1}.

\subsection{Relationship to other results}

We should explain the relationship between the strategy implemented
here and that used for free spectra in  \cite{gfreeq}. Both strategies
start with a category of $G$-spectra and end with a purely algebraic
category, and the connection in both relies on finding an intermediate
category which is visibly rigid in the sense that it is determined
by its homotopy category (the archetype of this is the category 
of modules over a commutative DGA with polynomial cohomology). 

The difference comes in the route taken. Roughly speaking,
the strategy in \cite{gfreeq} is to move to non-equivariant spectra as
soon as possible, whereas that adopted here is to keep working in 
the ambient category of $G$-spectra for as long as possible. 

The advantage of the strategy of \cite{gfreeq} is that it is close
to commutative algebra, and should be adaptable to proving uniqueness
of other algebraic categories. However, 
it is hard to retain control of the monoidal structure, and 
adapting the method to deal
 with many isotropy groups makes the formal framework very complicated. 
This was our original approach to the result for tori. 

The present method appears to have several advantages. It uses fewer
steps, and the monoidal structures are visible throughout. Furthermore, 
it reflects traditional approaches to the homotopy theory of $G$-spaces
in that it  displays the category of $G$-spectra as built from categories
of spectra with restricted isotropy group using Borel cohomology.

{Finally, we should explain that early versions of the present paper
(specifically arXiv:1101.2511 v1, v2, v3 posted in 2011) differed from
the present one in two important respects. Firstly, they included in
condensed form the parts of  \cite{cellprin, modulefps,
  diagrammodcats} that they required; we separated out those papers
partly to improve readability and partly because they appeared to be
of wider interest. During the process of revising this paper to take
advantage of the separation, we found a signficant simplification, and
this led to the  second main difference.  The method for dealing with the
equivalence between the category of $G$-spectra and a category of
diagrams is much simpler here than in the earlier versions
 because the diagrams themselves are finite. In the present version, the
 manipulations with diagrams are now largely replaced by an equivalence of
$G$-spectra showing how the sphere spectrum $\bbS$ can be constructed from isotropically 
simpler pieces. Having made that change, it was necessary to refer to 
the paper \cite{AGs} for the behaviour of an algebraic torsion
functor. 

\subsection{Conventions}
Certain conventions are in force throughout the paper. The most important is that {\em everything is rational}: 
henceforth all  spectra  and homology theories are rationalized without
comment.  For example, the category of rational $G$-spectra will now 
be denoted `$\Gspectra$'.  
Whenever possible we work in the derived category; for example, most
equivalences are verified at this level.
We also use the standard conventions that
`DG' abbreviates `differential graded' and that `subgroup' means 
`closed subgroup'. We attempt to let inclusion of subgroups follow
the alphabet, so that $G \supseteq H \supseteq K \supseteq L$. 

We often have to discuss classifying spaces of quotient groups, such
as $G/K$. We omit brackets and write  $BG/K=B(G/K)$. This should cause
no confusion because the only natural action of $K$ on  $BG$  is the
trivial  action (so we never have cause to make the construction $(BG)/K$).  

We focus on homological (lower) degrees, with differentials reducing degrees;
for clarity, cohomological (upper) degrees are called {\em codegrees} and 
may be converted to degrees  by negation in the usual way.
Finally, we write $H^*(X)$ for the unreduced cohomology of a space $X$
with rational coefficients.

We have adopted a number of more specific
conventions in our choice of notation, 
and it may help the reader to be alerted to them.

\begin{itemize}
\item There are several cases where we need to talk about 
ring $G$-spectra $\widetilde{R}$ and their fixed points
$R=(\widetilde{R})^G$. The equivariant form is indicated by 
a tilde on the non-equivariant one. 

\item We need to discuss rings in various categories of spectra, 
and then modules over them. Since it often needs to be made explicit, 
we write, for example, $R$-module-$G$-spectra for the category of 
$R$-modules in the category of $G$-spectra. 


\item The purpose of this paper is to give an algebraic model of a 
topological phenomenon. Accordingly, characters arise in various 
worlds, and it is useful to know they play corresponding roles. We
sometimes point this out by use of subscripts. For example
$R_a$ (with `$a$' for `algebra') might be a (conventional, graded) ring, 
$R_{top}$ its counterpart in spectra, $\tilde{R}_{top}$ its counterpart 
in $G$-spectra, and $R_t$ its counterpart in $DG$-algebra (a large
DGA, that is only described indirectly). 
  
\item We often have to discuss diagrams of rings and diagrams of modules 
over them, but we will usually say `$R$ is a diagram of rings' and
`$M$ is an $R$-module' (leaving the fact that $M$ is also a diagram to 
be deduced from the context). 

\end{itemize}

\subsection{Organization of the paper}\label{1.I}
Section \ref{sec:standard}  recalls the definition of the algebraic
model $\cA (G)$. Section \ref{sec:modelsofspectra}   
discusses the properties we need of our monoidal model of equivariant spectra,
and introduces the  Blumberg-Hill model we use ($\mcL$-spectra in
orthogonal spectra): beyond the good properties of orthogonal spectra
this has the property that $E_{\infty}$-ring $G$-spectra
are the commutative monoids. 

Section \ref{sec:ringsandmodules}  introduces 
the formalism for discussing modules over diagrams of rings.

In Section \ref{sec:isotropiccube} we explain that the sphere spectrum
is the homotopy pullback of a punctured $(r+1)$-cube of isotropically 
simpler ring spectra, and in  Section \ref{sec:formalcube} we explain that 
it is the homotopy pullback of a closely related punctured $(r+1)$-cube diagram $\Rttop$
of ring spectra  which are formal in the sense that they are determined by their
homotopy. This punctured cube is $\PCf$, and all the subsequent
diagrams have this shape. The results of \cite{diagrammodcats} then establish Equivalence (1), showing
that the category of rational $G$-spectra is equivalent to a category
of module $G$-spectra over the diagram  $\Rttop$ of ring $G$-spectra. 
This completes the isotropy separation step of the proof. 

Until this point, all arguments and calculations are within the
category of $G$-spectra. The remaining steps change ambient
categories. We not only need to recognize the categories of modules, 
but we also need to recognize the cells we use to cellularize them.
The fact that the natural cells $G/H_+$ are characterized by their 
homology (\cite[12.1]{tnq1}) means that we do not need to comment
further on the cells. 

Having shown the category of $G$-spectra is equivalent to a category of
modules over the diagram $\Rttop$ of ring $G$-spectra, we can move from 
$G$-spectra to non-equivariant spectra in Section
\ref{sec:removeequiv}, using the results of \cite{modulefps} to establish that this category is equivalent
to a category of modules over the diagram $\Rtop =(\Rttop)^G$ of ring
spectra (i.e., Equivalence (2)). In Section \ref{sec:spectratoDGAs} we
use the results of \cite{s-alg}  to establish that the category of
$\Rtop$-modules is equivalent to a category of modules over the
diagram $\Rt$ of DGAs (i.e.,  Equivalence (3)). It is then quite straightforward to establish Equivalence (4), showing
in Section \ref{sec:formality}
that the $\PCf$-diagram $\Ra=H_*(\Rt)$ is intrinsically
formal, so that the category of modules over $\Rt$ and $\Ra$
are equivalent. 

 In Section \ref{sec:AGasmodules}  we recognize our progress by
seeing  that  $\cA (G)$  can be viewed as a category of modules over
the diagram $\Ra$ of graded rings. Finally Sections \ref{sec:algebraicmodels} 
and \ref{sec:algcells} establish Equivalence (5), 
showing that the cellularization is equivalent to the  particular category
$\cA (G)$ of DG-$\Ra$-modules.

\section{The algebraic model}
\label{sec:standard}
In this section we recall relevant results from \cite{tnq1} which 
constructs an abelian category $\cA (G)$ giving an algebraic
reflection of the structure of the category 
of $G$-spectra and an Adams spectral sequence based on it; the present
account is very compressed and readers may need to refer to \cite{tnq1} for
details. The structures from that analysis will be relevant to much of what
we do here. 

This model is based on pairs of connected subgroups
and is denoted $\cA_c^p(G)$ in the more precise notation of
\cite{AGs}, and we use this form of the model since it is the most convenient and
practical model for calculations. In fact the first  output of the
topological argument is a  model based on flags of dimensions
of  subgroups  which is denoted $\cA_d^f(G)$ in \cite{AGs}. This was introduced and shown to be equivalent to
$\cA_c^p(G)$ in   \cite{AGs}; building on \cite{AGs}, we show in Section 
\ref{sec:AGasmodules} how to move directly from the algebraic model
coming from our proof (namely $\cA_d^f(G)$)  to $\cA_c^p(G)$.

\subsection{Definition of the category}
\label{subsec:defnAG}
First we must construct the category $\cA (G)$, which is a  category of modules
over a diagram of rings.  For a category $\bfD$ and a diagram $R
:\bfD \lra \rings$ of rings,   an $R$-module is given by a $\bfD$-diagram $M$ such that $M(x)$ is an
$R(x)$-module for each object $x$ in $\bfD$, 
and for every morphism $a: x\lra y$ in $\bfD$, the map $M(a): M(x) \lra M(y)$ 
is a module map over the ring map $R(a): R(x) \lra R(y)$.

The shape of the diagram for $\cA (G)$ is given by the partially ordered set $\connsubG$ 
of connected subgroups of $G$.  To start with we consider the single
graded ring
$$\cOcF =\prod_{F\in \cF }H^*(BG/F), $$
where the product is over the family $\cF$ of finite subgroups of $G$. 
To specify the value of the ring at a connected subgroup $K$,  
we use Euler classes: indeed if $V$ is a complex representation of $G$ with
$V^G=0$, we may define $c(V) \in \cO_{\cF}$ by specifying its components. 
In the factor corresponding to the finite subgroup $F$ we take
$c(V)(F):=c_{|V^F|}(V^F) \in H^{|V^F|}(BG/F)$ where $c_{|V^F|}(V^F)$ is the  classical Euler class
of $V^F$ in ordinary rational cohomology.

The diagram of rings $\cOtcF$ is defined by the following functor on $\connsubG$ 
$$\cOtcF (K)=\cEi_K \cOcF$$
where $\cE_K =\{ c(V) \st V^K=0\} \subseteq \cOcF$ is the multiplicative
set of Euler classes of $K$-essential representations. 
This localization is again a graded ring. 

Next we consider the category of modules $M$ over the diagram $\cOtcF$. 
Thus the value $M(K)$ is a module over $\cEi_K\cOcF$, and if
$L\subseteq K$, the structure map 
$$M(L)\lra M(K)$$
is a map of modules over the map 
$$\cEi_L \cOcF \lra \cEi_K \cOcF$$
of rings.   Note this map of rings is a localization since for any
complex representation $V$ of $G$, 
$V^L=0$ implies $V^K=0$ 
so that $\cE_{L}\subseteq \cE_{K}$. 
The category $\cA (G)$ is formed from a subcategory of the
category of $\cOtcF$-modules by adding structure. There are two requirements which
we briefly indicate here.    Firstly they must 
be {\em quasi-coherent}, in that they are determined by their 
value at the trivial subgroup $1$ by the formula 
$$M(K):=\cEi_K M(1). $$

The second condition involves the relation between $G$ and its quotients. 
Choosing a particular connected subgroup $K$, we consider the
relationship between the group
$G$ with the collection $\cF$ of its finite subgroups  
and the quotient group $G/K$ 
with the collection $\cF /K$ of its finite subgroups.  
For $G$ we have the ring $\cOcF$ and for $G/K$ we have 
the ring
$$\cOcFK =\prod_{\tK \in \cF /K}H^*(BG/\tK)$$
where we have identified finite subgroups of $G/K$ with 
their inverse images in $G$, i.e., with subgroups $\tK$ of $G$
having identity component $K$. Combining the inflation maps associated
to passing to quotients by $K$ for individual groups, there is an inflation map 
$$\cOcFK \lra \cOcF. $$

The second condition is  that the object should be {\em extended}, in 
the sense that for each connected subgroup $K$ there is a specified isomorphism 
$$M(K) \iso \cEi_K \cOcF \otimes_{\cOcFK} \phi^K M$$
for some $\cOcFK$-module $\phi^KM$, which is  a given part of the structure. 
These identifications should be compatible  when we have inclusions of connected
subgroups.  If we choose a subgroup $L$ then the modules $\phi^KM$ for
$K\supseteq L$ fit together to make an object of $\cA (G/L)$.

\subsection{Diagrams of quotient pairs.}
For some purposes it is useful to have an alternative view of $\cA
(G)$ as introduced in \cite{tnq2} making more of the structure
explicit. Here the values $\phi^HM$ are all
displayed in a single diagram indexed by pairs of quotient
groups. Pairs of quotient groups are equivalent to pairs of subgroups,
but here we will stick with the indexing by quotients $G/K$ as in
\cite{tnq2} since it is the quotients that enter most directly into
the model.  We use the notations $\RRc^p$ for the ring and
$\cA_c^p(G)$ for the category as in \cite{AGs}, since this is
descriptive of the fact that we use {\bf p}airs of {\bf c}onnected
subgroups.  

\begin{defn}
The diagram of {\em quotient pairs} of $G$ is the partially ordered
set with objects $(G/K)_{G/L}$ for $L\subseteq K \subseteq G$, and with two 
types of morphisms. The {\em horizontal} morphisms
$$h_K^H: (G/K)_{G/L} \lra (G/H)_{G/L} \mbox{ for } L\subseteq K \subseteq H \subseteq G$$
and the {\em vertical} morphisms
$$v_L^K: (G/H)_{G/K} \lra (G/H)_{G/L} \mbox{ for } 
L\subseteq K \subseteq H \subseteq G.$$

\end{defn}

One particular diagram will be of special significance for us. 
\begin{defn}
The structure diagram for $G$ is the diagram of rings $\RRc^p$ defined by 
$$\RRc^p (G/K)_{G/L}:=\cEi_{K/L} \cOcFL . $$
Since $V^K=0$ implies $V^H=0$,  we see that 
$\cE_{H/L}\supseteq \cE_{K/L}$,  so it is legitimate to take the 
horizontal maps to be localizations
$$h_K^H: \cEi_{K/L} \cOcFL \lra \cEi_{H/L} \cOcFL . $$
To define the  vertical maps,  we begin with the inflation map 
$\infl_{G/K}^{G/L}: \cOcFK \lra \cOcFL$, and then observe that 
if $V$ is a representation of $G/K$ with $V^H=0$, it may be regarded
as a representation of $G/L$, and Euler classes correspond in the sense
that  $\infl(e_{G/K}(V))=e_{G/L}(V)$. We therefore obtain a map  
$$v_K^L: \cEi_{H/K}\cOcFK \lra \cEi_{H/L}\cOcFL. $$
\end{defn}

Illustrating this for a group $G$ of rank 2, we obtain
$$\diagram
           &                    &\cOcFG    \dto\\
           &\cOcFK \rto \dto &\cEi_{G/K}\cOcFK\dto\\
\cOcF \rto &\cEi_K \cOcF \rto      &\cEi_G    \cOcF\\
\enddiagram
$$
At the top right, of course $\cOcFG=\Q$, but clarifies the formalism
to use the more complicated notation.

In discussing modules, we need to refer to the structure maps for rings, 
so for an $\RRc^p$-module $M$, if $L\subseteq K \subseteq H\subseteq G$, 
we generically write 
$$\alpha_K^L: M(G/H)_{G/K}\lra M( G/H)_{G/L}$$ 
for the vertical map, and 
$$\alphat_K^L: \cEi_{H/L}\cOcFL \tensor_{\cOcFK}M(G/H)_{G/K}
=(v_K^L)_*M(G/H)_{G/K} \lra M(G/H)_{G/L}$$ 
for the associated map of $\cOcFL$-modules. Similarly,
we generically write 
$$\beta_K^H: M(G/K)_{G/L}\lra M(G/H)_{G/L}$$ 
for the horizontal map, and 
$$\betat_K^H: \cEi_{H/L}M(G/K)_{G/L}
=(h_K^H)_*M(G/K)_{G/L} \lra M( G/H)_{G/L}$$ 
for the associated map of $\cEi_{H/L}\cOcFL$-modules, which 
we refer to as the {\em basing map} after \cite{s1q}.

\begin{defn}
If $M$ is an $\RRc^p$-module, we say that $M$ is {\em extended} if
whenever $L\subseteq K \subseteq H$ the vertical map  $\alpha_K^L$
is an extension of scalars along $v_K^L:\cEi_{H/K}\cOcFK \lra 
\cEi_{H/L}\cOcFL,  $ which is to say that 
$$\alphat_K^L: \cEi_{H/L}\cOcFL \tensor_{\cOcFK} M(G/H)_{G/K}
\stackrel{\cong}\lra M(G/H)_{G/L}$$ 
is an isomorphism of $\cEi_{H/L}\cOcFL $-modules. 

If $M$ is an $\RRc^p$-module, we say that $M$ is 
{\em quasi-coherent} if
whenever $L\subseteq K \subseteq H$ the horizontal map 
$\beta_K^H$ is an extension of scalars along $h_K^H:\cEi_{K/L}\cOcFL
\lra \cEi_{H/L}\cOcFL$, which is to say that 
$$\betat_K^H: \cEi_{H/L}M(G/K)_{G/L}\stackrel{\cong}\lra M(G/H)_{G/L}$$ 
is an isomorphism. 

We write $\mbox{$qc$-$\RRc^p$-mod}, \mbox{$e$-$\RRc^p$-mod}$ and 
$\cA_c^p(G):=\mbox{$qce$-$\RRc^p$-mod}$ for the full subcategories
of  $\RRc$-modules with the indicated properties. 
\end{defn}

Next observe that the most significant part of the information in an 
extended object is displayed in its restriction to the leading diagonal. 
For example in our rank 2 example they take the form
$$\diagram
           &                    & M(G/G)_{G/G}    \dto\\
           &M(G/K)_{G/K} \rto \dto &\cEi_{G/K}\cOcFK\tensor_{\cOcFG} M(G/G)_{G/G} \dto\\
M(G/1)_{G/1} \rto &\cEi_K \cOcF \tensor_{\cOcFK}M(G/K)_{G/K} \rto      
&\cEi_G    \cOcF\tensor_{\cOcFG}M(G/G)_{G/G}\\
\enddiagram
$$

In effect our description of the category $\cA (G)$  
abbreviates such a diagram by just writing the final 
row and taking $\phi^K M=M(G/K)_{G/K}$:
$$\diagram
\phi^1M \rto &\cEi_K \cOcF \tensor_{\cOcFK}\phi^KM \rto      
&\cEi_G    \cOcF\tensor_{\cOcFG}\phi^GM, \\
\enddiagram
$$
leaving it implicit that the particular decomposition as a tensor product
is part of the structure.

\begin{lemma} \cite[5.5]{tnq2}
The functor 
$$i: \cA (G) \lra \cA_c^p(G) =\mbox{$qce$-$\RRc^p$-$\mathrm{mod}$}$$ 
defined by 
$$i (M) (G/K)_{G/L}:=\cEi_{K/L}\phi^LM . $$
is an equivalence 
$$\cA (G) \simeq \cA_c^p (G).$$\qqed
\end{lemma}

Henceforth we will identify the two, thinking of $\cA (G)$ as given by the values of
$\cA_c^p(G)$ on the objects $(G/K)_{G/K}$ with additional structure
given by the horizontal and vertical maps. 

\subsection{Connection with topology}
The homotopy level connection between $G$-spectra and $\cA (G)$
is given by  a homotopy functor 
$$\piA_*: \Gspectra \lra \cA (G)$$
with the exactness properties of a homology theory. 
It is rather easy to write down the value of the functor as
a diagram of abelian groups.

\begin{defn}
\label{defn:piA}
For a $G$-spectrum $X$ we define $\piA_*(X)$ on $K$  by 
$$\piA_*(X)(K)=\pi^G_*(\DH \efp \sm \siftyV{K} \sm X).$$
Here $\efp$ is the universal space for the family $\cF$ of finite 
subgroups with a disjoint basepoint added and $\DH \efp =
F(\efp , S^0)$ is its functional dual
(the function $G$-spectrum of maps from $\efp$ to $S^0$).
For any closed subgroup $K$ of $G$, the $G$-space $\siftyV{K}$ is defined by  
$$\siftyV{K} =\bigcup_{V^K=0} S^V, $$
where $V$ runs through finite dimensional subrepresentations of a
complete $G$-universe, $\cU$.  When $K \subseteq H$, we find
$V^H\subseteq V^K$ so  there is a map 
$\siftyV{K} \lra \siftyV{H}$, and this  induces the map 
$\piA_*(X)(K)\lra \piA_*(X)(H)$.\qqed
\end{defn}

The  definition of $\piA_*(X)$ shows that quasi-coherence for 
$\piA_*(X)$ is just a matter of understanding Euler classes. 
The extendedness of  $\piA_*(X)$ is a little more subtle, and will 
play a significant role later.  Extendedness   follows from properties
of the geometric fixed point functor. We may take
$$\phi^K \piA_*(X)= \pi^{G/K}_*(\DH \efkp \sm \Phi^K(X)),$$ 
where $\PK$ is the geometric fixed point functor using the 
 map $\infl (D\efkp) \lra D\efp \sm \siftyV{K}$ (see \cite[9.2]{tnq1}
for details).

To see that $\piA_*(X)$ is a module over $\cO$, the key is to understand $S^0$.
 
\begin{thm}
\cite[1.5]{tnq1}
The image of $S^0$ in $\cA (G)$ is the structure functor:
$$\widetilde{\cO}_{\cF} =\piA_*(S^0), $$
with the canonical structure as an extended module. 
\end{thm}

Some additional work confirms that $\piA_*$ has the appropriate behaviour.  
\begin{cor}
\cite[1.6]{tnq1}
The functor $\piA_*$ takes values in the abelian category $\cA (G)$.
\end{cor}

\subsection{The Adams spectral sequence}
The homology theory $\piA_*$ may be used as the basis of an 
Adams spectral sequence for calculating maps between rational
$G$-spectra. The main theorem of \cite{tnq1} is as follows.

\begin{thm} (\cite[9.1]{tnq1})
For any rational $G$-spectra $X$ and $Y$ there is a natural
Adams spectral sequence
$$\Ext_{\cA (G)}^{*,*}(\piA_*(X) , \piA_*(Y))\Rightarrow [X,Y]^{G}_*.$$
It is a finite spectral sequence concentrated in rows $0$ to $r$ (the rank of $G$)
and strongly convergent for all $X$ and $Y$. \qqed
\end{thm}

This was what led us to attempt to prove the main theorem of the
present paper, and many of
the methods used to construct the Adams spectral sequence are adapted 
to the present work. Nonetheless, it appears that the only way we explicitly use the
Adams spectral sequence  is in the fact that cells are characterized by their
homology. 

\begin{cor}\label{cor.cell.homology} \cite[12.1]{tnq1}
If $X$ is a $G$-spectrum with $\piA_*(X)\cong \piA_*(G/H_+)$ then 
$X\simeq G/H_+$.
\end{cor}

The proof proceeds by giving an explicit resolution of $\piA_*(G/H_+)$
in  $\cA(G)$, and then observing that this gives appropriate vanishing
at the $E_2$-page so as to ensure an isomorphism 
$\piA_*(X)\cong \piA_*(G/H_+)$ lifts to a homotopy class of maps
$G/H_+ \lra X$. Since $\piA_*$
detects weak equivalences, this suffices. Evidently, this argument
applies in any model category with a similar Adams spectral sequence. 

In the present paper, we often need to know how our chosen cells
behave under functors between model categories. We will apply the 
corollary repeatedly to see that each cell maps to the obvious object up to 
equivalence.

\section{Cochain ring spectra}
\label{sec:modelsofspectra}

The purpose of this section is to discuss the particular model of
$G$-spectra that we use. Much of our argument takes place at the level
of  homotopy categories, but we need several formal properties that
depend on properties of functors at the model category level. In
various ways these are all associated to monoidal structure.

The required properties of the model category of $G$-spectra itself
are standard properties of monoidal model categories, enjoyed by all models
that we might consider. Next, we need certain properties of the change
of groups functors relating the properties of $G$-spectra to those of 
$Q$-spectra for subquotients $Q$ of $G$. One could imagine formalizing
the required properties of all these model categories and functors, 
but we will just use models for which fixed point and inflation
functors have  good behaviour with respect to
the smash product. 

Finally, we might ask for properties of the monoidal structures themselves. For
homotopy level arguments we just need to know standard properties
(bilinearity, compatibility with suspension and space level
constructions etc).  The delicacy arises from the need to consider arguments
modelled on those of commutative ring theory, in precisely the sense
that the ring spectra  are $E_{\infty}$-rings (i.e., algebras over a non-equivariant  $E_{\infty}$-operad, which is to say $E_{\infty}'$   in the sense of McClure
\cite{McClure}).  It is convenient for our arguments 
that the  commutative rings are the commutative monoids for the smash
product.  However, the commutative monoids for the smash product on 
orthogonal $G$-spectra have additional structure (such as norms) that
are incompatible with the homotopy type of some commutative
rings. Accordingly,  it is convenient to use a different smash product, but one
which has the same underlying homotopy type. 

In Subsections \ref{subsec:sphere} and \ref{subsec:coeffs} we
introduce some of our basic decisions about coefficients. 
In Subsection \ref{subsec:commrings} we discuss some of the ring
spectra that we need in general terms. Finally, in 
Subsection \ref{subsec:cLspectra} we describe the Blumberg-Hill model of
spectra that we use, and explain why it has the properties we require.

\subsection{The sphere spectrum }
\label{subsec:sphere}
Just as abelian groups are $\Z$-modules, giving $\Z$ a special role, 
so too spectra are modules over the sphere spectrum $\bS$. Although $\bS$
is the suspension spectrum of $S^0$, we will generally use the special 
notation  $\bS$ to emphasize its special role.  Since we are working rationally,  
$\bS$ will denote the rational sphere spectrum, a commutative ring
constructed as a localization of the unit object.

\subsection{Choice of coefficients}
\label{subsec:coeffs}
Central to our formalism is that we consider `rings of functions' on
certain spaces, and then consider modules over these. In effect we
take a suitable model for cochains on the space with coefficients in a
ring. The purpose of the present subsection is to describe the
options at the level of the homotopy category, and explain why we end up simply using the functional dual
$DX=F(X,\bS)$ rather than one of the natural alternatives. 

If $X$ is a $G$-space and $k$ is a ring $G$-spectrum then we may write
$$C^*(X;k):=D_kX_+:=F_{\bS}(X_+,k)$$
for the  $G$-spectrum of functions from $X$ to $k$.
The first notation comes from the special case of an Eilenberg-Mac~Lane
spectrum, which gives a model for cohomology. The second notation comes
from the special case $k=\bS$ of the functional dual.
 This spectrum has a ring structure using the multiplication on $k$
 and the diagonal map
of $X$. If $k$ is a commutative ring spectrum then so is $C^*(X;k)$.

There are a number of related ring spectra of this form associated to
different choices of $k$ and 
we briefly discuss their properties before explaining which 
is most relevant to us.

First, we could take $k$ to be the rational sphere $G$-spectrum $\bS$,
alternatively, we could take it to be one of two Eilenberg-Mac~Lane $G$-spectra associated to Green 
functors. The first Green functor is the Burnside functor $\A$, whose
value on $G/H$ is the Burnside ring of $H$, and the second
Green functor is the constant functor $\Q$.

To start with we observe that there are maps
$$\bS \lra H\A \lra H\Q$$
of ring $G$-spectra
where the first map kills higher homotopy groups and the second 
kills the augmentation ideal.  It is elementary to construct these
by killing homotopy groups in the category of $E_{\infty}$-rings. 
In fact $\A$ and $\Q$ are Tambara functors, and one expects such maps can be constructed by the
process of killing homotopy groups conducted in the category of
commutative ring $G$-spectra\footnote{A full justification would
  involve generalizing the arguments of \cite{Ullman} to the compact 
Lie case. Essentially we need to know that there are free  
$E_{\infty}^G$-ring spectra (i.e., free as equivariant $E_{\infty}$-rings) and that ordinary cohomology with
coefficients in a  Tambara functor is represented by a commutative 
ring spectrum. However in this section we are only explaining why we
discard some options,  so it suffices to consider the underlying maps of spectra.}.  Any $G$-space $X$ has a diagonal and a
map to the terminal object, making it a cocommutative coring, adding a
disjoint basepoint and mapping into our sequence of maps of
commutative ring spectra we obtain the sequence
$$D_{\bS}X_+ \lra D_{H\A}X_+\lra D_{H\Q}X_+$$
of ring spectra. These are very far from being equivalences in general. For the
second map that is clear since $\A (G/H)\neq \Q$ if $H$ is
a non-trivial finite subgroup. For the first, it is clear from 
the fact that $\bS$ has non-trivial higher homotopy (even
rationally) when $G$ is not finite. 

\begin{lemma}
(i) If $X$ is free, the above maps are equivalences
$$D_{\bS}X_+ \simeq D_{H\A}X_+\simeq  D_{H\Q}X_+$$
of $G$-spectra. 

(ii) If  $X$ has only finite isotropy,
then the first map is an equivalence
$$D_{\bS}X_+ \simeq D_{H\A}X_+ $$
of $G$-spectra. 
\end{lemma}

\begin{proof}
For Part (i) we note that $\bS, H\A$ and $H\Q$ all have
non-equivariant homotopy $\Q$ in degree 0. 

For Part (ii), $\bS$ is (rationally) an Eilenberg-Mac~Lane
spectrum for any finite group of equivariance by tom Dieck splitting
(see \cite[Appendix A]{Tate}). 
\end{proof}

The functor $D_{H\Q}$ has the convenient property that there is an equivalence
$$(D_{H\Q}Y)^G\simeq D_{H\Q}(Y/G)$$
for any based $G$-space $Y$. On the other hand, this lets us calculate values 
which show the functor is not the one we want to use
(specifically, the homotopical analysis of \cite{tnq1} makes
clear that the homotopy groups of the cochains on $\efp$ should
be those of $D_{\bS} \efp$).  Henceforth we simply write
$$D(\cdot )=D_{\bS}(\cdot )$$
for the rational functional dual.

\subsection{Some commutative ring spectra}
\label{subsec:commrings}

Our arguments use ideas from commutative algebra, so we want to 
 work in a context where certain $G$-spectra $R$ behave like
 commutative rings.  What we need is a symmetric monoidal 
category of $R$-modules with a well behaved homotopy category,
which behaves well under various change of groups constructions.  
It is conceptually simplest if we work in a category of $G$-spectra
where the relevant rings $R$ actually are commutative monoids, and
we will describe such a context in  Subsection \ref{subsec:cLspectra}.

For now we identify ring structures by operad actions. We are
essentially recording the observations of McClure \cite{McClure}, but
updating terminology. We say  that a $G$-spectrum $X$
is an $E_{\infty}$-ring if it has an action of a non-equivariant $E_{\infty}$-operad
(viewed as a $G$-fixed $G$-space), such as the linear isometries
operad on a $G$-fixed universe. This is the least restrictive type of
 $N_{\infty}$-operad (in the sense of Blumberg-Hill
 \cite{BlumbergHill}), the one that  is as free as possible (so that the $n$-th term is universal for the
family of all subgroups of $G\times \Sigma_n$ of the form $H\times 1$). 
A $G$-spectrum is  an $E_{\infty}^G$-ring if it
has an action of a {\em $G$-equivariant} $E_{\infty}$-operad (such as the
linear isometries operad on a complete $G$-universe).  
This is the most restrictive type of $N_{\infty}$-operad, 
 with isotropy as large as possible (so that the $n$-th term is universal for the
family of all subgroups of $G\times \Sigma_n$ intersecting $\Sigma_n$ in the trivial group). 
The ring spectra we need are rather obviously $E_{\infty}$-ring
spectra,  whereas in some cases it requires extra work to show that they
are  $E_{\infty}^G$-ring spectra. 
McClure observes that $E_{\infty}^G$-rings have more structure that $E_{\infty}$-rings; this structure is 
used in \cite{GMMU} to define multiplicative norm maps, and the
relationship between the $E_{\infty}^G$ structure and the norm maps is
studied systematically by Hill and Hopkins \cite{HillHopkins}.  In
particular this shows that for a finite group $F$, an $E_{\infty}^G$-ring spectrum which is
non-equivariantly contractible must be $F$-equivariantly contractible
(since the norm of the unit is the unit).

Our examples start with the function spectrum $D\efp$. This is an
$E_{\infty}^G$-ring by \cite[Lemma 4 (a)]{McClure}
since it consists of maps from a $G$-space of the form $X_+$ (which has a strictly
cocommutative diagonal) into the $E_{\infty}^G$-ring $\bbS$.   We then
wish to consider  the spectra 
$\sifty{H} \sm D\efp$ for connected subgroups $H$, where
$$\sifty{H}=\bigcup_{V^H=0}S^V.$$
The homotopy type $\siftyV{H}\sm D\efp$ can be obtained as a smash product as written, or as the
Bousfield localization of $D\efp$ with respect to $\sifty{H}$. The
importance of $\sifty{H}$ is firstly that it has geometric
isotropy consisting of precisely the subgroups containing $H$,  and secondly that because of the way  it is
built from spheres it gives a close connection to algebraic
localizations. 

Furthermore $\sifty{H}$ also has excellent multiplicative properties. 
It is clear that it is a commutative ring up to homotopy, and by the
argument of \cite[Lemma 3]{McClure} it is a based $E_{\infty}$-space. 
Alternatively one may apply  \cite{HillHopkins} to see that Bousfield localization preserves the
existence of an action by a non-equivariant $E_{\infty}$-operad.
Because $G$ acts trivially on the operad, an $E_{\infty}$-ring
$G$-spectrum has the property that its categorical $H$-fixed point
spectrum is a $G/H$-spectrum which is also an algebra over an
$E_{\infty}$-operad. To avoid having to discuss flatness, if $R$ is
an $E_{\infty}$-ring, we will always construct $\sifty{H}\sm R$ as an
$E_{\infty}$-ring by Bousfield localization. The derived smash
product gives the homotopy type, and we use the  notation 
$\sifty{H}\smb R$ for the localization to remind us of this. 


The other construction we will need corresponds to taking countable 
products of commutative rings. It is clear that if objects $A_i$ admit 
an action of an operad $\cO$, then so does the product $\prod_i A_i$. 
This only  uses categorical properties of products. If this is to be 
homotopically meaningful we need to assume 
as usual that the objects $A_i$ are fibrant. We will apply this when 
$\cO$ is an  $E_{\infty}$-operad. 

\begin{remark}
\label{rem:siftyVEinftyG}
The  geometric isotropy of the spectrum $\sifty{H}$ consists of the subgroups containing $H$. 
If $H$ is not connected, the norm from the identity component $H_e$ to $H$ shows
that the  spectrum  $\sifty{H}$ does not  admit the structure of and
$E_{\infty}^G$-ring. However, if $H$ is {\em connected}, any inclusion $L\subseteq K$ from outside the geometric
isotropy (i.e., $L\not \supseteq H$) to inside the geometric isotropy
(i.e., $L\subseteq H$) is of infinite index, so the norm obstructions vanish. In fact \cite{McClureplus} one may generalize McClure'sargument  to show that in this case $\sifty{H}$ is an
$E_{\infty}^G$-ring. This means that in fact all the rings we need to discuss are
$E_{\infty}^G$-rings. 
\end{remark}

\subsection{The category of orthogonal $\mcL$-spectra}
\label{subsec:cLspectra}

We wish to work in a monoidal category of $G$-spectra in which the
rings we work with are commutative monoids. In this section we describe
our chosen category. One option is to work with  orthogonal $G$-spectra:  in view
of Remark \ref{rem:siftyVEinftyG} the rings we work with are
$E_{\infty}^G$-rings, and  these are precisely the the commutative
monoids in orthogonal  $G$-spectra. As shown in Remark
\ref{rem:otherfoundations}, this would provide  
foundations for our work.  However, since \cite{McClureplus}
is not yet in final form,  we have followed an alternative route. 

We retreat to $E_{\infty}$-rings and  use  a category of
spectra in which the commutative monoids are the
$E_{\infty}$-rings. It is natural to use the `operadic smash product'
approach of Elmendorf-Kriz-Mandell-May \cite{ekmm} applied to
orthogonal $G$-spectra. Such a category has been set up  by  Blumberg
and Hill  in 
\cite{BlumbergHill2}; their main concern is to understand the homotopy 
theory of different types of norm and different degrees of commutativity. Since we are only concerned with
the simplest type of commutativity and not with norms at all, we only
need the more formal parts of their argument,  which apply to all
compact Lie groups \cite[Appendix B]{BlumbergHill2}.

We are very grateful to Blumberg and Hill for discussions about their category,
and for  explicitly including statements from which the properties we
require are apparent.  We would also like to thank Blumberg for
suggestions which led to the current approaches to Properties (11) and (12) in Proposition \ref{prop:axiom}
below.

 The construction starts with the category $\SpOG$ of orthogonal $G$-spectra
based (additively) on a complete orthogonal $G$-universe $\cU$ as usual.  For the {\em multiplicative} properties, we now choose 
a  $G$-fixed universe $\cV$ (i.e., infinite dimensional but with
trivial $G$-action) with a view to constructing an operadic
smash product based on the $\cV$-linear isometries operad. 

More precisely, we  let  $\mcL$ denote the non-equivariant linear isometries operad
defined by 
$$\mcL (n) = \mathrm{Isom}(\cV^n, \cV).$$
There is an associated monad $\bL$  given by smashing with $\mcL
(1)_+$  and we  consider the category $\SpOG [\bL ]$ of $\bL$-algebras in orthogonal
$G$-spectra.  Applying $\bL$ is left adjoint to the forgetful functor
relating orthogonal $G$-spectra to those with an $\bL$ action. Since
$\mcL (1)$ is contractible, the functors relate objects with the same
underlying homotopy type. Precisely as in \cite{ekmm}, the category of $\bL$-spectra    has a symmetric monoidal smash product
$\smL$ and we restrict to those which are unital in the sense that the
unit map $S\sm_{\mcL} X\lra X$ is an isomorphism.  This category $\SpLOG$ of unital
$\bL$-orthogonal  $G$-spectra (denoted $G{\cS}_U$ in
\cite{BlumbergHill2})  is analogous to the category referred
to as $S$-modules in \cite{ekmm}, and  is a monoidal  model category satisfying the monoid
axiom.

\begin{prop}
\label{prop:axiom}
For a compact Lie group $G$, the category of orthogonal $G$-equivariant $\mcL$-spectra, $\SpLOG$,
has the following properties
\begin{enumerate}
\item It is a weakly symmetric,  monoidal,  proper,  $G$-topological model
  category satisfying the monoid axiom with weak equivalences detected by the forgetful functor 
$$U: \SpLOG \lra \SpOG$$
to orthogonal $G$-spectra.  
\item The functor $U$ is a right Quillen functor inducing an
  equivalence of homotopy categories, and preserves {and detects} all weak
  equivalences. 
\item The functor $U$ is lax monoidal, so the smash product is 
compatible with this equivalence of homotopy  categories, and  in the non-equivariant setting is monoidally equivalent to the usual
 smash product of orthogonal spectra.  
 \item The monoids in $\SpLOG$ are non-$\Sigma$ algebras over the
  linear isometries operad $\mcL$ 
\item The commutative monoids in $\SpLOG$ are algebras over the 
 linear isometries operad $\mcL$, and hence $E_{\infty}$-ring
 spectra. 
\item The Bousfield rationalization $\bbS$ of the sphere spectrum (see
  Property (8))
is a commutative monoid in $\SpLOG$.
\item The category of commutative monoids in $\SpLOG$ is cotensored over unbased spaces.
\item All left Bousfield localizations preserve commutative monoids,
  so that if $A$ is a commutative monoid, for any $E$ the map $A\lra
  L_EA$ is a map of commutative monoids. 
\end{enumerate}
The equivariant categories for $G$ and its quotients are related as
follows. 
\begin{enumerate}
\setcounter{enumi}{8}
\item For any closed normal subgroup $K$,  inflation 
$$\infl_{G/K}^{G}: \SpLOGK \lra \SpLOG$$
from $G/K$-spectra to $G$-spectra is strong symmetric monoidal and
therefore takes commutative monoids to commutative monoids. Inflation is a
left Quillen functor. 
\item The $K$-fixed point functor $(\cdot )^K: \SpLOG \lra \SpLOGK$
 is lax symmetric monoidal and hence preserves commutative
 monoids. The $K$-fixed point functor is a right Quillen functor. 
\item There is a zig-zag of Quillen equivalences between commutative
  monoids in $1$-spectra and commutative monoids in symmetric spectra. Let 
$$\FF : \Ho (\mbox{comm-mon-$\SpLOone$})\lra  \Ho (\mbox{comm-mon-$\SpSigma$})$$
denote  the derived functor.
\item For a commutative monoid $A$ in $1$-spectra, there is a Quillen
  equivalence 
$$\mbox{$A$-mod-$\SpLOone$}\simeq_Q \mbox{$\FF A$-mod-$\SpSigma$}$$
between the categories of $A$-modules over $1$-spectra and $\FF A$-modules over symmetric spectra.
\end{enumerate}
\end{prop}

\begin{proof}
(1)  is \cite[4.2, 4.6]{BlumbergHill2}. 

(2) is \cite[4.3, 4.8]{BlumbergHill2}. 

(3) is \cite[4.3, 4.9]{BlumbergHill2}. 

(4) and (5) are  \cite[3.18]{BlumbergHill2}. 

(6) is a special case of (8) given 
 the fact that the sphere spectrum is a commutative monoid. 

(7) is \cite[3.19]{BlumbergHill2}. 

(8) is \cite[6.4]{HillHopkinsGsymm}.  {This shows that the non-equivariant proof in~\cite[VIII.2.2]{ekmm} extends to the equivariant case.}

(9)  and (10) are  \cite[3.24, 4.18]{BlumbergHill2}. 

(11): 
First, \cite[3.18]{BlumbergHill2}, referring to ~\cite[II.4.6]{ekmm}, 
shows that commutative monoids in non-equivariant $\mcL$-spectra are isomorphic to the category of $E_{\infty}$-algebras in orthogonal spectra.  Then~\cite[0.14]{mmss} shows that $E_{\infty}$-algebras in orthogonal spectra and in symmetric spectra are Quillen equivalent and both are Quillen equivalent to the respective categories of commutative monoids.

 (12): 
The category of modules over a commutative monoid $A$ in
non-equivariant $\mcL$-spectra is isomorphic to the category of operadic modules over the associated $E_{\infty}$ orthogonal spectrum $\bbU A$ from~\cite[3.18]{BlumbergHill2}. As above, this follows by an analogue of the argument in~\cite[II.5.1]{ekmm}. See also~\cite{ElmendorfMandell} for a careful definition of operadic modules via multicategories.  

Next we use the monoidal Quillen equivalence between orthogonal
spectra and symmetric spectra from~\cite[0.4]{mmss} to show that the
category of operadic modules over an $E_{\infty}$-algebra in
orthogonal spectra is Quillen equivalent to the category of operadic
modules over an associated $E_{\infty}$-algebra in symmetric spectra
by~\cite[2.14]{BergerMoerdijk}. Note that the condition in~\cite[2.14]{BergerMoerdijk} about units is satisfied because the map in question is an isomorphism, and hence a cofibration. Finally,~\cite[1.4]{ElmendorfMandell}
shows that this category of modules over an $E_{\infty}$-algebra is
Quillen equivalent to the category of modules over a commutative monoid in symmetric spectra. The commutative monoid here may differ from the image of $\FF$, but the two are weakly equivalent. The statement then follows by~\cite[5.4.5]{hss}.
\end{proof}

\begin{remark}
\label{rem:otherfoundations}
Properties (1) - (12) embody very natural requirements of equivariant
spectra. We would expect our general strategy to be effective in other models 
when analogous properties hold.

In particular, for the category $\SpOG$ of orthogonal $G$-spectra
itself, properties (1), (2) and (3) are obvious (since $U$ is replaced
by the identity functor), and similarly properties (11) and (12)
become trivial. Properties (7), (9) and (10) are basic properties of
orthogonal spectra \cite{mm}. Property (6) is true because the rational
sphere is inflated from a fixed spectrum. 

The counterparts of Properties (4) and (5) for orthogonal $G$-spectra
replace the linear isometries operad $\mcL$ on the {\em $G$-fixed} universe
$\cV$ by the linear isometries operad $\mcL_G$ on a {\em complete}
$G$-universe $\cV$. This uses the traditional argument of
\cite[15.5]{mmss}, using \cite[B.117]{HHR}, which in turn corrects 
\cite[III.8.4]{mm}. We note that the statement in \cite{HHR} is only
given for finite groups, but their argument applies to arbitrary
compact Lie groups giving the full replacement for the statement in  \cite{mm}.

Property (8) is not true in $\SpOG$ for general $A$: this is the topic of
\cite{HillHopkins}. We need it for $A=\sifty{H}$ with $H$
connected:  \cite{McClureplus} shows  $\sifty{H}$ is a commutative
ring and we may obtain the localization by taking tensor product with $\sifty{H}$. 
\end{remark}

\part{Formality of the sphere spectrum}
\section{Diagrams of rings and modules}
\label{sec:ringsandmodules}
Throughout this paper we consider categories of modules over diagrams of
rings in three contexts: differential graded modules over DGAs, module
spectra over ring spectra, and module $G$-spectra over ring $G$-spectra. 
  In this section we describe the
context and the basic Quillen equivalences arising from a pullback
diagram of rings. These and related results are discussed more fully and proved in
\cite{diagrammodcats}.

\subsection{The archetype}
\label{sec:archetype}
Given a diagram shape $\bfD$, consider a diagram of rings $R:\bfD \lra \C$ in a symmetric monoidal category $\C$.  
Each map $R(a): R(s)\lra R(t)$ gives rise to an extension of scalars functor
$$ \modcat{R(s)}\stackrel{a_{*}}\lra \modcat{R(t)} $$
defined by $a_*(X)=R(t)\tensor_{R(s)}X$, 
 with right adjoint the restriction of scalars functor
$$ \modcat{R(s)}\stackrel{a^*}\lla \modcat{R(t)}. $$

Now consider a category of {\em $R$-modules}; 
the objects are diagrams $X: \bfD \lra \C$ for which $X(s)$ is an
$R(s)$-module for each object $s$, 
and for every morphism $a: s\lra t$ in $\bfD$, the map $X(a): X(s) \lra X(t)$ 
is a {\em module map over the ring map} $R(a): R(s) \lra R(t)$. More
precisely, there is a map $X(s) \lra a^*X(t)$ of $R(s)$-modules (the {\em
 restriction} form)  or,  equivalently, there is a map
$$R(t)\tensor_{R(s)} X(s) =a_*X(s) \lra X(t) $$
of modules over the ring $R(t)$ (the {\em extension of  scalars}
form).  Despite the simplicity of restriction of scalars,   we view
the {\em left} adjoint $a_*$ as the primary one, following the
convention that the left Quillen functor dicates the direction of a
Quillen pair.

\subsection{Model structures}\label{subsec:modstructure}
We say that a pseudo-functor $\mccM : \bfD \lra \mathrm{Cat}$ is {\em
  a diagram of model categories} if each category $\mccM (s)$ has a
model structure, the functors $a_*$ all have right
adjoints and  the adjoint pair $a_* \dashv a^*$ of functors relating the 
model categories form a Quillen pair. 

For instance, the motivating example of a diagram of ring spectra (or DGAs) gives a diagram of model
categories if we use the projective model structure on the category 
$\mccM (s)$ of  $R(s)$-modules; {see, for example,~\cite[4.1]{ss1}.}

When $\mccM$ is a diagram of model categories, there are two ways to attempt to put a model structure on 
the category of $\mccM$-diagrams $\{ X(s)\}_{s \in \bfD}$. The
{\em diagram-projective} model structure (if it exists) has its fibrations and weak  
equivalences defined  objectwise. The {\em diagram-injective} model
structure  (if it exists) has its cofibrations and 
weak  equivalences defined objectwise.  It must be checked in each 
particular case whether or not these specifications determine a model
structure. When both model structures exist, it is clear that the identity functors
define a Quillen equivalence between them.

We will apply \cite[Theorem 3.1]{diagrammodcats} to show that 
the diagram-projective and diagram-injective model structures exist in
the cases of interest to us. 

\subsection{Pullback diagrams of rings}
\label{subsec:pullbackofrings}
The basic input from the diagrams of model categories from
\cite{diagrammodcats} is as follows.  {Recall from~\cite[5.1.1]{hovey-model} that an {\em inverse category}, $\ccI$, requires a linear extension from $\ccI^{\mathrm{op}}$ to an ordinal.  This is dual to the notion of a direct category.}

\begin{prop} \cite[Proposition 4.1]{diagrammodcats}
\label{prop-gen-pb}
For $\ccI$ a finite, {inverse} category with at most one morphism in
each $\ccI(s,t)$ and $R$ a $\ccI$-diagram of ring spectra with
homotopy inverse limit $\bbarmccR$, there is a zig-zag of Quillen
equivalences between the category of $\bbarmccR$-modules and the
cellularization with respect to $\mccR$ of $\mccR$-modules (with the
diagram-injective model structure):
\[ \modcat{\bbarmccR} \simeq_{Q} \mbox{$\mccR$-cell-}\modcat{\mccR} \]
\end{prop}

We will apply this when $\ccI$ is a punctured cube, and
$\bbarmccR=\bbS$ is the sphere spectrum. Indeed the category of $G$-spectra
is equivalent to the category of module-$G$-spectra over the sphere
spectrum $\bbS$. By Proposition \ref{prop-gen-pb}, this 
is in turn equivalent to the cellularization of a category of modules
over a diagram of ring $G$-spectra. The rest of the work will be based
on diagrams of this punctured cube shape. The argument   proceeds by replacing the diagram of
ring $G$-spectra successively by diagrams of nonequivariant ring
spectra, DGAs and finally graded commutative rings. 

Our next task is the homotopical core of the paper.  We show that the
sphere is the pullback of a diagram of spectra which are both
isotropically  simpler and very rigid.

\section{The sphere as an isotropic pullback}
\label{sec:isotropiccube}

Our analysis is based on expressing the sphere spectrum as the
homotopy pullback of an $(r+1)$-cube of ring $G$-spectra. When $G$ is
the circle group this is the $\cF$-Tate square \cite{Tate}
$$\diagram 
S^0\rto \dto  & \etf \dto \\
D\efp \rto   & \etf \sm D\efp,  
\enddiagram$$
but usually it is more complicated. In fact  we will construct a diagram $\Rttop : C \lra
\mbox{Ring-$G$-spectra}$ where the cube $C$ is the poset  of subsets of  $\{
0, \ldots , r\}$, where $\bbS$ is the value on the initial vertex (the
empty set) and so that this is equivalent to the homotopy pullback of
the restriction of $\Rttop$ to the punctured cube $PC$ of non-empty
subsets. 

\subsection{Strategy}

In the course of the proof, we will need to consider an extension of  $\Rttop$ to a bigger
diagram, and we introduce this extended diagram as we go along. The cube $C$ above will
appear as $C=\Cf$ in due course.  The letter $f$ stands for `formal' though the word `affine'  or the word `rigid'
would be sensible alternatives. Corollary \ref{cor:SisPCfpb} will show
that the sphere spectrum is the homotopy pullback of $\Rttop$
restricted to the punctured cube $\PCf$, so that the results of \cite{diagrammodcats} (as quoted in
Proposition \ref{prop-gen-pb})  show that 
the category of modules over the sphere spectrum is equivalent to the 
cellularization of the category of  modules over the  $PC_f$-diagram
of ring $G$-spectra. The reason
 this is useful is that the  ring $G$-spectra  $A$ at the vertices of the punctured cube $\PCf$
have two very special rigidity properties. Firstly, as in
\cite{modulefps} passage to $G$-fixed points gives
an equivalence between categories of $A$-module-$G$-spectra and
categories of  $A^G$-module-spectra. This means we can reduce from
considering $\Rttop$-modules in $G$-spectra to considering  modules over
 the $\PCf$-diagram $\Rtop =(\Rttop)^G$ of non-equivariant ring spectra. We can then use the second
 author's results to move to considering modules over a $\PCf$-diagram
 $\Rt$ of DGAs. The second feature of the spectra $A$ is that
 $\pi^G_*(A)=\pi_*(A^G)$ is intrinsically formal in that any commutative DGA with
 this homology is equivalent to $\pi^G_*(A)$ with zero
 differential. As shown in Section \ref{sec:formality}, the proof of this is compatible with the $\PCf$-diagram, so we
 are reduced to considering  DG-modules over a $\PCf$-diagram $\Ra$ of
 graded rings. We may then show the cellularized category of
 $\Ra$-modules is equivalent to the category $\cA (G)$ of \cite{tnq1}.

Our first task (Sections \ref{sec:isotropiccube} and
\ref{sec:formalcube})  is  to describe the $(r+1)$-cube $\Cf$ of ring
spectra with the sphere spectrum at the initial vertex and to show  it
is a homotopy pullback. We will do this in steps: we 
identify $\Cf$ inside a larger diagram $\Cif$ containing a second cube
$\Ci$, giving inclusions of diagrams
$$\Ci \subseteq \Cif \supseteq \Cf, $$
and we will prove equivalences  
$$\diagram
S^0\ar[r]^-1_-\simeq&\holim_{v\in \PCi} \Rttop (v) &\holim_{v\in \PCif}
\Rttop (v) \ar[l]_-2^-\simeq\ar[r]^-3_-\simeq&
\holim_{v\in \PCf} \Rttop (v) .
\enddiagram $$
 Of these, Equivalence 2 is elementary, since  $\PCi$ is cofinal in
 $\PCif$.  Equivalence 1 (Proposition \ref{prop:CiPB}) is the essential one, since in fact
 Equivalence 3 (Proposition \ref{prop:CfPB})  is
 essentially given by using  Equivalence 1 repeatedly for quotient groups of lower rank.

 For this reason we will  begin with the  cube $C_i$ of ring spectra
 constructed purely on isotropical principles,  and Equivalence 1. 
Since the  ring $G$-spectra at the vertices
of the  punctured cube $\PCi$ do not have the rigidity properties we
need, we will then  take the further step of reducing to the diagram
on the punctured cube $\PCf$.

For the rest of Sections \ref{sec:isotropiccube} and \ref{sec:formalcube} 
we  simpify notation and write $\Rti =\Rttop$.

\subsection{The isotropic cube}
We consider the coordinates $(a_0, a_1, \ldots , a_r)$
where each coordinate $a_i$ can take the value 0 or 1. 
For $0\leq c\leq r-1$ the $c$th
coordinate refers to connected subgroups of codimension $c$. 
The $r$th coordinate also refers to codimension $r$ (i.e., to finite
subgroups), but these must be treated differently, and in effect it 
refers to whether or not the ring is complete (roughly
speaking, whether it is $S^0$ or $D\efp$). Throughout the rest of Part
2, subgroups $H, K, L$ will be connected, and the disconnected
subgroups only enter through the final factor $D\efp$ and its
counterparts for quotient groups. 

To a first approximation, the idea is that the  cube is obtained by smashing together $r+1$ maps
of rings, with $S^0=A_i(0)\lra A_i(1)$ in the $i$th coordinate, so that 
$\Rti (a_0, \ldots, a_r)=\bigwedge_{i=0}^rA_i(a_i)$. However  we need
to refine this, so as to assemble information from individual
subgroups, and reflect containments of subgroups. 

The simplest coordinates are the $0$th and $r$th, where we have 
$A_0(1)=\siftyV{G}$ and  $A_r(1)=D\efp$. In the rank 1 case, this is
everything, so we obtain the usual diagram 
$$\diagram
S^0\sm S^0\rto\dto &\siftyV{G}\sm S^0\dto\\
S^0\sm D\efp\rto &\siftyV{G}\sm D\efp.
\enddiagram$$
To get a diagram of commutative rings, smashing with $\siftyV{G}$ is
replaced by localization with respect to it. 

Supposing $r\geq 2$ we now move on to the other coordinates. For $1\leq i\leq r-1$, for each
connected subgroup $K$, we take
$$A^K(1)=\siftyV{K}, $$
and then 
$$A_i(1 )=\prod_{\codim (K)=i}A^K(1)=\prod_{\codim (K)=i}\siftyV{K}.$$
The formula for codimensions $0$ and $r$ fits the same pattern, although there is only one term in the
product and containment imposes no restrictions. To make the formulae
typographically manageable  we need to introduce some more
notation. Indeed, in the $i$th spot we need to have index sets 
$I(i,0)$ and $I(i,1)$ for certain products. The index set $I(i,0)$ is
a singleton and 
\begin{equation}
\label{eqn.indices}
I(i,1)=\{ H \st H \mbox{ is connected and } \codim (H)=i\}. 
\end{equation}
Now we can define the ring spectrum to be placed at the $(a_r, \ldots,
a_0)$ vertex.  We recall that $\siftyV{H}\smb A$ denotes the Bousfield
localization of a $E_{\infty}$-ring spectrum $A$ with respect to
$\siftyV{H}$, and that this has the homotopy type of the ordinary
smash product $\siftyV{H}\sm A$. The functor $S^0\smb $ is the
identity. 

\begin{multline*}
\Rti (a_0, \ldots , a_r)=A_0(a_0)\smb \prod_{H_{1}\in I(1,a_{1})} 
\big[ A^{H_{1}}(a_{1})
\smb \prod_{H_{2}\in I(2,a_{2}), H_2\subset H_{<2}} 
\big[ A^{H_{2}}(a_{2}) \smb \cdots \\
\cdots \smb \prod_{H_{r-1}\in I(r-1,a_{r-1}), H_{r-1}\subset H_{<r-1}} 
\big[ A^{H_{r-1}}(a_{r-1})\smb A_r(a_r)\big] \cdots \big]\big]
\end{multline*}

\begin{remark}
(a) To help parse this, note that in the $s$th term we have  $S^0$ if $a_s=0$ and
otherwise it is the product of copies of $\sifty{H}$ as $H$ runs
through codimension $s$ subgroups contained in the earlier
subgroups (the notation $H_s\subset H_{<s}$ allows for the fact that
only terms with $a_t=1$ correspond to actual subgroups). 

(b) Note that the products include everything to the right of
them so the ordering of the vertices is important.  From now on, we
will often omit parentheses, relying on the `products include
everything to the right' convention to simplify typography. 
 
(c) This notation shows all structure maps clearly, but the formula
is easier to digest if   we pick out just those indices with $a_i\neq
0$, say $i_{c_0}<i_{c_1}<\ldots <i_{c_s}$. In this case 
if $c_s<r$ we have 
\begin{multline*}
\Rti (a_0, \ldots , a_r)=
\prod_{\codim H_{0}=c_0}  \big[ \sifty{H_0} \smb 
\prod_{\codim H_{1}=c_1, 
H_1\subset H_0}  \big[ \sifty{H_1} \smb \\
\prod_{\codim H_{2}=c_2, H_2\subset H_1}  \big[ \sifty{H_2} \smb 
\cdots 
\smb \prod_{\codim H_{s}=c_s, H_s\subset H_{s-1}}
\big[\sifty{H_s}\big]\cdots \big]
\big] \big] 
\end{multline*}
and if $c_s=r$ we have 
\begin{multline*}
\Rti (a_0, \ldots , a_r)=
\prod_{\codim H_{0}=c_0}  \big[ \sifty{H_0} \smb 
\prod_{\codim H_{1}=c_1, H_1\subset H_0}  \big[ \sifty{H_1} \smb \\
\prod_{\codim H_{2}=c_2, H_2\subset H_1}  \big[ \sifty{H_2} \smb 
\cdots 
\smb \prod_{\codim H_{s-1}=c_{s-1}, H_s\subset H_{s-2}} \big[ \sifty{H_{s-1}} \smb
D\efp \big] \cdots \big]
\big] \big] 
\end{multline*}

\end{remark}

The notation somewhat obscures the simplicity of this construction.
Thus in rank 2, we have
$$\tilde{R}(a_0,0, a_2)=A_0(a_0)\smb S^0\smb A_2(a_2)$$
and
$$\tilde{R}(a_0,1, a_2)=A_0(a_0)\smb  \prod_{\dim(H)=1}\big[ A^H(1) \smb A_2(a_2)\big] .$$
It is worth writing the diagram completely in this case. The layout is 

$$
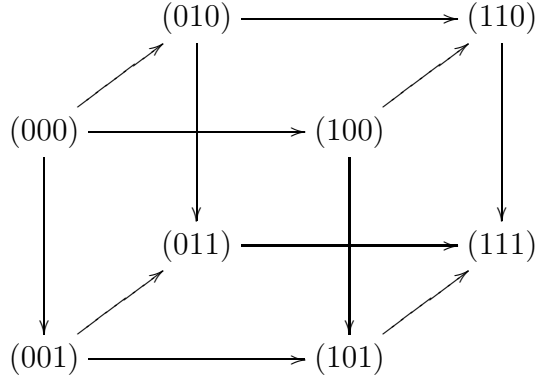
\diagram
&(010)\rrto \ddto&&(110)\ddto\\
(000)\rrto \ddto \urto&& (100) \ddto \urto&\\
&(011)\rrto &&(111)\\
(001)\rrto  \urto&& (101) \urto&
\enddiagram$$
and the diagram of ring spectra is as follows: 
$$
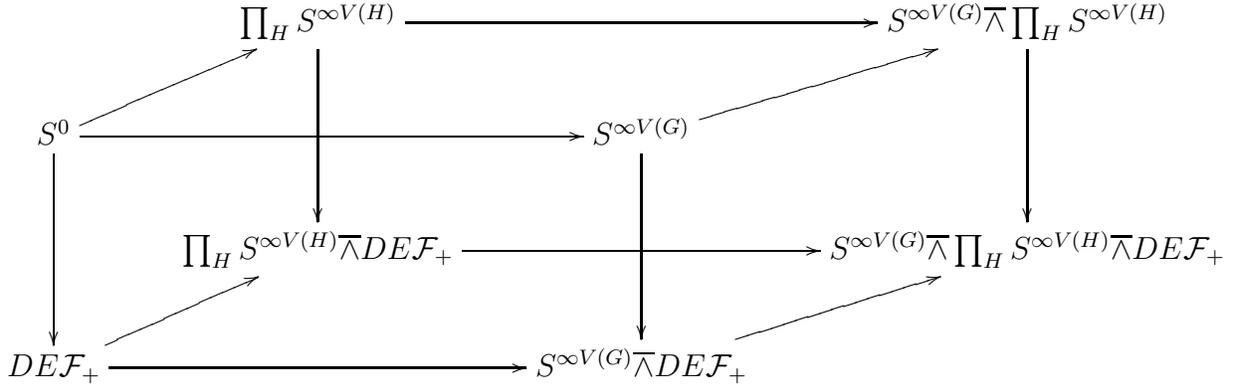
\diagram
&\prod_H\siftyV{H}\rrto \ddto&&\siftyV{G}\smb \prod_H\siftyV{H}\ddto\\
S^0\rrto \ddto \urto&& \siftyV{G} \ddto \urto&\\
&\prod_H\siftyV{H}\smb D\efp \rrto &&\siftyV{G}\smb \prod_H\siftyV{H}\smb D\efp\\
D\efp\rrto  \urto&& \siftyV{G} \smb D\efp\urto&
\enddiagram$$

One $r$ dimensional face will play a preferred  role in our proof that
this cube is a homotopy pullback, so we give a
special name to the $a_0=0$ face (the left hand face in the above
illustration). The  $r$-cube diagram $R'$ is defined by  
$$R'(a_{1}, \ldots, a_{r})=R (0,a_{1}, \ldots, a_{r}). $$
We note that 
$$R= (S^0\lra \siftyV{G})\smb R'.$$ 
This notation will be even more convenient when we  refine the filtration $S^0\lra
\siftyV{G}$.

\subsection{Observations about isotropy}
It is natural to consider a filtration of all subgroups by dimension, so we let
$$\cF^{\leq i}=\{ H\st \dim(H)\leq i\} \mbox{ and }
\cC_{\geq i}=\{ H\st \dim(H)\geq i\}$$
The first is a family and the second is a cofamily. We also need to
consider the subgroups above and below a fixed group $K$:
$$\Lambda (K)=\{ H\st
H\subseteq K \}  \mbox{ and } V(K)=\{ H\st K\subseteq H \}   .$$
Again, the first is a family and the second is a cofamily.

 The point is that the category of spectra with geometric isotropy in the cofamiliy
 $V(K)$ of subgroups  (the spectra ``over $K$'') is equivalent to the 
category of $G/K$-spectra. To obtain good inductive arguments we want
to express the naturally occurring sets of isotropy in terms of those
of the form $V(K)$.

We are burdened with a standard notation in which the geometric
isotropy is given by 
$\GI (\widetilde{E}\cF)=All\setminus \cF$, so we adopt the convention that
for any cofamily $\cC$ 
$$X\cC := X\sm \widetilde{E}(All \setminus \cC), $$
so that 
$$\GI (X\cC)=\GI (X) \cap \cC$$ and in particular
$$\bbS \cC = \widetilde{E}(All \setminus \cC),  \mbox{ giving \ \ }  \GI
(\bbS \cC)=\cC. $$
This notation extends naturally to families, and indeed to any
collection of subgroups which can be expressed as an intersection between a family and a cofamily.

We abbreviate further, taking 
$$\fbbS{i}=\bbS \cC_{\geq i}$$
and consider the filtration
$$\bbS=\fbbS{0}\lra \fbbS{1}\lra \fbbS{2}\lra \cdots \lra
\fbbS{r}=\siftyV{G}.$$
More precisely we realise this filtration by cofibrations in the category of
commutative ring spectra by a process of localization; this is
possible by Proposition \ref{prop:axiom} (8). 

\begin{lemma}
\label{lem:PBisotropy}
For any map $f: X\lra Y $ of $G$-spectra which  is  which is an
equivalence in $\Phi^H$-fixed points for $H$ of dimension $i$ then
the square 
$$\diagram 
\fbbS{i}\sm X \rto\dto & \fbbS{i+1}\sm X\dto \\
\fbbS{i} \sm Y\rto &\fbbS{i+1}\sm Y 
\enddiagram$$
 is a homotopy pullback. If the map $X\lra Y$ is a map of commutative
ring spectra and $\sm$ is replaced by $\smb$ then this is a homotopy
pullback of commutative ring spectra. 
\end{lemma}

\begin{proof} The space $ \fbbS{i+1}/\fbbS{i}$  has geometric isotropy
concentrated on subgroups $H$ of dimension exactly $i$. This means it can be
built from cells $G/K_+$ where $K$ has dimension $\leq i$. 
\end{proof}

We will  apply this to a large number of slightly different maps, but
it is worth highlighting one which embodies the philosophy. 

\begin{cor}
\label{cor:PBisotropyprelim}
For any $G$-spectrum $X$, the square 
$$\diagram 
\fbbS{i}\sm X \rto\dto & \fbbS{i+1}\sm X\dto \\
\fbbS{i} \sm \prod_{\dim(H)=i} \siftyV{H} \sm X\rto &\fbbS{i+1}\sm
\prod_{\dim(H)=i}\siftyV{H}\sm X
\enddiagram$$
where the products are over {\em connected} subgroups of dimension
$i$,  is a homotopy pullback. If the map $X$ is a commutative
ring spectrum and $\sm$ is replaced by $\smb$ then this is a homotopy
pullback of commutative ring spectra. 
\end{cor}

\begin{remark}

(i) The bottom left hand entry is equivalent to
$ \prod_{\dim(H)=i}\siftyV{H}\sm X$ since all terms are $\cF^{\leq  i-1}$-contractible. 

(ii) The essence of the corollary is that we can start with  $\fbbS{r}
\sm X =\siftyV{G}\sm X$, and build $X=\fbbS{0}\sm X$ in steps.  At
each stage $\fbbS{i} \sm X$ can be constructed as a homotopy pullback
from $\fbbS{i+1} \sm X$ by using only spectra of the form
$\siftyV{H}\sm X$ for subgroups $H$ of dimension
$i$. 

Since the category of module $G$-spectra  over $\siftyV{H}$ is equivalent to the category of
$G/H$-spectra, this establishes an inductive scheme. 
\end{remark}

\begin{proof}[of  \ref{cor:PBisotropyprelim}] 
We apply Lemma \ref{lem:PBisotropy} to the map $X\lra \prod_{H}\siftyV{H}\sm
X$, so we need to verify this is a non-equivariant equivalence in geometric $K$-fixed
points for $K$ of dimension $\leq i$.

If $K$ is a subgroup of dimension less than $i$ then all terms are
$K$-contractible. If $K$ is of dimension $i$, there is precisely one
factor in the product which is not $K$-contractible (namely with $H$
the identity component of $K$), and $\Phi^K\siftyV{H}=S^0$, so 
that $X\lra  \siftyV{H}\sm X$ is an  equivalence after applying
$\Phi^K$. 
\end{proof}

The variant that we will apply is obtained by adapting a special case
of this corollary. 

\begin{cor}
\label{cor:PBisotropy}
For  $X=\Rti (a_0, \ldots, a_{i-1}, 0, a_{i+1}, \ldots , a_r)$ and 
$Y=\Rti (a_0, \ldots, a_{i-1}, 1, a_{i+1}, \ldots , a_r)$, the square
$$\diagram 
\fbbS{i}\smb X \rto\dto & \fbbS{i+1}\smb X\dto \\
\fbbS{i} \smb Y\rto &\fbbS{i+1}\smb Y
\enddiagram$$
is a homotopy pullback of commutative ring spectra.
\end{cor}

\begin{proof}  First note that if $a_s\neq 0$ for some $s<i$, this is
  immediate, since if $H$ is of dimension $s>i$ then  $\siftyV{H}$
  is $\cF^{\leq i}$-contractible. 

Now consider the case $a_0=\cdots =a_{i-1}=0$. This is very close to the special case of Corollary
\ref{cor:PBisotropyprelim} in which 
\begin{multline*}
X= \prod_{H_{i+1}\in I(i+1,a_{i+1})}
\big[ A^{H_{i+1}}(a_{i+1} )
\sm \prod_{H_{i+2}\in I(i+2,a_{i+2}),   H_{i+2}\subset H_{<i+2} } 
\big[ A^{H_{i+2}}(a_{i+2}) \sm \cdots \\
\cdots \sm \prod_{H_{r-1}\in I(r-1,a_{r-1}),  H_{r-1}\subset H_{<r-1}}  
\big[ A_{r-1}^{H_{r-1}}(a_{r-1})\sm A_r(a_r)\big] \cdots \big]\big]
\end{multline*}
This is the value of $X$ in the present corollary. The main difference
is that instead of taking $Y$ to be given as $\prod_H \siftyV{H} \sm
X$, the products in the $H$th factor are now restricted to subgroups
of $H$.

To see that this does not alter the fact that we have a pullback square, we need only observe that
the omitted factors in the products in the $H$th factor are
$\cF/H$-contractible. Indeed,  if $K$ is a connected subgroup
with $K\not \subseteq H$, and $\tH$ has identity subgroup $H$ 
then $K \not \subseteq \tH$ and so $\siftyV{K}$ is
$\tH$-contractible. 
\end{proof}

\subsection{The isotropic cube is a homotopy pullback}
We are ready to prove that the isotropic cube is a homotopy pullback. 

\begin{prop}
\label{prop:CiPB}
The $\Ci$-diagram $\Rti$ is a homotopy pullback, which is to say that the sphere spectrum  $\bbS$
is the homotopy pullback of $\tilde{R}$ restricted to the punctured cube
$\PCi$:
$$\bbS \simeq \holim_{v\in \PCi}\Rti (v).$$
\end{prop}

\begin{remark}
The corresponding statement is also true for the diagram in which the
products in the definition of the ring spectrum are over {\em all}
$\siftyV{H}$ with $H$ connected  of  a fixed codimension (the proof is the same,
except that one applies Corollary \ref{cor:PBisotropyprelim} instead
of Corollary \ref{cor:PBisotropy}). The reason for restricting to
products over decreasing flags is to obtain an algebraically tractable result. 
\end{remark}

\begin{proof}[of Proposition \ref{prop:CiPB}] 
Some readers may find it helpful to refer to the case
of Rank 2 made explicit in Subsection \ref{subsec:Cipbrank2} whilst reading
this proof. 

The method is to use a  succession of intermediate homotopy pullbacks
inside the cube.  We place
the terms of the intermediate homotopy pullbacks along the $a_0$ edges of $\PCi$.
It is helpful to describe first the basic filtration we are using.

The general reconstruction process works by enlarging the diagram 
to permit the $0$th coordinate to run through the entire filtration 
$$\bbS=\fbbS{0}\lra \fbbS{1}\lra \fbbS{2}\lra \cdots \lra
\fbbS{r}=\siftyV{G}.$$
We do this by letting $a_0$ take on the fractional values $0=0/r,
1/r, \ldots , r/r=1$ and take 
$$\Rti (i/r, a_{1}, \ldots, a_r)=\fbbS{i}\smb \Rti' (a_{1}, \ldots ,
a_{r});  $$
For brevity we write
$$\Rti'(i/r)=\fbbS{i}\smb \Rti'$$
for these  $r$-cube diagrams.

The idea is to imagine filling in the values of the diagram from scratch.  To start
with,  we are given the values at $\PCi$ (this includes all
entries with $a_0=1=r/r$). We then show successively for $a_0=(r-1)/r,
(r-2)/r, \ldots , 1/r, 0/r=0  $ that the entries in the diagrams of ring spectra
$\Rti'(a_0)$ can be  filled in  (using only homotopy equivalences and homotopy pullbacks)
from values already filled in. The only value of real importance  is
$\bbS=\fbbS{0} =\Rti (0, \ldots, 0)$, but it is easier to describe a uniform procedure
which fills in other entries on the way. 

At the start, we are given the ring spectra $\Rti (a_0, \ldots , a_r)$  for
vertices of $\PCi$. This means all vertices with  $a_r\in \{0,1\}$ and
not all entries $a_i$ zero. We observe first that the entries at many
other points are equivalent to these. 

\begin{lemma}
\label{lem:easyfillin}
Provided $a_j=1$ for some $j\leq i$ we have an equivalence
$$\fbbS{i}\smb \Rti'(a_{1}, \ldots , a_{r})\simeq \Rti'(a_{1},  \ldots
, a_{r})$$
of commutative ring spectra. 
\end{lemma}

\begin{proof}
The mapping cone of the comparison map is $E\cF^{\leq i-1}_+\sm 
\Rti'(a_{1},  \ldots , a_{r}).$ If $\dim(H)=j$ then $\siftyV{j}$ is
$\cF^{\leq i-1}$-contractible so the mapping cone is contractible.  
\end{proof}

Now suppose that the entries of $\Rti'((i+1)/r)$ are filled in. To fill in 
the entries of $\Rti'(i/r)$ with $a_i=1$  we use Lemma
\ref{lem:easyfillin}, and for the points with $a_i=0$ we  apply Corollary \ref{cor:PBisotropy},
with $X=\Rti'(a_{1}, \ldots, a_{i-1}, 0, a_{i+1}, \ldots , a_r)$ and 
$Y=\Rti'(a_{1}, \ldots, a_{i-1}, 1, a_{i+1}, \ldots , a_r)$. 
\end{proof}

\subsection{The case of rank 2}
\label{subsec:Cipbrank2}
The above inductive scheme is sufficiently complicated that it seems worth
making one case explicit. 

Consider the following diagram. 
$$
\resizebox{0.9\textwidth}{!}{    
\diagram
&\prod_H\siftyV{H}\rrto^{\simeq} \ddto  &&\etf \smb
\prod_H\siftyV{H}\rrto \ddto  &&\etp \smb \prod_H\siftyV{H}\ddto \\
S^0\rrto \ddto \urto &&\etf\rrto \ddto \urto &&\etp \ddto \urto &\\
&\prod_H D\efp \smb \siftyV{H}\rrto^{\simeq}  &&\etf \smb \prod_HD\efp \smb
\siftyV{H}\rrto  &&\etp \smb \prod_HD\efp \smb \siftyV{H}  \\
D\efp\rrto \urto &&\etf \smb D\efp\rrto  \urto &&\etp \smb D\efp \urto &
\enddiagram
}
$$

We have used traditional names $S^0=\fbbS{0}$, $\etf =\fbbS{1}$ and
$\etp =\fbbS{2}$, where $\cF$ is the family of finite subgroups and
$\cP$ is the family of proper subgroups.  
The zeroth coordinate is horizontal (left to right on the printed
page), the first coordinate is into the paper (diagonally on the printed page) and 
the $r$th coordinate is vertical (downwards on the printed page). 
The left hand square is $\Rti'(0/2)$, the central square is $\Rti'(1/2)$ and
the right hand square is $\Rti'(2/2)$. 

Thus the  left and right hand end (except for
$S^0$) are in the $\PCi$-diagram of which we want to identify the
homotopy limit. The back central entries can be filled in by the equivalences
illustrated on the two left hand horizontals without affecting the
homotopy limit. Now the top and bottom faces of the right hand cube are
homotopy pullbacks. This means that we can fill in the two central entries on the front
face without affecting the whole homotopy limit.
 Finally the front face of the left hand cube is a homotopy pullback,
 so that $S^0$ is the homotopy limit of the original $\PCi$-diagram.  

\section{The sphere as a formal  pullback}
\label{sec:formalcube}
We now move towards introducing  the formal cube. As described above,
we will define this by extending $\Ci$ to  a larger diagram $\Cif$ and
then finding $\Cf$ inside it. We briefly explain the motivation. 

The $\Ci$-diagram does not do what we require,  since the terms  $\siftyV{K}$ are
not formal unless $K=G$. However a strategy is already apparent from
our work on the isotropic cube in lower ranks. To see the idea, we may imagine
that we have already completed the proof for lower ranks, and
constructed the $G/K$-sphere $S^0$ 
from formal ring $G/K$-spectra $B$. Accordingly, we can construct $\siftyV{K}=\siftyV{K}\smb
S^0$ from formal $G$-spectra $\siftyV{K}\smb B$. Most of the
spectra $B$ that occur are products of those of the form
$\siftyV{H/K}$ (and since  $\siftyV{K}\smb \siftyV{H/K}\simeq
\siftyV{H}$ these correspond to ring spectra in our diagram) which have
already been constructed in lower rank. There is  only one other spectrum $B$, 
namely $D\efkp$,  and it is  the most important one. This outlines why 
 the sphere can be constructed from spectra of the form
$\siftyV{K}\smb D\efkp$, and we will give a detailed proof below. 

We will extend the $\Ci$-cube to a larger poset $\Cif$ also containing the
formal cube $\Cf$, and we will extend $\Rti$ to $\Cif$. Now the
$a_r=1$ face of the $\Ci$ cube is the $a_r=1$ cube of $\Cf$  and 
$\Rti$ already takes formal values on that face. The values of $\Rti$
on the $a_r=0$ face of $\Ci$ are not formal, and for each point we give
a new value at the corresponding point of $\Cf$. The new formal ring is obtained
by identifying the smallest codimension $c$ for which $\siftyV{K}$
(rather than $S^0$) occurs with  $\codim(K)=c$ and then smashing with 
$D\efkp$.  

We flesh out this sketch in the course of the next few subsections,
starting by describing the  larger diagram
$\Cif$ and then  identify $\Cf$  inside it. 

\subsection{A subdivision of the isotropic cube}
\label{subsec:RttoponCif} 

The diagram $\Cif$ is obtained from $\Ci$ by inserting  new layers in the  $a_r$
direction. It may be helpful to refer to the rank 2 pictures in
Subsection \ref{subsec:Cfpbrank2} whilst reading this account. 

Altogether we have $r+1$ layers placed at $a_r=i/r$ for
$i=0, 1,\ldots , r$, interpolating between  the $a_r=0$ and $a_r=1$
layers, which  are just as before. 

We will be using    maps to relate the various ring spectra
$D\efkp$ as $K$ varies. Indeed, $D\efkp$ is a commutative ring
$G/K$-spectrum by Proposition \ref{prop:axiom} (9) and
 if $L\subseteq K$ there is a map 
$$\infl_{G/K}^{G/L} D\efkp \lra D\eflp  $$
of ring $G/L$-spectra.  To see where this comes from, we observe that its adjunct
$$\eflp \sm \infl_{G/K}^{G/L}D\efkp \lra S^0$$
is obtained by composing the $G/L$-map $\eflp \lra \efkp$ with
evaluation. 

If we have any decreasing sequence
$$G = H_0\supseteq H_1 \supseteq \cdots \supseteq
H_{r-1}\supseteq H_r=1$$
of connected subgroups with $\codim (H_i)=i$, then, omitting notation for inflation,
we have a sequence of
maps of ring $G$-spectra
$$S^0=DE(\cF /G)_+ \lra D(E\cF /H_1)_+ \lra \ldots \lra D(E\cF /H_{r-1})_+ \lra D(E\cF /1)_+ =D\efp$$

To define the $\Cif$ diagram of rings we use the same formula as before
except that the range of values of $a_r$ is extended to the fractional
values and the $r$th entry becomes dependent on other coordinates. 
More briefly,  $A_r(a_r)$ is replaced by $A_r^{i_1, \ldots, i_r}(a_0,
\ldots , a_r)$.  
Thus, with $I(i,0)$ a singleton and  $I(i,1)=\{ H \st H \mbox{ is connected and } \codim (H)=i\} $
as before,  we define the ring
$G$-spectrum  to be placed at the $(a_0, \ldots, a_r)$ vertex:
\begin{multline*}
\Rti (a_0, \ldots , a_r)=A_0(a_0)\smb \prod_{H_{1}\in I(1,a_{1})} 
\big[ A^{H_{1}}(a_{1})
\smb \prod_{H_{{2}}\in I(2,a_{2}), H_2\subset H_{<2}} 
\big[ A^{H_{2}}(a_{2}) \smb \cdots \\
\cdots \smb \prod_{H_{r-1}\in I(r-1,a_{r-1}), H_{r-1}\subset H_{<r-1}} 
\big[ A^{H_{r-1}}(a_{r-1},)\smb A_r^{i_0, \ldots,
  i_r}(a_0, \cdots, a_r)\big] \cdots \big] \big]
\end{multline*}
For the last term, we take
$$A_r^{i_1, \ldots, i_r}(a_0, \cdots, a_r)=\infl_{G/H}^GDE\cF/H_+$$
where the subgroup $H=H(i_1, \ldots, i_r; a_0, \ldots,  a_r)$ is determined as
follows. When $a_r=s/r$, we consider the sequence $a_0, \ldots , a_s$;
if it is zero we take $H =H_{i_0} = G$, and otherwise we  find the last nonzero term $a_t$ and 
take the codimension $t$ subgroup $H_{{t}}$:
$$H(i_1, \ldots, i_r; a_0, \ldots , a_{r-1}, s/r):=H_{{\lnz (a_0,\ldots , a_s)}} $$
where 
$$\lnz (a_0, \ldots , a_s)=\max (\{ t \st a_t\neq 0\} \cup \{0\}). $$
Note that since 
$\lnz (a_0, \ldots, a_s)\leq \lnz (a_0, \ldots, a_s, a_{s+1})$
we have  an inclusion 
$$H(i_0, \ldots , i_{r-1};a_0, \ldots , a_{r-1}, s/r) \supseteq H(i_0, 
\ldots , i_{r-1};a_0, \ldots , a_{r-1}, (s+1)/r)$$
so that we do have the appropriate comparison maps. 

The diagram $\Cif$ is not a cube, so we should state explicitly that
the punctured diagram $P\Cif$ is obtained by omitting the $r$ points $(0, \ldots ,0, a_r)$ with
$a_r\neq 1$, which  are the points where $\Rti $ takes the value $\bbS$.

\subsection{Selecting the formal  cube}
The formal cube $\Cf$ consists of the $a_r=1$ face together with an
opposite face that we need to describe.  First, the initial vertex is
the point $(0, \ldots , 0)$. Next, the
point in the opposite face  corresponding to  a non-zero $(a_0,
\ldots , a_{r-1})$ can  be found by looking for the least value of $a_r$ 
for which the entry at $(a_0, \ldots, a_{r-1}, a_r)$ is formal. 
The formal entries in the diagram  are those with a term $D\efkp$ for some $K$, where we take this to  include
the terms $DE\cF/G_+=S^0$ when $a_0=1$. Thus the least value of $a_r$
with a formal entry is $a_r=\lnz (a_0, \ldots , a_{r-1})/r$. 
Continuing with  the convention that $\lnz (0,\ldots, 0)=0$, 
$$\Cf=\{ (a_0, \ldots, a_r)\st a_r=1 \mbox{ or } a_r = \frac{\lnz
  (a_0,\ldots, a_{r-1})}{r}\}. $$ 

As a poset,  these vertices form a cube. To see this,  we identify the vertex
$(a_0, \ldots , a_r)$ of $\Cf$ with the subset 
$$S (a_0, \ldots , a_r) =\{ i \st a_i=1\}.$$

To see that the morphisms correspond to containment of subsets (so that
$\Cf$ is a cube) we note that if  $(a_0, \ldots , a_{r-1}) $ and
$(b_0, \ldots , b_{r-1})$ differ only by changing some entries $a_i=0$
to $b_i=1$ (so $S(a)\subseteq S(b)$) then $a_r:=\lnz(a_0, \ldots , a_{r-1}) \leq \lnz(b_0, \ldots ,
b_{r-1})=:b_r$, so that there is a path from $ (a_0, \ldots, a_{r-1}, a_r)
$ to $(b_0, \ldots, b_{r-1}, b_r)$ in $\Cif$.

\begin{prop}
\label{prop:CfPB}
The inclusion $\Cf \subseteq \Cif$ induces an equivalence
$$\holim_{v\in \PCf} \Rti (v) \simeq \holim_{v\in \PCif} \Rti (v).$$
\end{prop}

Before proving Proposition \ref{prop:CfPB},  we note that with
Proposition \ref{prop:CiPB} and the fact that $\PCi$ is cofinal in $\PCif$,  
it implies that $\bbS$ is the homotopy pullback of the formal
ring spectra. 

\begin{cor}
\label{cor:SisPCfpb}
The $\Cf$-diagram $\Rti$ is a homotopy pullback, which is to say that
$\bbS$  is the homotopy pullback of the $\PCf$-diagram $\Rti$:
$$\bbS \simeq \holim_{v\in \PCf}\Rti (v).\qqed$$ 
\end{cor}

It then follows from Proposition~\ref{prop-gen-pb} that we have the
desired Quillen equivalence. For this statement we revert to the full notation $\Rttop =\Rti$.

\begin{cor}\label{cor.pullback}
There is a Quillen equivalence between equivariant $G$-spectra, modelled by the category of $\bbS$-modules, and the cellularization of the diagram-injective model structure on $\Rttop$-modules.
$$\Gspectra \simeq  \cellmodcatG{\Rttop} $$
\end{cor}

It remains to give the proof comparing the limits over $P\Cf$ and
$P\Cif$. \\[2ex]

\begin{proof}[of Proposition \ref{prop:CfPB}]
Some readers may find it helpful to refer to the case
of Rank 2 made explicit in Subsection \ref{subsec:Cfpbrank2} whilst reading
this proof. 

We will work in the diagram  $\Cif$ (i.e.,  permitting $a_r\in \{0/r,
1/r, \ldots , r/r \}$). In Subsection \ref{subsec:RttoponCif} we defined $\Rti (v)$ for all
vertices $v$. The proof here consists of  showing  how we could recover all of them
from the entries in $\PCf$ alone,  using homotopy pullbacks. This will
show in particular that the entry $\Rti (0, \ldots , 0)=\bbS$ at the
initial vertex is the homotopy pullback of the $\PCf$-diagram $\Rti$.

We view this as starting with an empty slate, adding the entries at 
points of $\PCf$ and steadily filling in the values at different
vertices by using homotopy pullbacks of entries filled in previously.

First, we fill in all the points of $\PCif$ which admit a map from 
 an entry of $\PCf$; this does not change the homotopy pullback, since
$\PCf$ remains cofinal. For example, since $(1, 0, \ldots, 0)$ is in
$\PCf$, we may fill in all vertices $(1, 0, \ldots , 0, a_r)$ with
$a_r\neq 1$, which  all  have value $\siftyV{G} \smb DE\cF/G_+ \simeq 
\siftyV{G}$.

The $\Cif$-diagram $\Rti$  takes the value $\bbS$ at  $(0,0, \ldots , 0,
a_r)$ for $a_r\neq 1$. The rest of the diagram is called $P\Cif$ and  has  $r+1$ initial 
points, namely the vertices $v_c=(0, \ldots , 0, 1, 0, \ldots, 0)$
(where the  $1$ in the $c$th position) for $0\leq c \leq r$.  The entries at
$v_r=(0, \ldots , 0, 1)$ (viz $D\efp$)   and $v_0=(1, 0, \ldots , 0)$ (viz
$\siftyV{G}$) lie in $\PCf$  and are therefore already filled in. The entry when $0<
c <r$  is $\siftyV{c}:=\prod_{\codim (H)=c}\siftyV{H}$, and we need to
explain how this is filled in by homotopy pullbacks. 

Note first that  $\siftyV{c}$ is also the entry at the points $(0, \ldots , 0, 1, 0,
 \ldots, a_r)$   for $a_r=0/r,1/r, \ldots 
 (c-1)/r$. The point with $a_r= c/r$ lies in $\PCf$, and the entry 
there is therefore  filled in at the start.  To fill in the
entry at the initial vertex $v_c=(0, \ldots 0, 1, 0, \ldots 0)$  we consider a
$(c+1)$-cube $\Cf (c)$ with initial vertex at $(0, \ldots , 0, 1, 0,
\ldots, 0)$. More precisely
$$\Cf (c)=\{ (a_0,a_{1}, \ldots, a_{c-1},1,
0, \ldots, 0, a_r)\st a_r =0 \mbox{ or } c/r\}. $$
We note that entries at $P\Cf (c)$ are already filled in, and the
following lemma  shows that the entry $\siftyV{c}$ can be filled in as
a homotopy pullback of entries on $P\Cf(c) $.

\begin{lemma}
\label{lem:CfcPB}
The $\Cf (c)$-diagram $\Rti$ is a homotopy pullback, which is to say that
$\siftyV{c}$ is the homotopy pullback of the $\PCf(c)$-diagram $\Rti$.  
\end{lemma}

\begin{proof}
The proof follows precisely the same pattern as Proposition
\ref{prop:CiPB} above. The cube is rather similar to a product of
copies of the isotropic pullback diagrams for the rank $c$
quotients, but it is slightly different, so we provide some reference
points for the proof. 

We first note that $\siftyV{c}=\fbbS{c}\smb \siftyV{c}$ and then 
filter the $0$th coordinate by 
$$\fbbS{r-c}\lra \fbbS{r-c+1}\lra \ldots \lra \fbbS{r}=\siftyV{G}. $$
We refine the map from $a_0=0$ to $a_0=1$ into $c$ steps. The
structure of  the proof is precisely like that of Proposition
\ref{prop:CiPB}. The only difference is that our application of 
Corollary \ref{cor:PBisotropy} is in the special case 
 $X=\Rti (a_0, \ldots, a_{i-1}, 0, a_{i+1}, \ldots , a_{c-1}, 1, 0,
\ldots , 0, a_r)$ and 
$Y=\Rti (a_0, \ldots, a_{i-1}, 1, a_{i+1}, \ldots , a_{c-1}, 1, 0,
\ldots , 0, a_r)$. 
\end{proof}

Since we have now filled in the initial points of $P\Cif$, we may fill in the remaining vertices without changing
the homotopy pullback. Accordingly the homotopy pullback over $\PCf$
agrees with that over $\PCif$ as required. 
\end{proof}

\subsection{The case of rank 2} 
\label{subsec:Cfpbrank2}
The above account is again sufficiently complicated that it is worth
making one case explicit. For typographical reasons we have only
illustrated the case  $r=2$, though in fact some features only appear
at rank 3. As before, we have used traditional names $S^0=\fbbS{0}$, $\etf =\fbbS{1}$ and
$\etp =\fbbS{2}$, where $\cF$ is the family of finite subgroups and
$\cP$ is the family of proper subgroups.

Consider the diagram 
$$\diagram
&\prod_H\siftyV{H}\rrto \ddto &&\etp \smb \prod_H\siftyV{H} \ddto\\
S^0\rrto \urto \ddto&& \etp \urto\ddto\\
&\prod_H\siftyV{H}\smb D\efhp \rrto \ddto &&\etp \smb \prod_H \siftyV{H}
\smb D\efhp \ddto\\
S^0\rrto \urto \ddto&& \etp \urto\ddto\\
&\prod_H\siftyV{H}\smb D\efp \rrto &&\etp \smb \prod_H \siftyV{H}
\smb D\efp \\
D\efp\rrto \urto && \etp \smb D\efp \urto
\enddiagram$$
The whole diagram is $\Cif$.  The top square has $a_2=0/2$ the middle
square has $a_2=1/2$ and the  bottom square has  $a_2=2/2$. The cube
$\Cf$ consists of the bottom square, the middle horizontal on the back
face and the top front edge.

Wiping the slate clean, and starting with the entries in $P\Cf$ we
describe how to fill in the other entries.  First, we may fill in  $\etp \smb \prod_H\siftyV{H}$ at the top right back position
without changing the homotopy pullback since it admits a map from
$\etp$ at the top right front.  Now Lemma \ref{lem:CfcPB} with $c=1$ states that the top back square is a
homotopy pullback so that we have filled in  $\prod_H\siftyV{H}$ at the top,
back left. This gives all vertices of $\PCi$  from those of $\PCf$, and
$S^0$ is the homotopy pullback of $\PCi$ by Proposition
\ref{prop:CiPB}.

\subsection{Diagrams }
Now that we have a $\PCf$-diagram $\Rttop$ of ring $G$-spectra we
should explicitly introduce the corresponding diagrams in other
contexts. 

\begin{defn}
\label{defn:PCfdiagrams}
From the $\PCf$ diagram $\Rttop$ of commutative ring $G$-spectra we
form
\begin{enumerate}
\item the $\PCf$ diagram $\Rtop =(\Rttop)^G$ of commutative ring
  spectra,
\item  the $\PCf$ diagram $\Rt$ of commutative DGAs obtained from
  $\Rtop$ using the fact \cite{s-alg} that the category of commutative
  $H\Q$-algebras is equivalent to commutative DGAs over $\Q$ (see Section
  \ref{sec:spectratoDGAs}), 
\item the $\PCf$ diagram $\Ra =\pi^G_*(\Rttop)=\pi_*(\Rtop)=H_*(\Rt)$
  of graded rings. 
\end{enumerate}

\end{defn}

\part{From $G$-spectra, through spectra to algebra}

\section{Fixed point equivalences for module categories}
\label{sec:removeequiv}

The category of $G$-spectra is modelled by $\bbS$-modules in
$G$-spectra, and since $\bbS$ is a homotopy pullback of the
$\PCf$-diagram $\Rttop$ of ring $G$-spectra, $G$-spectra is also
modelled by a category of $\Rttop$-modules in $G$-spectra.
Our next step is to remove equivariance and find a model in terms of a category of
non-equivariant module spectra over a $\PCf$-diagram of
non-equivariant ring spectra.  

\subsection{The fixed point adjunction for module spectra}
We briefly recall some results of \cite{modulefps} for an individual
ring $G$-spectrum. 

The context is that when we are given a fibrant  ring $G$-spectrum, $\Ati$ with
fixed point spectrum $A=\Ati^G$ there is a Quillen adjoint pair

$$\RLadjunction{\Psi^G}{\AtimodGspectra}{\Amodspectra}{\einfl_1^G} .$$
Here $\Psi^G$ takes Lewis-May fixed points and then uses the fact
that  the fixed point functor is lax monoidal by Proposition \ref{prop:axiom} (10) to view the result as a
module over $A$. The inflation functor views a non-equivariant
spectrum as a $G$-spectrum by pullback along the quotient and then
extends scalars along $\infl A\lra \Ati$ to give an $\Ati$-module. The
tilde on $\infl_1^G$ refers to this extension of scalars (this was
omitted in \cite{modulefps}). 

\begin{remark}
We note that in \cite {modulefps} we worked with orthogonal
 spectra, but we may compose that Quillen adjunction with the Quillen
 equivalence between orthogonal spectra and orthogonal $\mcL$-spectra
 noting that the two adjunctions have the same direction and therefore
 give a single Quillen pair to which the discussion of
 \cite{modulefps} applies without change.

In \cite{modulefps} we did not discuss monoidal
structures on the module categories, since we weren't assuming the
 rings were commutative. However we remark here that the Quillen pair
 is monoidal. Indeed, we are composing (1) the adjunction
 between orthogonal spectra and orthogonal $\mcL$-spectra (2)  the  fixed
 point-inflation adjunction and (3) a change of rings adjunction all
 of which are weak symmetric monoidal Quillen adjunctions.   
\end{remark}

Since the category $\Amodspectra$ is generated by $A$, the Cellularization
Principle gives a Quillen equivalence
$$\AticellAtimodGspectra \simeq \Amodspectra.$$
Surprisingly often (in particular \cite[4.4]{modulefps} when $G$ is a torus and $A$ has
Thom isomorphisms), the category $\AtimodGspectra$ is generated by $\Ati$, so that 
we obtain a Quillen equivalence
$$\AtimodGspectra\simeq \Amodspectra$$
showing that a category of equivariant module spectra is equivalent to
a category of non-equivariant module spectra.

Before turning to our applications it will be helpful to mention three
special cases. 

\begin{example} {\em (Eilenberg-Moore Theorem \cite[8.1]{modulefps})}
We take $\Ati =\DH EG_+$, so that $A =\DH BG_+$ and obtain a version of the 
Eilenberg-Moore theorem:
when $G$ is a torus, there is a Quillen equivalence
$${\DEGmod}\simeq {\DBGmod}.  $$ 

We emphasize that no cellularization is necessary here for a torus. 
\end{example}

\begin{example} {\em (Spectra over $G$ \cite[VI.5.3]{mm},
    \cite[3.3]{modulefps}; no rationalization is necessary})
We take $\Ati =\siftyV{G}$ so that $A=S^0$ and note that the category modules over
$\siftyV{G}$ is a model for spectra over $G$ (i.e., for spectra with
geometric isotropy in $\{G\}$), whilst the category of $S^0$-modules
is the category of spectra. Thus we recover the well
known result that there is a Quillen equivalence
$$ {\mbox{$G$-spectra/$G$}}\simeq {\mbox{spectra}}.$$
\end{example}

The variant of the first example with all finite isotropy collected
together is directly relevant to us. 
\begin{example} {\em (Almost free spectra \cite[Corollary 9.2]{modulefps})}
Continuing with $G$ a torus and taking $\Ati =D\efkp$ we obtain 
$$\mbox{$D\efkp$-mod-$G/K$-spectra}\simeq 
\mbox{$D(\efkp)^{G/K}$-mod-spectra}.$$
 \end{example}

\subsection{Fixed point adjunctions for diagrams of ring $G$-spectra}

We now move to the case of {\em diagrams} of ring spectra. Suppose  $\Rti$
is a diagram of ring $G$-spectra, fibrant in the diagram-injective
model structure and consider the corresponding diagram $R=\Rti^G$ of
spectra where fixed points are applied objectwise. We may again consider the
diagram-injective model categories of $\Rti$-module $G$-spectra and
$R$-module spectra and once again form the Quillen pair
$$\RLadjunction{\Psi^G}{\RtimodGspectra}{\Rmodspectra}{\einfl_1^G} . $$

\begin{lemma}
\label{lem:Qequivobjectwise}
The Quillen adjunction on diagrams with the diagram-injective model
structure is a Quillen equivalence provided it is a Quillen
equivalence objectwise. 
\end{lemma}

\begin{proof}
We note that unit and counit when evaluated at any vertex give the unit and counit of the adjunction for
a single ring $G$-spectrum. We claim this is also true for the derived
unit and counit. Since weak equivalences are detected
objectwise, this will suffice. 

To see that the statement about the derived unit and counit follows
from that about the underived ones, we need to consider fibrant and
cofibrant replacement.  For fibrant replacement the implication is
clear  since fibrancy is defined objectwise. For cofibrant replacement, we note that
cofibrant diagrams are objectwise cofibrant. Finally, weak
equivalences of equivariant orthogonal spectra are defined in terms of homotopy
groups of fixed points so in the light of Proposition \ref{prop:axiom} (2), 
 fixed points  preserve all weak equivalences. It follows that the derived unit
and counit of the Quillen pair on diagram categories are objectwise 
the derived unit and counit. 
\end{proof}

 \subsection{The fixed point adjunction for $\Rttop$.}
 We consider the special case $\Rti=\Rttop$ of the above discussion. 
The category of spectra is generated by the cells $G/H_+$ as $H$
varies over closed subgroups of $G$ and the
cellularization in the following statement is  with respect to the
images of these  generating cells.

\begin{thm}
\label{thm:RtopmodisRtildetopmod}
There is a Quillen equivalence
$$\RLadjunction {\Psi^G}
{\Rttopmod}
{\mbox{$\Rtop$-mod-spectra}}
{\einfl_1^G}. $$
It follows by cellularizing both categories that there is a Quillen equivalence
$$\RLadjunction {\Psi^G}
{\cellRttopmod}
{\cellRtopmod}
{\einfl_1^G}. $$

\end{thm}

\begin{proof}[of \ref{thm:RtopmodisRtildetopmod}]
Without changing notation, we take the fibrant replacement of 
$\Rttop$ in the diagram-injective model category 
of $\PCf$-diagrams of commutative ring $G$-spectra~\cite[5.1.3]{hovey-model}.   
By \cite[Lemma 4.2]{diagrammodcats} the category of modules over this fibrant replacement
is Quillen equivalent to the original category $\Rttopmod$.   

By Lemma \ref{lem:Qequivobjectwise}  it suffices to deal with the individual $G$-spectra
at a particular vertex $v$ of $\PCf$, so we take $\Ati =\Rti (v)$ for
some vertex $v$. 

For any ring $G$-spectrum $\Ati$  we get the equivalence 
$$\mbox{$\Ati$-cell-$\Ati$-mod-$G$-spectra}\simeq
\mbox{$A$-cell-$A$-mod-spectra}.$$
It is clear that $A$ generates the category of $A$-modules so that the
$A$-cellularization on the right is a Quillen equivalence. It remains only to
show that the cellularization on the left has no effect. 
  
To establish that the $\Ati$-cellularization on the left is also a
Quillen equivalence, it suffices to show that $\Ati$ generates
the category of $\Ati$-modules. The argument (as in \cite[4.4]{modulefps})
is to show that cells $G/H_+$ are all built from complex representation
spheres. 

If  $\Ati$ has Thom isomorphisms this is exactly as in 
\cite[4.4]{modulefps}, but we need the slightly more general argument
from  \cite[Section 9]{modulefps}.  We will show that for each complex representation $W$
we may express $\Ati$ as a finite  product $\Ati\simeq \prod_i \Ati_i$  of factors $\Ati_i$ 
so that $\Ati_i \sm S^W$ is a $G$-fixed suspension of $\Ati_i$. This
will show that 
$\Ati \sm S^W$ is in the thick category generated by $\Ati$ as
required. 

Now, turning to the proof,  $\Ati =\Rttop (v)$ and suppose that the last non-zero entry of $v$
is of codimension $c$. Then $\Ati$  takes the form
$$\Ati= \prod_{\codim H_0=c_0}\siftyV{H_0}\smb \prod_{\codim H_1=c_1}
\cdots  
\prod_{\codim H=c} \siftyV{H}\smb DE\cF/H_+$$
Furthermore $D\efhp\simeq \prod_{\tilde{H}}DE\lr{\tilde{H}}$, where
the product is indexed by closed subgroups $\tilde{H}$ with identity
component $H$.  
First note that  $S^W$ admits the structure of a finite $G$-CW complex, and therefore can be
moved inside all the products. For each $H$, we have  $W=W^H\oplus
W'(H)$ and $\siftyV{H}\sm S^W\simeq \siftyV{H}\sm S^{W^H}$ so that 
 $$
\Ati \sm S^V \simeq  
 \prod_{\codim H_0=c_0}\siftyV{H_0}\smb \prod_{\codim H_1=c_1}
\cdots  
\prod_{\codim H=c} \siftyV{H}\smb DE\cF/H_+\smb S^{V^H}$$
Now if $\tilde{H}$ has identity component $H$,  we use the Thom
isomorphism for Borel cohomology of $H$-fixed points \cite[8.1]{tnq1}
to give an equivalence
$$DE\lr{\tilde H} \sm S^{V^H}\simeq DE\lr{\tilde H} \sm
S^{|V^{\tilde{H}}|}.$$
Collecting together all the factors with the same suspension: 
$$\Sigma_i=\{\tilde{H} \st \codim (\tilde{H})=c \mbox{ and } \dim
  (V^{\tilde{H}})=i\}$$
we obtain a decomposition   $\Ati \simeq \prod_i \Ati_i$ as required. 
\end{proof}

\subsection{Modules over product rings}
We are repeatedly working with infinite products $R=\prod_i R_i$ of ring
spectra $R_i$, and we let $e_i$ be the idempotent projecting onto the
$i$th factor.  Even in algebra, such infinite products are poorly
behaved (for example infinite products of Noetherian rings need not be
Noetherian). If $M$ is a module over $\prod_iR_i$ and we take
$M_i=e_iM$ then we have maps
$$\bigoplus_i M_i \lra M \lra \prod_i M_i .$$
The first is a monomorphism, but typically neither will be an
isomorphism (for example if we take $M=\prod_iR_i/\bigoplus_iR_i$ then
$M_i=0$ for all $i$).

It seems worth observing that from the point of view of model
categories we may rather generally apply the Cellularization Principle \cite{cellprin} 
to recover the more familiar product of module  categories from the category of
modules over the product ring by suitable
cellularization. 

\begin{lemma}
We have a Quillen equivalence
$$\mbox{$\{ R_s\}_s$-cell-$\left( \prod_sR_s\right)$-modules} \simeq 
\prod_i\left[  \mbox{$R_i$-modules} \right]. $$
\end{lemma}

\begin{proof}
For each $s$ we have the projection $\pi_s: R \lra R_s$ inducing a
restriction on module categories. This has both a left and a right
adjoint, and the natural map from the extension of scalars to
coextension of scalars is an isomorphism (the idempotent subobject
agrees with the idempotent quotient object). Combining these we obtain 
$$ p:\mbox{$R$-mod}\lra \prod_s\left[\mbox{$R_s$-modules}\right]$$
whose right adjoint $p^R$ takes the product of the terms and whose left
adjoint $p^L$ takes the sum. 

The adjoint pair $(p^L,p)$ is a Quillen pair if the categories have
the injective model structures. The adjoint pair $(p,p^R)$ is a
Quillen pair if the categories are given the projective model structures.

In the second case, the objects $R_s$ are small generators in
$\prod_s\left[\mbox{$R_s$-modules}\right]. $
Since both $p$ and $p^R$ preserve all weak equivalences, the unit and
counit are equivalences on the generators $R_s$ and 
we may apply the Cellularization Principle to  give the desired
conclusion. 
\end{proof}

\section{From spectra to DGAs}
\label{sec:spectratoDGAs}

In this section we observe that the results from~\cite{s-alg}
show very directly that the category 
of module spectra over the diagram $\Rtop$
of commutative ring spectra is Quillen equivalent to a category of differential
graded modules over a diagram $\Rt$ of commutative DGAs.   It then follows
that the cellularizations of these model categories are also Quillen equivalent.
Since \cite{s-alg} is based on symmetric spectra, we use Proposition \ref{prop:axiom} (12) to show that there is a Quillen equiavlence between the respective categories of modules over $\Rtop$ and $\FF\Rtop$.

We next apply the functors from~\cite{s-alg} to move from symmetric spectra to differential graded modules.
In more detail, in~\cite[1.1]{s-alg} a composite functor $\Theta$ is defined which produces a Quillen equivalence between $H\bZ$-algebra spectra 
and DGAs. 
Given an $H\bZ$-algebra spectrum, 
$B$, 
it is shown in~\cite[2.15]{s-alg} that the category of module spectra
over $B$ is Quillen equivalent to the category of differential graded
modules over a DGA $\Theta B$.  
Furthermore, rationally there is a second functor $\Theta'$
which is symmetric monoidal, so that  it takes rational commutative rings
spectra to rational commutative DGAs. Finally, over the rationals the two
functors are naturally equivalent, so that by \cite[1.2]{s-alg},  
if $B$  is a commutative $H\bQ$-algebra then
$\Theta B$ is naturally weakly equivalent to the commutative DGA $\Theta' B$.

\begin{defn}
\label{defn:Rt}
Applying functors to the  $\PCf$-diagram of commutative rational ring spectra
$\Rtop$, we define $\Rt$ to be the  $\PCf$-diagram $\Theta' (H\bQ \sm \FF\Rtop)$ of 
commutative DGAs. 
\end{defn}

 Note, throughout this section we are implicitly considering the standard (diagram projective)
 model structures {from~\cite[3.1(i)]{diagrammodcats} on modules over diagrams of rings.}

\begin{prop}  \label{prop-alg-qe}
There is a zig-zag of Quillen equivalences
$$\Rtopmod\simeq_Q \Rtmod$$
 between the category of module spectra $\Rtopmod$ 
and the category of differential graded modules $\Rtmod$.
\end{prop}

\begin{proof}
As mentioned above, the first step is a Quillen equivalence between $\Rtopmod$ over 
$1$-spectra and $\modcat{\FF\Rtop}$ over symmetric spectra by
Proposition \ref{prop:axiom} (12) extended to diagrams of rings.  Since $\Rtop$
is rational, the unit map $\FF\Rtop \to H\bQ \sm \FF\Rtop$ is a weak equivalence 
which induces a Quillen equivalence on the associated module categories by
extension and restriction of scalars, {~\cite[4.2]{diagrammodcats} and~\cite[5.4.5]{hss}.}

Combining these steps with \cite[2.15]{s-alg} 
produces a Quillen equivalence between $\Rtopmod$ and 
$\modcat{\Theta (H\bQ \sm \FF\Rtop)}$.  Since $H\bQ\sm \FF\Rtop$ is a diagram of
commutative $H\bQ$-algebras, it follows from the proof of \cite[1.2]{s-alg} that $\Theta' (H\bQ\sm\FF\Rtop)$
is a diagram of commutative rational DGAs  which is weakly equivalent to the
diagram $\Theta (H\bQ\sm\FF\Rtop)$.  

{By~\cite[4.2]{diagrammodcats} and~\cite[5.4.5]{hss},}
extension and restriction of scalars
over these weak equivalences produce the last steps in the stated zig-zag of
Quillen equivalences. 
\end{proof}

The Cellularization Principle, \cite[Corollary 2.8]{cellprin} shows that cellularization preserves
zig-zags of Quillen equivalences as long as the cells in the target
category are taken to be the images under the relevant derived
functors of the cells in the 
source category. Here we begin with the cellularization of $\Rtopmod$ with respect to 
the images of $G/H_+$ as $H$ runs through closed subgroups.   
Then, at each of the next steps, the cells are the images of $G/H_{+}$ under the appropriate derived functor.  

\begin{cor}
There is a zig-zag of Quillen equivalences 
$$\cellRtopmod \simeq_Q \cellmodcatGG{\Rt}$$
between the cellularizations of the model categories
in Proposition \ref{prop-alg-qe}. 
\end{cor}

\section{Formality}
\label{sec:formality}
We have shown that the category of rational $G$-spectra is equivalent
to the cellularization of modules over a suitable $\PCf$ diagram of
commutative DGAs. On the other hand, we know very little about the
diagram except its homology and that the terms are commutative. The
purpose of this section is to show that this is enough to determine the diagram up to equivalence. 

\subsection{Terminology}

A map $f: \Rtilde \lra \Rtilde'$ of commutative DGAs inducing an
isomorphism in homology is
called a  {\em homology isomorphism}. Two commutative DGAs related by a zig-zag
of homology isomorphisms of commutative DGAs are said to be {\em
  quasi-isomorphic}.

A commutative DGA which is quasi-isomorphic to its homology is said to be {\em formal}. A
graded commutative ring $R$ is said to be {\em intrinsically formal}
if every  commutative DGA $\Rtilde$ with $H_*(\Rtilde)\cong R$ is
formal.  We say that $\Rtilde$ is {\em strongly formal} if there is a homology
isomorphism $H_*(\Rtilde)\lra \Rtilde$.  A commutative graded ring
is {\em strongly intrinsically formal} if every commutative DGA with
homology $R$ is strongly formal. 

All of these notions apply similarly to diagrams of commutative DGAs,
and it is our purpose to show that the $\PCf$-diagram $\Ra=\pi^G_*(\Rttop)$ is intrinsically
formal. This is based on the fact that polynomial rings are strongly
intrinsically formal amongst commutative rings. This single fact is
extended in generality in both the algebraic and diagrammatic senses. 

\subsection{Constructing new formal objects from old}
The general form of the results is not surprising, but care is
necessary in their formulation.

\begin{lemma}
\label{lem:algformal}

\begin{enumerate}[(i)]
\item For any commutative ring $k$, the $k$-algebra $k[x_1, \ldots ,
  x_r]$ on even degree generators is strongly intrinsically formal
amongst commutative DG $k$-algebras.  
\item If $R_i$ is  intrinsically formal for all $i$ then $\prod_i R_i$ is
  intrinsically formal. 
\item If $R$ is strongly intrinsically formal and $\cE$ is a multiplicatively
  closed subset of $R$ then $\cEi R$ is intrinsically formal relative
  to $R$ in the sense that if $\Rtilde \lra \Rtilde_{\cEi}$ is a map
  of DGAs inducing $R\lra \cEi R$ in homology, then there exists a
 homology isomorphism $ \Rtilde_{\cE^{-1}}  \to \tilde{R}_{\cEi}'$  such that the diagram   
$$\diagram 
\Rtilde \rto & \Rtilde_{\cE^{-1}} \dto^{\simeq} \\
&\tilde{R}_{\cEi}'\\
R\uuto \rto & \cEi R \ar@{-->}[u]
\enddiagram$$
can be completed by a dotted arrow which is a homology isomorphism.  
\end{enumerate}
\end{lemma}

\begin{proof}
(i) If $H_*(R)=k[x_1, \ldots , x_r]$ then we may pick representative
cycles $\tilde{x}_1, \ldots , \tilde{x}_r$ for $x_1, \ldots, x_r$ in $R$ and then since $k[x_1, \ldots, x_r]$ is free as a commutative ring,
there is a map $k[x_1, \ldots , x_r]\lra R$ taking $x_i$ to
$\tilde{x}_i$, and this induces an isomorphism in homology.  

(ii) Suppose $H_*(\Rtilde)=\prod_i R_i$. First, we replace $\Rtilde$ by a DGA which is actually a
product. Indeed, we may choose cycles $\tilde{e}_i$ representing the
idempotents for the factors. Now form $\Rtilde_i =\Rtilde [1/
\tilde{e}_i]$, so that $H_*(\Rtilde_i)=R_i$. We therefore have a
quasi-isomorphism $\Rtilde \lra \prod_i \Rtilde_i$, and then we may take
the product of the individual zig zags of quasi-isomorphisms
connecting
$\Rtilde_i$ and $R_i$. 

(iii) Since $R$ is strongly intrinsically formal, we have a map $R \to
\Rtilde$; let $\tilde{\cE}$ denote the image of the multiplicatively closed subset
$\cE$ in $\Rtilde$. Then the map $\Rtilde_{\cEi}\lra
\tilde{\cE}^{-1}\Rtilde_{\cEi}$ is a quasi-isomorphism and by the
universal property of localization we may extend $R\lra 
\tilde{\cE}^{-1}\Rtilde_{\cEi}$ to a quasi-isomorphim
$\cEi R\stackrel{\cong}\lra \tilde{\cE}^{-1}\Rtilde_{\cEi}$. 
\end{proof}

When using these facts in diagrams we frequently apply the following
observation.

\begin{lemma}
\label{lem:extendRdiagram}
 Suppose given a partially ordered set $A$, a subset $B\subseteq A$
 with no maps out of it, and a diagram $R: A\lra \DGAs$.
 If we have a $B$-diagram $R': B\lra \DGAs$ and a map $\theta_B:
 R|_B\lra R'$, we may extend $R'$ to an $A$-diagram $\Rhat'$ (taking
 $\Rhat'(a)=R(a) $ if $a\not \in B$) and extend $\theta_B$ to a map $\theta : R\lra \Rhat'$. If
$\theta_B$ is a homology isomorphism, so is $\theta$. \qqed

\end{lemma}

\begin{example} {\em (Extending a diagram of rings along a map    at a vertex $v$.)}
\label{eg:extend}
Suppose $v$ is a vertex in a poset $A$ and we have a map $R(v)\lra
R'(v)$.  We may take $B$ to be the set of
 vertices with a map from $v$, and  define $R'$ on $B$ by taking 
$$R'(b)=R'(v)\tensor_{R(v)}R(b). $$
We obtain a map $R|_B\lra R'$ by  identifying $R|_B(b)$ as $R(v) \otimes_{R(v)} R|_B(b)$ and using the
map $R(v)\lra R'(v)$ at each point.  

Applying Lemma \ref{lem:extendRdiagram} we obtain a map of $A$-diagrams $R\lra \Rhat'$.  This is a pointwise homology isomorphism
provided it is a homology  isomorphism at $v$ and all the rings $R(b)$ are flat over
$R(v)$. 
\end{example}

\subsection{The intrinsic formality of the diagram $R_a$}
We are now prepared to prove the intrinsic formality of the $\PCf$-diagram
$\Ra=\pi^G_*(\Rttop)$ of graded rings. 

The reader may find it helpful to refer to Subsections \ref{subsec:rankone} and \ref{subsec:ranktwo}
where the rank 1 and rank 2 cases are made rather explicit. 

\begin{prop}
\label{prop:RCfformal}
The $\PCf$-diagram $\Ra$ is intrinsically formal, and in
particular $\Rt$ is formal.  
\end{prop}

\begin{proof}
The punctured cube $\PCf$ is a poset (indeed, it is the barycentric
subdivision of the $r$-simplex $\Delta^r$; we may identify each
vertex $v$ of $\PCf$ with the non-empty subset $S(v)=\{ i \st a_i =1\}$ of 
$\{0, \ldots ,r\}$).  The collection of vertices is ordered by the
size of $S(v)$, and we will work in order of increasing size. 

More precisely, we let $\PCf^{(d)}$ denote the $d$-skeleton of the
subdivided $r$-simplex (i.e.,  it contains all vertices $v$ with  $|S(v)|\leq d+1$.

Given a $\PCf$ diagram $\Rtilde$ with homology isomorphic to $\Ra$, we
replace it by an equivalent cofibrant diagram without change in
notation, and then proceed to construct a succession of homology isomorphisms
$$\Rtilde=\Rtilde_0\stackrel{i_0}\lra \Rtilde_1 \stackrel{i_1}\lra  \cdots \stackrel{i_{r-1}}\lra \Rtilde_{r}=\Rtilde$$
of $\PCf$-diagrams of DGAs, where $i_{d-1}:\Rtilde_{d-1}\lra \Rtilde_{d}$ is
constant on $\PCf^{(d-1)}$. As we do this, we construct
 maps
$$\theta_d: \Ra|_{\PCf^{(d)}}\lra \Rtilde_d|_{\PCf^{(d)}}$$
for $d \geq 1$ which are homology isomorphisms on the diagram on which they are
defined. For $d \geq 1$, the map $\theta_d$ extends $i_{d-1}\circ \theta_{d-1}$. 

After $r+1$ steps we obtain a homology isomorphism 
$$R_a=R_a|_{\PCf^{(r+1)}}\lra \Rtilde_{r+1}|_{\PCf^{(r+1)}}
=\Rtilde. $$

To start with, we construct $\Rtilde_1$. Note first that for each
of the $r+1$ vertices  $v$ of $\Delta^r$ the DGA  $\Ra(v)$ is a product of polynomial
rings indexed by $i$ (if the vertex corresponds to connected subgroups of
codimension $c$, then we take a product of all the $\cOcFH$ with $H$
connected of codimension $c$, each of which is a product of the
cohomology rings $H^*(BG/\tH)$ as $\tH$ runs through the subgroups
with identity component $H$. Altogether,  $i$ will run through all subgroups of codimension
$c$, connected or not). 

As in Lemma \ref{lem:algformal} (ii) we construct
DGAs $\Rtilde (v)_i$ with homology $\Ra(v)_i$ and a quasi-isomorphism
$$\Rtilde (v) \lra \prod_i \Rtilde (v)_i.$$ 
Choosing some ordering of
the vertices, we extend $\Rtilde_0$ along each of these
quasi-isomorphisms (as in Example \ref{eg:extend}) in turn to
obtain $\Rtilde_1$. We note that since there are no maps from one
vertex to another, all $r+1$ vertices end up with a product of DGAs.
Now using  Lemma \ref{lem:algformal} (i) at each vertex we obtain a map 
$$\theta_1: R_a|_{\PCf^{(1)}}\lra \Rtilde_1|_{\PCf^{(1)}}. $$

We continue inductively, supposing that after $d$ steps we have
defined $\Rtilde_s$ for $s\leq d$, and 
$$\theta_d: R^a|_{\PCf^{(d)}}\lra \Rtilde_d|_{\PCf^{(d)}}. $$
Once again we will form $\Rtilde_{d+1}$ from $\Rtilde_d$ by extending
the diagram of rings along ring maps at  the $\binom{r+1}{d+1}$ vertices  $v$ with $|S(v)|=d+1$
in turn.  When it comes to the turn of $v$, 
since there are no maps between these vertices, we still have
$\Rtilde_d(v)$ at $v$. This  has homology
$$H_*(\Rtilde_{d}(v))=H_*(\Rtilde (v))=\Ra (v)$$
 and this is obtained from polynomial rings
by alternately taking products and localizing with respect to sets of
Euler classes. Furthermore, we note that the Euler classes concerned come from
the vertices $w$ with $|S(w)|\leq d$, so that $\theta_d$ gives their images in
 the DGAs. We now form a new $\PCf$-diagram of DGAs by extending
$\Rtilde_d(v)$ along the alternate products and localizations
 using Lemma \ref{lem:algformal}.  When we
have extended along all these vertices we have obtained 
$\Rtilde_{d+1}$ from $\Rtilde_d$, and the products and localizations
let us extend $\theta_d$ to $\theta_{d+1}$.
\end{proof}

\subsection{The example of rank 1}
\label{subsec:rankone}
The argument proceeds as follows. We start with the cofibrant $\PCf$-diagram
$\Rtilde$ as in the top row. Extending along the top left hand
vertical we form the second row. The upward maps from the two outer vertices
of $R_a$ on the bottom row
can then be defined. The Euler classes are defined by the image of $R_a(0,1)$, and those are
inverted to form the third row, after which the middle vertical can be
filled in. 
$$\diagram 
\Rtilde \dto  &\Rtilde (0,1) \rto \dto & \Rtilde (1,1) \dto &\Rtilde (1,0) \lto
\dto\\
\Rtilde_1 \dto &\prod_i\Rtilde (0,1)_i \rto \dto_= & \prod_i \Rtilde(0,1)_i
\tensor_{\Rtilde (0,1)}\Rtilde (1,1) \dto &\Rtilde (1,0) \lto \dto^=\\
\Rtilde_2& \prod_i\Rtilde (0,1)_i \rto  & \cEi_G\prod_i \Rtilde(0,1)_i
\tensor_{\Rtilde (0,1)}\Rtilde (1,1)  &\Rtilde (1,0) \lto \\
R_a\uto &\cOcF \rto \uto & \cEi \cOcF \uto &\Q \lto \uto\\
&&&\\
\Rtop& (D\efp)^G\rto &(\siftyV{G}\smb D\efp)^G&(\siftyV{G})^G\lto\\
\Rttop& D\efp\rto &\siftyV{G}\smb D\efp&\siftyV{G}\lto
\enddiagram$$

\subsection{The example of rank 2}
\label{subsec:ranktwo}
It is too typographically complicated to display the full argument in
the way we did for rank 1, but it still seems worth displaying $\Ra$
and $\Rttop$. This lets one see the way that extending along (say) a
map of rings at the top vertex only affects the three other points not
on the bottom face, and then extending along (say) the middle vertex
on the bottom face only affects the central vertex.

\resizebox{0.95\textwidth}{!}{     
$$
\diagram
&&\prod_F\Q[c,d] \drto \dlto&&\\
&\prod_H \cEi_H\prod_F\Q[c,d] \drto && \cEi_G\prod_F\Q[c,d] \dlto &\\
&&\cEi_G \prod_H \cEi_H\prod_F\Q[c,d] &&\\
\prod_H\prod_{\tilde{H}}\Q[c] \uurto \rrto &&
\cEi_G \prod_H\prod_{\tilde{H}}\Q [c] \uto &&\Q \llto \uulto
\enddiagram
$$
}

\vspace{2ex}
The subgroups  $F$ run through finite subgroups, the subgroups $H$ run
through circle subgroups, and the subgroups $\tilde{H}$ run through
subgroups with identity component $H$. The polynomial rings $\Q [c,d]$
are the cohomology rings of $B(G/F)$ (all different but isomorphic),
and the polynomial rings $\Q [c]$ are the cohomology rings of
$B(G/\tilde{H})$. The polynomial ring $\Q$ is the cohomology ring of
$B(G/G)$. 

The above diagram is obtained by taking homotopy groups of the
following diagram $\Rttop$  of ring $G$-spectra. 

\vspace{2ex}

\resizebox{0.98 \textwidth}{!}{   

$$
\diagram
&&D\efp \drto \dlto&&\\
&\prod_H \siftyV{H}\smb D\efp \drto &&      \siftyV{G}\smb D\efp \dlto &\\
&&\siftyV{G}\smb \prod_H \siftyV{H}\smb D\efp &&\\
\prod_H\siftyV{H}\smb D\efhp \uurto \rrto &&
\siftyV{G}\smb \prod_H\siftyV{H}\smb D\efhp \uto &&\siftyV{G} \llto \uulto
\enddiagram
$$
}
\vspace{3ex}

\part{Algebra}

We have now established that the category of $G$-spectra is equivalent
to the cellularization of the category of DG-$\Ra$-modules, where
$\Ra$ is  a $\PCf$-diagram of rings.  It remains to
show this is Quillen equivalent to the category $d\cA (G)$ of DG objects in $\cA
(G)$. 

\section{Modules over  $\Ra$ and the standard model $\protect \cA_c^p(G)$}
\label{sec:AGasmodules}
In this section we make the punctured $(r+1)$-cube of rings $\Ra$
explicit and recall a number of basic structures from \cite{AGs}.

\subsection{Strategy}
\label{subsec:algstrategy}
We will use the algebraic machinery and terminology set up in
\cite{AGs}. As described in Section \ref{sec:standard} above,  $\cA (G)=\cA_c^p(G)$ is a category of
modules over the diagram $\RRc^p$ of rings based on {\bf p}airs of
{\bf c}onnected subgroups. However the topological argument delivers a category of modules over
the diagram  $\Ra $  based on subsets of $[0,r]=\{ 0, 1, \ldots , r\}$ which
are the dimensions of subgroups. For a totally ordered poset like $[0,r]$ there is no
distinction between subsets and  flags: taking a subset of $[0,r]$ with
$s$ elements  in decreasing order, we obtain a {\bf f}lag
$d_0>d_1>\cdots >d_s$. We will make the diagram $\Ra$ explicit in Subsection
\ref{subsec:RaRRd},   and observe that  $\Ra=\RRd^f$ in the notation of
\cite{AGs}. 

It is shown in \cite{AGs} that there is a subcategory $\cA_d^f(G)$ of
$\RRd^f$-modules equivalent to $\cA_c^p(G)$, namely $pqce$-modules,
which is to say that  satisfy a quasi-coherence condition ($qc$) 
an extendedness condition ($e$) and  whose values on vertices are
products ($p$). There is in fact a diagram of categories and adjoint pairs
$$\diagram
\cA_c^p(G)\ar@{=}[d]&\cA_c^f(G)\ar@{=}[d]&\cA_d^f(G)\ar@{=}[d]\\
\mbox{$qce$-$\RRc^p$-mod} \ar@{<->}[r]^{\simeq}_{p,f}&
\mbox{$qce$-$\RRc^f$-mod} \ar@{<->}[r]^{\simeq}\dto<-0.7ex>_i&\mbox{$pqce$-$\RRd^f$-mod} \dto<-0.7ex>
\\
&\rto<0.7ex>^{d_*}\uto<-0.7ex>_{\Gamma_c^f}
\mbox{$\RRc^f$-mod} \rto<0.7ex>^{d_*}\uto<-0.7ex>_{\Gamma_c^f}&\mbox{$\RRd^f$-mod}
\lto<0.7ex>^e\uto<-0.7ex>_{\Gamma_d^f}
\enddiagram$$
The absence of a label on the functor left adjoint to $\Gamma_d^f$ is
intentional: the functor is obtained by following round the other
three sides of the square, and is not the inclusion (the inclusion
does  not preserve sums). In fact, 
there is no need to  give further details of $pqce$ $\RRd^f$-modules
here,  since we will proceed
directly between $\RRd^f$-modules and $qce$-$\RRc^f$-modules. The
relevant result from \cite{AGs} is as follows. 

\begin{prop}
\label{prop:RRdqceRRcpadjunction}
\cite[Subsection 11.C]{AGs}
There is an adjoint pair
$$\adjunction{l}
{\mbox{$qce$-$\RRc^p$-$\mathrm{mod}$}}
{\mbox{$\RRd^f$-$\mathrm{mod}$}}
{\Gamma}$$
where $l=d_*if$ and $\Gamma=p\Gamma_c^fe$. \qqed
\end{prop}

We will briefly describe the functors in Subsection \ref{subsec:functors}  below.

\subsection{The diagram $\Ra$}
\label{subsec:RaRRd}

We will make explicit the  diagram $\Ra =\pi^G_*(\Rttop)$ of homotopy rings of
our $\PCf$-diagram $\Rttop$ of ring spectra as in Definition \ref{defn:PCfdiagrams}. 
It will appear that this is a special case of the machinery of
\cite{AGs}, so that $\Ra =\RRd^f$ in the notation of \cite{AGs}.  

Since $\pi^G_*(\siftyV{H}\smb D\efhp)=\cOcFH$ and since the map $S^0\lra S^V$ induces
multiplication by the Euler class $c(V)$ in $\pi^G_*(D\efp)=\cOcF$, it is straighfroward to read
off from the definition  of $\Rttop$ in Subsection
\ref{subsec:RttoponCif} 
 an  explicit and totally algebraic account. 

At the point $(a_0, \ldots, a_s,0,
\ldots , 0)$ with $a_s=1$, we form a ring from the product 
$$\prod_{\codim (H)=s}\cOcFH$$
by taking retracts and alternating products and localizations. To write this down, 
we recall from Equation~\ref{eqn.indices} the indexing set $I(t,a_t)$ which is a singleton if $a_t=0$ or 
all codimension $t$ connected subgroups otherwise. We also recall that
$\cE_K$ consists of Euler classes of all representations $W$ with
$W^K=0$, and adopt a convention to let us refer to a vacuous
localization in a similar notation: we take $\cE_{K,1}=\cE_{K}$ and 
$\cE_{K,0}=\{ 1\}$. Now we may write 
\begin{multline*}
\Ra (a_0, \ldots , a_s,0, \ldots , 0)=\\
\cEi_{G,a_0}
\prod_{H_1\in I(1,a_1)}  \cEi_{H_1,a_1} 
\prod_{H_2\in I(2,a_2)} \cEi_{H_2,a_2} 
\cdots 
\prod_{H_{s-1}\in I(s-1, a_{s-1})}\cEi_{H_{s-1}, a_{s-1}} \prod_{H_s\in I(s,a_s)}  \cO_{\cF/H_{s}}.
\end{multline*}
To save on the notation required to say we have nested subgroups, we
use the convention that inverting $\cE_H$ is deemed to
annihilate factors corresponding to lower dimensional subgroups $K$
not contained in $H$.

We will say more about what is meant by  inverting Euler classes in
Subsection \ref{subsec:Eulerintext}, but
first it is helpful illustrate the definition in low ranks to show its simplicity. 

\begin{example} {\em (The diagram $\Ra$ in rank 1.)}
 In
rank 1, if the objects of $\PCf $ are layed out as
$$v_1=(01)\lra (11)\lla (10)=v_0$$
the rings are
$$\cOcF \lra \cEi_G\cOcF\lla \cOcFG=\Q$$
\end{example}

\begin{example} {\em (The diagram $\Ra$ in rank 2.)}
\label{eg:Raranktwo}
In rank 2, if the objects are layed out as

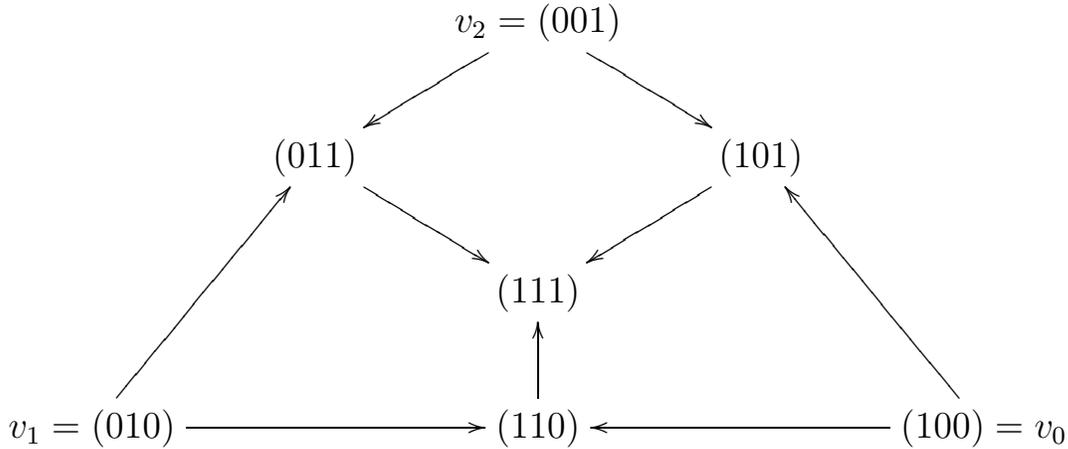
\resizebox{0.9\textwidth}{!}{     
$$
\diagram
&&v_2=(001) \drto \dlto&&\\
&(011)\drto && (101) \dlto &\\
&&(111) &&\\
v_1=(010) \uurto\rrto&&
(110)\uto &&(100)=v_0\llto \uulto
\enddiagram
$$
}

\vspace{2ex}
\noindent
the diagram of rings is 
\vspace{2ex}

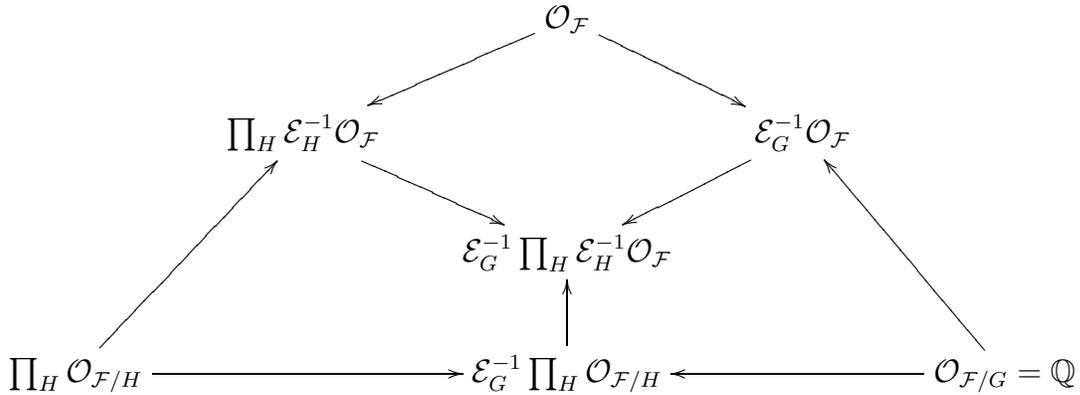
\resizebox{0.9\textwidth}{!}{  
$$
\diagram
&&\cOcF \drto \dlto&&\\
&\prod_H \cEi_H\cOcF \drto && \cEi_G\cOcF \dlto &\\
&&\cEi_G \prod_H \cEi_H\cOcF &&\\
\prod_H\cOcFH \uurto \rrto &&
\cEi_G \prod_H\cOcFH\uto &&\cOcFG = \Q \llto \uulto
\enddiagram
$$
}
\end{example}

\begin{example} {\em (The diagram $\Ra$ in rank 3.)}
The diagram in rank 3 is that of a subdivided 3-simplex, and a little
too complicated to display in print.  However we note that a new phenomenon
occurs in rank 3 since not every circle subgroup is contained in every 
2-torus subgroup (in lower ranks, containment of connected subgroups 
was determined by dimension). This means that at points of the form
$(a_0 11 a_3)$, we have
$$\Ra = \cdots  \prod_H\cEi_H \prod_K \cdots .$$
where $H$ is of codimension 1 and $K$ of codimension $2$. In view of
our  convention about inverting $\cE_H$, the second product is in fact over 
circle subgroups $K$ contained in $H$ (and not over all circle subgroups).

\end{example}
\subsection{Internal and external Euler classes} 
\label{subsec:Eulerintext}
The $G$-equivariant  homotopy of $\siftyV{H}\sm X$ is always the
$G/H$-equivariant homotopy of the
geometric fixed point spectrum $\Phi^HX$. Sometimes this is calculated
from geometric knowledge of $X$, but if $X$ has Thom
isomorphisms for representations $V$ with $V^H=0$ it can also be
calculated from $\pi^G_*(X)$ by inverting Euler classes when they are
defined.  However, some slightly extended use of the algebraic
notation for inverting Euler classes requires some explanation. 

The issue first arises  at $(110)$ in rank
2. A brief explanation of this special case will make plain the
general meaning.

The notation suggests we are inverting $G$-equivariant Euler
classes (elements of $\cOcF$) on something (viz $\prod_H \cOcFH$), but the object in
question  is not an $\cOcF$-module. Considering the geometry of the
situation we see  that what is really
happening is passage to a direct limit along maps $S^{W_1}\lra
S^{W_2}$ coming from inclusions $W_1\subseteq W_2$ with 
$W_1^G=W_2^G=0$. Since the spheres are finite complexes 
this passes inside the product. To see what happens on the $H$th
factor we  write $W=W^H\oplus W'$, and note that
$S^{W} \sm \siftyV{H}\simeq S^{W^H} \sm \siftyV{H}$. Thus when we
write  $\cEi_G \prod_H \cOcFH$, this means a direct limit
 over multiplication by  the product elements $\prod_H
 c(W_2^H/W_1^H)$, which is  the Euler class of the inclusion $W_1^H\lra
 W_2^H$, as an element of $\cOcFH$.

Note that this discussion also explains why the $\siftyV{G}$ does not lead to any algebraic
inversion at $(100)$.

\subsection{Structure maps for rings}

Next we  describe  the structure maps in $\Ra$
more precisely. Once again, the main complication is notational. 

If we have an inclusion $\ist : \sigma \lra \tau$  of subsets of $\{
0, \ldots, r\}$ then we have a structure map 
$$\Ra (\ist)  :\Ra (\sigma )\lra \Ra (\tau). $$
Suppose $s$ is the largest element of $\sigma$. 
We start by  describing the case when $\tau$ has exactly one more element
than $\sigma$, say $\tau  =\sigma\cup \{t\}$. There are two cases.

{\bf Case 1: $t>s$.}
In this case  $t$ is the last non-zero term in $\tau$ and we may
concentrate on the  contribution of the last two non-trivial terms, namely the $s$th and
$t$th. Thus we must describe 
$$j_s^t: \prod_{H_s \in I(s,1)}\cO_{\cF/H_s}\lra 
\prod_{H_s \in I(s,1)} \cEi_{H_s} \prod_{H_t\in I(t,1)}\cO_{\cF/H_t}$$
in the sense that the map is obtained from this by applying
alternating products and localizations for the $0$th to the $(s-1)$st
terms. Now $j_s^t$ is itself a
product over $I(s,1)$ of terms given as the composite
$$ \cO_{\cF/H_s}\lra  \prod_{H_t\in
I(t,1)}\cO_{\cF/H_t}\lra \cEi_{H_s} \prod_{H_t\in
I(t,1)}\cO_{\cF/H_t}. $$
The first map has components given by inflations for $H_s\supseteq
H_t$ and the second is localization. 

{\bf Case 2: $t<s$.}
In this case  $s$ is the last non-zero term in both $\sigma $ and
$\tau$ and the only change is to replace the expression 
$\prod_{H_t\in I(t,0)}\cEi_{H_t,0} $ (which actually means take the
  product over a singleton of a localization doing nothing!) with 
$\prod_{H_t\in I(t,1)}\cEi_{H_t} $, and here a diagonal map is used. 

More precisely if 
\begin{multline*}
\Ra (a_{t+1}, \ldots , a_s,0, \ldots , 0)
=\\
\prod_{H_{t+1}\in I(t+1,a_{t+1})} \cEi_{H_{t+1},a_{t+1}} 
\prod_{H_{t+2}\in I(t+2,a_{t+2})} \cEi_{H_{t+2},a_{t+2}} 
\cdots 
\prod_{H_{s-1}\in I(s-1, a_{s-1}}\cEi_{H_{s-1},a_{s-1}} \prod_{H_s\in I(s,a_s)}    \cO_{\cF/H_s}
\end{multline*}
we take the map into the product whose components are localizations
$$\{ l_{i_t}\}: \Ra (a_{t+1}, \ldots , a_s,0, \ldots , 0) \stackrel{}\lra
\prod_{H_t\in I(t,1)} \cEi_{H_t}\Ra (a_{t+1}, \ldots , a_s,0, \ldots ,
0)$$
and then apply alternate products and localizations to incorporate the terms from the $0$th to the $(t-1)$st. 

When $\tau$ has more than one extra vertex than $\sigma$ the map $\Ra
(i_{\sigma}^{\tau})$ is the composite of the maps adding one vertex at
a time. It is apparent from the description above that the order in
which this is done makes no difference. 

\subsection{The algebraic diagram $\Ra$ is the diagram  $\RRd^f$ from \cite{AGs}}
We briefly recall the framework of \cite{AGs}, so that we may observe
that $\Ra$ is precisely the diagram of rings appearing there as $\RRd^f$.

The diagram $\RRc$ is the contravariant functor on the poset
$\connsubG$ of connected subgroups
of $G$ with value $\cOcFK$ at $K$, and with inflation maps between
them. The dimension function $d:\connsubG\lra [0,r]$ gives rise to
a dimension function on the posets of flags. In \cite{AGs} it is
explained that such a function induces a map 
$d_!^e$ collecting together the subgroups of the same
dimension, and extends to flags using localizations and products. This 
specializes precisely to the description of  $\Ra$, so that $\RRd^f=\Ra$.

\subsection{Description of the functors}
\label{subsec:functors}
We now briefly recall from \cite{AGs} the functors appearing in the
diagram from Subsection \ref{subsec:algstrategy} above. 

The left hand horizontal translates between indexing over pairs and
indexing over flags. For $qce$-modules the value of a module on a flag
only depends on the largest and smallest subgroup in the flag, so this
translation is nugatory; the letter $p$ is for the translation to
pairs and the letter $f$ for the translation to  flags.

The vertical $i$ is the inclusion of $qce$-modules in all
$\RRc^f$-modules, and the functor $\Gamma_c^f$ is the right adjoint to
$i$ constructed in \cite[Section 11]{AGs} following the pattern of  \cite{tnq2};
we will not need to use an explicit  construction.  

The functor $e$ is obtained by taking idempotent pieces. Indeed, if
$M$ is an $\RRd^f$-module and $F=(K_0\supset K_1\supset \cdots \supset
K_s)$ is a flag of connected subgroups with dimension 
$dF=(d_0> d_1>\cdots >d_s)$ there is an idempotent $e_F \in
\RRd^f(dF)$ picking out the flag $F$; we take $(eM)(F)=e_F(M(dF))$
(see \cite[Section 6]{AGs} for further details). 

The functor $d_*$ is left adjoint to $e$. The natural idea is to 
take direct sums: if $N$ is an $\RRc^f$-module then 
$(d_*N)(d) =\bigoplus_{dF=d}N(F)$. However this is
not compatible with  structure maps and one must take the submodule of the
product it generates. There is a little work to be done to check this
makes sense, and the  construction is described  in detail in \cite[Section 6]{AGs}.

\section{Model structures and equivalences on the algebraic categories}
\label{sec:algebraicmodels}

The output of Parts 1-3 is a Quillen equivalence between 
the category of rational $G$-spectra and an algebraic 
category $\cellRamod$, the cellularization of 
the category of modules over the diagram $\Ra$ of rings. The purpose of this section and the  next
is to simplify the model by avoiding the need for cellularization: we
show that the cellularization of the category of $\Ra$-modules is
Quillen equivalent to  the smaller category of objects in the category of qce-$\RRc$-modules,
 $\cA_c^p (G)$.

This section gives a model structure on $d\cA (G)$ and recalls some
facts about the torsion functor relating it to the appropriate
category of $\RRc^p$-modules.  
\subsection{Two examples}
Before turning to general results we give two examples of this
phenomenon in a simpler context: the first for free spectra in
general, and the second for semifree spectra for the circle group. 

Algebraically, the first example is for modules over  a single polynomial ring. 

\begin{example} {\em (Free $G$-spectra and torsion modules over a polynomial ring.)}
If $G$ is a connected compact Lie group, the category of free rational
$G$-spectra is Quillen equivalent to the category of torsion modules
over the polynomial ring $H^*(BG)$ \cite{gfreeq}.

The topology gives a Quillen equivalence with the model category
$\cellHBGmodp$:  the category $\HBGmodp$ of DG-modules over $H^*(BG)$ 
with the algebraically projective model structure cellularized with respect to 
the residue field $\Q$. This in turn is Quillen equivalent to the model category
$\cellHBGmodi$, the category $\HBGmodi$ of DG-modules over $H^*(BG)$ 
with the algebraically injective model structure cellularized with respect to 
the residue field $\Q$. The model structure $\HBGmodp$ is well known:
it is the cofibrantly generated structure  right-lifted from vector
spaces. Similarly, the model structure $\HBGmodi$ may be constructed by left-lifting from vector
spaces using \cite[Theorem 2.2.3]{HKRS}.

Finally, if $\fm$ is the ideal of positive codegree elements in
$H^*(BG)$, we consider the adjunction 
$$\adjunction
{i}
{\torsHBGmod}
{\HBGmodi}
{\Gamma_{\fm}}$$
where $\Gamma_{\fm}$ is the $\fm$-power torsion functor. The category
of torsion modules has an injective  model structure (weak
equivalences are homology isomorphisms and cofibrations are
monomorphisms). This can be constructed by left lifting from the injective model
structure on all modules using \cite[Theorem 2.2.1]{HKRS}.  Accordingly,  $i$ preserves
cofibrations and acyclic cofibrations and the adjunction is a 
Quillen adjunction.  Finally, $\Q$ is a small generator
of the torsion modules, so the Cellularization Principle
\cite{cellprin}  shows this
induces a Quillen equivalence
$$\torsHBGmod\simeq \cellHBGmodi. $$

This example is directly relevant to the algebraic model $\cA (G)$ for
a torus $G$.  Indeed, if we consider objects of $\cA (G)$ which are
concentrated at the connected subgroup $1$, and for which there is 
no contribution from other finite subgroups, the quasicoherence
condition on $\RRc$-modules in
$\cA (G)$ implies that objects concentrated at the subgroup $1$  are
precisely the torsion $H^*(BG)$-modules. 
\end{example}

Algebraically, the second example works with a rather small diagram of rings, with
each of the rings Noetherian.

\begin{example} {\em (Semifree $\T$-spectra.)}
For the circle group $\T$, our models  are over a punctured square of
rings. If we simplify the category by restricting attention to
semifree spectra, the rings that occur are much smaller
and we can see the issues introduced by diagrams without having the infinite
number of subgroups to complicate matters. 

The diagram of rings for semifree $\T$-spectra is 
$$\Ra =
\left( 
\begin{array}{ccc}
&&R^v\\
&&\downarrow\\
R^n&\lra &R^t
\end{array}
\right)
=
\left( 
\begin{array}{ccc}
&&\Q \\
&&\downarrow\\
\Q [c]&\lra &\Q [c,c^{-1}]
\end{array}
\right)
$$
An $\Ra$ module $M$ consists of a diagram
$$
M =
\left( 
\begin{array}{ccc}
&&M^v\\
&&\downarrow\\
M^n&\lra &M^t
\end{array}
\right)
=
\left( 
\begin{array}{ccc}
&&V\\
&&\downarrow\\
N&\lra &P
\end{array}
\right)
$$
where $V$ is a $\Q$-module, $N$ is a $\Q [c]$-module and $P$ is a
$\Q[c,c^{-1}]$-module. 

There are four relevevant model categories. To start with, on each of
the three objectwise module categories we can choose either
the algebraically projective model structure or the algebraically 
injective model structure. We need to make the same choice at each 
vertex so that the maps in the diagram respect the model structures. Secondly, having made that choice,
we may choose either the diagram theoretically projective or injective
model. Since the diagrams are both direct and inverse,  the results of
\cite{diagrammodcats} show these models all exist, and it is clear 
there are Quillen equivalences between either of the two binary choices by using the 
identity functors.  In fact, we only need three of the four possibilities; a
diagram-projective, algebraically-injective model structure does
not appear.

Having made a choice, we cellularize with respect to the two
modules corresponding to basic geometric generators
$$
\bS= \Ra =
\left( 
\begin{array}{ccc}
&&\Q\\
&&\downarrow\\
\Q [c]&\lra &\Q [c,c^{-1}]
\end{array}
\right)
\mbox{ and }
G_+=
\left( 
\begin{array}{ccc}
&&0\\
&&\downarrow\\
\Q&\lra &0
\end{array}
\right)
$$
By \cite[Corollary 2.8]{cellprin}, cellularization preserves the Quillen equivalences mentioned above. 

Finally, for $\qcemodcat{R}$, the underlying category consists
of  quasi-coherent extended modules. The quasi-coherence 
condition is that  the horizontal map is localization in the sense that
$$M^t \iso M^n[1/c]. $$
The extendedness is the condition that the vertical is induction
in the sense that 
$$M^t \iso \Q [c,c^{-1}]\tensor V. $$
The inclusion of this category of modules has a right adjoint, and we
may argue as in the previous example.  We will give the category of
$qce$-modules a model structure so that 
it is Quillen equivalent to the cellularization of the doubly injective model structure. 
\end{example}

\subsection{Construction of  model structures.}
\label{subsec:algmodelstructures}

In the remainder of this section we turn to the full $\PCf$-diagram
$\Ra$ of rings. We saw in Section \ref{sec:AGasmodules} that
$\Ra=\RRd^f$ in the notation of \cite{AGs}, and we  outline here the proof that the cellularization of the
doubly projective model category  of $\Ra$-modules is equivalent to the category of DG
qce-$\RRc^p$-modules $\cA_c^p (G)$ as in Section \ref{sec:AGasmodules}. 

We begin by  formally introducing the algebraic model structures we
use. 

These are model structures on diagrams of modules over diagrams of
 DGAs.  For each individual DGA there is an algebraically projective
model structure \cite[Theorem 4.1]{ss1}, which is the cofibrantly
generated model structure  lifted along the right adjoint forgetful
functor to $\Q$-modules
in the usual way; the proof may be obtained by adapting
\cite[Section 2.3]{hovey-model}.    The adaption of the construction
of the injective model for  an individual DGA  uses a little more
algebra, so instead we construct the injective model structure by
lifting along the left adjoint forgetful functor to $\Q$-modules using
\cite[Theorem 2.2.3]{HKRS}.

Making a choice of algebraically projective or injective model
structures at all points in the diagram we may then seek to define
a diagram-theoretically projective model structure (in which weak
equivalences and fibrations are given pointwise)  or
a diagram-theoretically injective model structure (in which weak
equivalences and cofibrations are given pointwise).  Since the finite diagram shapes we are interested in here are both direct and inverse, both diagram-projective and diagram-injective model structures exist by~\cite[Proposition 3.1]{diagrammodcats} for either of the algebraic choices (made consistently throughout the diagram).  Only three of the four choices appear in our work here, the doubly-projective 
case (which also follows from~\cite[6.1]{ss-mon}), the doubly-injective case, and the diagram-injective,  algebraically-projective 
case.

\subsection{A model structure on torsion modules}
\label{subsec:torsionmodel}

We consider the category $\cA_c^p (G)$ of $qce$-$\RRc^p$-modules and 
show the associated category of DG objects  admits a model structure
with quasi-isomorphisms as the  weak equivalences.

\begin{prop}
The category $d\cA_c^p (G)$ of DG $qce$-$\RRc^p$-modules admits a
model structure with weak equivalences the  quasi-isomorphisms
and cofibrations the monomorphisms at each object. 
The fibrant objects are injective if the differential
is forgotten, and fibrations are surjective maps with fibrant kernel. 
\end{prop}

\begin{proof}
We use the method of \cite[Appendix B]{s1q}, where it is shown that
one can often construct a model structure using a type of fibrant
generation argument provided one has a suitable finiteness of
injective dimension.

We have an abelian category $\cA =\cA_c^p (G)$ and we aim to put a model 
structure on the category of DG objects of $\cA$. 
We will specify a set $\cBI$ of {\em basic injectives}
containing sufficiently many injectives (i.e., any object of 
$\cA$ embeds in a product of basic injectives).  An injective $I$ 
is viewed as an object $K(I)$ of $d\cA$ with zero differential. The
notation is chosen to suggest an Eilenberg-Mac~Lane object (or
cosphere). Next, we let $P(I)=\fibre (1: K(I)\lra K(I))$, with the
notation chosen to suggest a path object (or
codisc).  The set $\mcL$ of generating fibrations consists of the maps
$P(I)\lra K(I)$ for $I$ in $\cBI$. 
The set $\cM$ of generating acyclic fibrations consists of the 
maps $P(I)\lra 0$ for $I$ in $\cBI$. 

We now take $\we$ to consist of quasi-isomorphisms, 
$\cof$ to be the maps with the left lifting property with respect to ${\cM}$ and $\fib$ to
be the maps with the right lifting property with respect to ${(\we \cap \cof)}$, and prove this 
forms the model structure of the lemma. We outline the four main steps
and then turn to proving they can be completed in our current
situation. 

Step 1:  Show that $\cof$ consists of objectwise
monomorphisms. 

Step 2: Show that for any $X$ there is an objectwise monomorphism 
$\alpha: X\lra P(I)$ for some injective $I$. 

Step 3: Show that the maps $P(I)\lra K(I)$ and $P(I)\lra 0$ in $\mcL$ and $\cM$
respectively are in $\fib$. 

Note that since any injective is a retract of a product of basic
injectives, it follows that $P(I)\lra K(I)$ and $P(I)\lra 0$  are fibrations for any
injective $I$.   Since we have chosen $\cBI$ to contain enough injectives,
one of the factorization axioms follows immediately, since we may
factorize $f: X\lra Y$ as 
$$X\stackrel{\{ f,\alpha\} }\lra Y\times P(I)\stackrel{\simeq}\lra
Y, $$
with  $\alpha$ as in Step 2. 

Step 4: Prove the second factorization axiom using only fibrations
formed from those named in Step 3. 

More precisely, given $f: X \lra Y$, we form a factorization $X\lra
X'\lra Y$ with $X\lra X'$ a quasi-isomorphism and $X'\lra Y$ a
fibration formed by iterated pullbacks of fibrations $P(I)\lra
K(I)$. This is precisely dual to the usual argument attaching cells to
make a map of spaces into a weak equivalence, but because the
dual of the small object argument does not apply, we use the
finiteness of injective dimension of $\cA$ to see that only finitely many
steps are involved in the process (details below). The map $X\lra X'$ can be made into a
cofibration by taking the product of $X'$ with a suitable $P(I)$  as in the proof of the first
factorization argument. It follows using the defining right lifting property that an
arbitrary fibration is a retract of one formed by iterated pullbacks of fibrations $P(I)\lra
K(I)$ or $P(I) \lra 0$.

It remains to verify the four steps can be completed. 
We
follow the pattern from the case of the circle group in~\cite[Appendix B]{s1q}. We note that for 
each connected subgroup $H$ of $G$ there is an evaluation functor 
$$ev_H: \mbox{$\Ra$-modules} \lra \mbox{$\cOcFH$-modules}$$
with right adjoint $f_H$. In particular, if $N$ is a torsion
module, $f_H(N)$ lies in $\cA (G)$ and 
$$\Hom_{\cA (G)}(X,f_H(N))=\Hom_{\cOcFH}(\phi^HX, N). $$
We take the basic injectives to be those of the form 
$$\bI_{\Ht}=f_H(H_*(BG/\Ht))$$
where $\Ht$ is any subgroup with identity component $H$. 
It is shown in \cite[2.20]{tnq1}  that this set contains sufficiently many injectives. 

The following elementary lemma lets us reduce verifications to
statements about modules with zero differential over a (single object)
ring. We write $\Hom$ for the differential graded object of graded
$\cA$-morphisms  and let $DG-\Hom$ denote the group of morphisms
commuting with the differential. The differential on $\Hom$ is defined
so that the DG-morphisms are the $0$-cycles in $\Hom$.

\begin{lemma}
\label{lem:hom}
\begin{enumerate}[(i)]
\item $\Hom_{\cA} (X, K(f_H(M)))=\Hom_{\cOcFH} (\phi^H X,M)$ 
\item $DG-\Hom (X, K(f_H(M)))=\Hom (\phi^H X/d\phi^HX,M)$
\item $DG-\Hom (X, P(f_H(M)))=\Hom (\Sigma \phi^H X,M)$
\end{enumerate}
\end{lemma}

It follows from this lemma by the left lifting property that $\cof$ consists of objectwise
monomorphisms (Step 1), see also~\cite[Lemma B.2]{s1q}, and that we may find a monomorphism $\alpha$ 
in the first factorization argument (Step 2): for this we first embed all $\phi^HX$ in
some injective $I_H(X)$  ignoring the differential and use Lemma
\ref{lem:hom}(iii) to obtain a map to $P(f_H(I_H(X)))$, and take the
product of these over all $H$  to obtain $P(I)$.

This lemma also makes it straightforward to 
verify that objects of $\mcL$ and $\cM$ are fibrations. The case of $P(I)\lra 0$ 
is simply the defining property of an injective. The problem 
$$\diagram
A\dto^i \rto^{\alpha} & P(f_H(\bI_{\Ht}))\dto\\
B \rto^{\beta}  \urdashed|>\tip^h&K(f_H(\bI_{\Ht})))
\enddiagram$$
is equivalent to 
$$\diagram
\Sigma^{-1}\phi^HA /d\phi^HA \rrto^d\ddto^i&&\phi^H A\ddto^i \dlto^{\overline{\alpha}}\\
&\Sigma^{-1}\bI_{\Ht}&\\
\Sigma^{-1}\phi^HB /d\phi^HB \urto^{\tilde{\beta}} \rrto^d&&\phi^H B \uldashed|>\tip^{\overline{h}}
 \enddiagram$$

To find a solution we use  a standard diagram chase. We first use the fact that $i$ is a homology
epimorphism to deduce that $\tilde{\beta}$ vanishes on cycles and the
fact that it is a homology monomorphism to see that this means that
$\tilde{h}$ is consistently  defined on $\phi^H A+d\phi^H B$. Finally,
we use the defining property of injectives to extend it over $\phi^HB$.

This leaves Step 4. 
Here we start by forming an exact sequence 
$$0\lra H_*(X)\lra H_*(Y)\oplus  I_0 \lra I_1 \lra \cdots \lra I_N \lra 0$$
in $\cAs (G)$, where the $I_s$ are injective. The finite injective
dimension of $\cAs (G)$ ensures such an exact  sequence exists. 
 We now realize this by a tower of fibrations
$$Y\lla X_0 \lla \cdots \lla X_N=X',  $$
together with lifts 
$$\diagram 
&\dto\\
                            &X_1\dto \\
                            &X_0\dto \\
X\rto^f \urto^{\hspace*{20ex}f_0} \uurto^{f_1} & Y. 
\enddiagram$$
We take $X_0=Y\oplus K(I_0)$, and the subsequent
objects and maps are constructed using the diagram
$$\diagram
X \drrto \drto \ddrto &&\\
&X_{s}\dto \rto &P(\Sigma^{-s}I_s)\dto\\
&X_{s-1} \rto&K(\Sigma^{-s}I_s)
\enddiagram$$
where the lower horizontal is chosen to realize the inclusion of 
$\im (I_{s-1} \lra I_s)$ in $I_s$. 
The map $f_N: X\lra X_N$ is necessarily a quasi-isomorphism, and can be
made into a monomorphism by taking a product with a suitable $P(I)$. 

This completes the sketch proof of the proposition.
\end{proof}

\subsection{Equivalence of models of torsion modules}
\label{subsec:itorsequiv}

We recall from Subsection \ref{subsec:RaRRd} that $\Ra=\RRd^f$, and work with the adjunction of Proposition 
\ref{prop:RRdqceRRcpadjunction}.

\begin{prop}
The adjunction
$$\adjunction{l}{\cA_c^p (G)=\mbox{$qce$-$\RRc^p$-$\mathrm{mod}$}}{\Ramod_{ii}}{\Gamma}$$
is a Quillen adjunction, where the subscript $ii$ refers to the   use
of  the doubly injective model structure on $R_a$-modules 
(i.e., injective in both the module theoretic and diagram theoretic 
sense) and  where $l=d_*if$ and $\Gamma=p\Gamma_c^fe$.

Cellularizing with respect to the images of the topological cells
induces a Quillen equivalence
$$\cA(G)=\cA_c^p(G)=\mbox{$qce$-$\RRc^p$-modules} \simeq \cellmodcat{\Ra}_{ii}. $$
\end{prop}

\begin{proof}
First we need to check that $l=d_*if$ preserves cofibrations and acyclic
cofibrations so that we have a Quillen adjunction. 

The cofibrations in $\cA_c^p(G)$ are the monomorphisms, which are the
objectwise monomorphisms. Similarly, the cofibrations in an algebraically injective model structure are precisely the
monomorphisms. The cofibrations in the doubly injective $\Ra$-module category are 
precisely the morphisms which are objectwise cofibrations, namely the
objectwise monomorphisms.  It is obvious that $f$ and $i
$ preserve monomorphisms. It is also clear that the functor $d_!$
(given by taking the product of the values) preserves
monomorphisms. Since $d_*N \subseteq d_!N$, it follows that  $d_*$ also
preserves momomorphisms. 

The weak equivalences in both categories are objectwise quasi-isomorphisms, and we will show $l$
preserves all homology isomorphisms. Since $l$ is defined at the level
of abelian categories, it takes mapping cones to mapping cones. 
It therefore suffices to show that if $X$ is a $qce$-module  with $H_*(X)=0$
then $H_*(lX)=0$. For this we use a filtration described in
\cite[Section 6]{AGs} (the map $d: \Sigma_c\lra [0,r]$ and the diagram
$\RRc^f$ take the roles of the map $\pi :\Sigma\lra \overline{\Sigma}$  and the
ring $R^f$). To avoid clutter, 
we will omit the notation $if$  since $ifX$ takes the same values as
$X$ on pairs. 

For each flag $f=(f_0>\cdots >f_s)$ of dimensions we consider
the value $(d_*X)(f)$ at $f$.  Inside this we have the generating
submodules $M_{f_i}$ for $i=0,1, \ldots , s$ (this is the submodule generated by
the image of  $(d_*X)(f_i) =\bigoplus_{\dim K=f_i}X(K)$). There is an
associated Mayer-Vietoris spectral sequence for these, showing that it 
suffices to show that for each face $e=(e_0>e_1>\cdots >e_t)\subset  (f_0>f_1>\cdots >
f_s)=f$ the intersection 
$$M_e=\bigcap_{j} M_{e_j}$$
is acyclic. A combinatorial lemma \cite[Lemma 6.7]{AGs} shows that $M_e$ is generated by the image
of the diagonals including $e$ in $f$. Furthermore
$$M_e =\sum_{\dim E=e}M_E =\bigoplus_{\dim E=e}M_E$$
so it suffices to show that $M_E$ is acyclic.

Now consider the diagram 
$$\diagram
\RRd (f)\tensor_{\RRd (e)} X(E) \rto \dto & (d_!M)(f)\dto^{\cong}\\
d_!e[\RRd (f)\tensor_{\RRd (e)} X(E)] \rto  & (d_!M)(f).
\enddiagram$$
in which $M_E$ is the image of the top horizontal. We argue that the
top horizontal is in fact a monomorphism, and it then follows since
($\RRd (f)$ is flat over $\RRd (e)$) that $M_E$ is acyclic.  

In fact the  bottom horizontal is an isomorphism since $X$ is $qce$;
indeed the $F$th idempotent piece is the map  $\RRd (F)\tensor_{\RRd
  (E)}X(E)\lra X(F)$. The left hand vertical is a monomorphism since 
it can be viewed as a composite
$$\RRd (f)\tensor_{\RRd (e)} X(E) \lra \RRd (f)\tensor_{\RRd (e)}
\prod X(E) \lra 
d_!e \RRd (f)\tensor_{\RRd (e)} X(E); $$
the first is a monomorphism since the diagonal is and $\RRd (f)$ is
flat over $\RRd (e)$, and the second map is an isomorphism. 
It follows that  the top horizontal is a monomorphism as required.

This shows that we have a Quillen pair, and we now  cellularize with respect to the images of the cells
$G/H_+$.   By the Cellularization Principle \cite{cellprin}   
this induces a Quillen equivalence of cellularizations since the cells are small and lie
in $\cA_c^p (G)$. 

Finally, it remains to check that cellularization is the identity on $\cA_c^p (G)$.  
This will be completed by
Theorem \ref{prop:cellequivisequiv} which states that cellular equivalences for qce
modules are precisely the quasi-isomorphisms.   Thus, 
$$\cA_c^p(G)=\mbox{$qce$-$\RRc^p$-modules} =\cell \mbox{$qce$-$\RRc^p$-modules} .\qedhere$$
\end{proof}

\begin{remark}
We would like to upgrade the equivalence to being monoidal, but we
note that although the category $\cA (G)$ is monoidal,  the injective
model on  $d\cA (G)$ we have described is not a  monoidal model
structure.  The first step will be to extend
Barnes's dualizable model structures from the rank 1 case to the
arbitrary case, using \cite{tnq2}.  
\end{remark}

\section{Cellular equivalences in $\protect D(\cA (G))$}
\label{sec:algcells}

We aim to show that cellularization has no effect on $d\cA (G)$
(equipped with the model structure described in Section
\ref{sec:algebraicmodels}).

\subsection{The two notions of equivalence}
Recall that a map $f:X \lra Y$ in $d\cA (G)$ is a weak equivalence if 
it is a homology isomorphism (i.e., if $f_*: H_*(X)\lra H_*(Y)$ is an
isomorphism in $\cA (G)$). This means that it is an isomorphism when
evaluated at each connected subgroup $K$. 
 
The map $f: X\lra Y$ is a {\em cellular  equivalence} if the $\Hom (A,X)\lra \Hom (A,Y)$ is a homology
isomorphism for all (cofibrant) cells $A$ (i.e., $[A,X]_*\lra
[A,Y]_*$ is an isomorphism for all $A$). We note that for each
$A$ this just involves a single graded vector space. 

\begin{thm}
\label{prop:cellequivisequiv}
The triangulated category $D(\cA (G))$ is generated by the cells $G/K_+$. 
Accordingly, a cellular equivalence of objects of $d\cA(G)$ is a homology
isomorphism. 
\end{thm}

\begin{remark}
Although the idea of the proof seems rather simple, organizing the
implementation requires some delicacy. We tried several approaches,
hoping to minimize the verifications, but in the end all seemed to use
very  similar ingredients: Koszul models for cells, and the associated 
 apparatus of torsion and completion,  a filtration by dimension of isotropy groups 
and the objects $f_K(M)$ (where $f_K$ is right adjoint to evaluation
at $K$). 
\end{remark}

By the use of mapping cones, it suffices to show that 
if an object $X$ is cellularly trivial then $H_*(X)=0$.
This also proves the statement about generation, since the cells are
small, and for any $Y$ we may use the usual process of cellular
approximation to construct a cellular object $\mathrm{cell} (Y)$ and a map
$\mathrm{cell}(Y)\lra Y$ which is a cellular equivalence.

Suppose $X$ is cellularly trivial. We will argue by induction on the
codimension of $K$ that $H_*(X)(K)=0$. Suppose then that $\codim (K)=c$
and that we have already proved that $H_*(X)(H)=0$ if $\codim
(H)<c$. This is certainly true if $c=0$, so the induction
starts. Since there are no infinite decreasing chains of subgroups
this suffices. We return to the inductive step in Subsection
\ref{subsec:indstep} after some preparation.

\subsection{Motivation for the proof }
\label{subsec:motivation}
To guide us, and to recall some  standard notation, we consider the
derived category $D(tors-R)$, where $R=k[x_1, \ldots , x_r]$ is a
polynomial  ring over a field $k$.  The corresponding claim is that if $[k, M]_*=0$
then $M\simeq 0$ (or equivalently that $k$ generates the category). 
One proof is as follows.  

First we recall some standard constructions. The Koszul complex for an element $x$ is defined by 
$\Kos (x)=\fibre (x: R\lra R)$, and the stable Koszul complex is
defined by $\Kos^{\infty}(x)=\fibre (R\lra R[1/x])$. It is easy to see
that 
$$\Kos^{\infty}(x)=\colim_s\Kos (x^s). $$
For a sequence of elements the Koszul and stable Koszul complex are
obtained by tensoring those of the terms together. The stable Koszul
complex $\Kos (y_1, \cdots , y_r)$ only depends on the radical of the 
ideal $(y_1, \ldots , y_r)$ and we write
$$\Gamma (M)=\Kos^{\infty}(x_1, \ldots , x_r)\tensor M, $$
and this is the $k$-cellularization of $M$. 

We may now proceed with the proof. 

Step 1. $k$ is self-dual up to suspension. Indeed, it is equivalent to the
Koszul complex for the generators $x_1, \ldots, x_r$. 

It follows from Step 1 that  if $[k,M]_*=0$ then  $k\tensor M\simeq 0$. 

Step 2. It then follows formally that  $\Gamma R \tensor M\simeq 0$, where $\Gamma
R$ is the $k$-cellularization of $R$.

Step 3. From the cofibre sequence $\Gamma R \lra R \lra \check{C}R$ we
deduce $M\simeq \check{C}R\tensor M$. 

Step 4. Since $M$ is torsion $\check{C}R\tensor M\simeq 0$. Indeed,
$\check{C}R$ has a finite filtration with subquotients $R[1/x_{T}]$
where $x_T$ is a (non-empty) product of variables $x_i$, and 
$$R[1/x_T]\tensor M=M[1/x_T]\simeq 0$$
since $M$ is torsion. 

\subsection{Cells as Koszul complexes }

We explain how to view the cells $G/K_+$ as Koszul complexes. 

First, if $H$ is a codimension 1 subgroup we choose a one dimensional
representation $\alpha =\alpha (H)$ so that  $H=\ker (\alpha)$.  The
cofibre sequence 
$$G/H_+\simeq S(\alpha )_+ \lra S^0 \lra S^{\alpha}$$
suggests that $G/H_+$  is equivalent to the Koszul
complex of $e(\alpha)$. More precisely, this follows from the form of
the models of $S^0$ and $S^{\alpha}=\Sigma^{\alpha}S^0$  since $e(\alpha)$
is a non-zero divisor on $\cOcF$.

In general, we may choose codimension 1 subgroups $H_1, \ldots , H_c$
so that $K=H_1\cap \cdots \cap H_c$. Exactly as for the
polynomial ring, for a complex representation $V$ with $V^G=0$,  we
have  $e(V) \in \cOcF$ and we take 
$$\Kos (e(V))=\fibre (S^0\lra S^V). $$
We note that all terms are projective, and for a sequence of Euler classes
the Koszul complex is defined by tensoring these together.

\begin{lemma}
The Koszul complexes give a projective model for the cells: 
$$G/K_+\simeq Kos(e(H_1), \ldots , e(H_c)). $$
\end{lemma}

\begin{proof}
Since the Euler classes form a regular sequence in $\cOcF$,  the homology 
$\piA_*(Kos (e(H_1), \ldots , e(H_c))$ is calculated as a quotient of
$\piA_*(S^0)=\widetilde{\cOcF}$ (in the notation of Subsection \ref{subsec:defnAG}). This applies equally well in the
category of spectra, so that the homology of the Koszul complex agrees
with the image $\piA_*(G/K_+)$ of the spectrum $G/K_+$. Since $G/K_+$ 
is intrinsically formal by Corollary \ref{cor.cell.homology}, this completes the proof.  
\end{proof}

We will  use the following duality property, familiar in topology. 

\begin{cor}
\label{cor:DGK}
The algebraic cell $G/K_+$ is self dual: if $\codim (K)=c$ then  
$$DG/K_+\simeq \Sigma^{-c}G/K_+$$
\end{cor}

\begin{proof}
Since 
$$\Hom (S^V, X)=\Sigma^{-V}X=S^{-V}\tensor X, $$
the dual of the Koszul complex is the Koszul complex. However,  
a priori the shift is by a representation rather than an integer. To
see the representation may be
replaced by an integer, we note that if $\alpha$ is a one dimensional
representation with kernel $K$
$$\Sigma^{\alpha}\cOcF= \Sigma^2 e_{\cF K}\cOcF \oplus (1- e_{\cF
  K})\cOcF $$
where $\cF K$ consists of the finite subgroups of $K$, and $e_{\cF K}$
is the corresponding idempotent. It follows that 
$$D\Kos (e(\alpha))\cong \Sigma^{-1}\Kos (e(\alpha)) .$$ 
The general case follows by tensoring $c$ instances of this together. 
\end{proof}

\subsection{Proof of the inductive step }
\label{subsec:indstep}
We may in fact now follow the motivating pattern described in
Subsection \ref{subsec:motivation}. 

Suppose then  that  $H_*(X)(H)=0$ for $H$ of codimensions $<c$  and
that $\codim K=c$. We will show that
the fact $X$ is cellularly trivial means $H_*(X)(K)=0$.

By Corollary \ref{cor:DGK},  $[G/L_+, X]_*=[S^0, \Sigma^{-d}G/L_+\sm
X]_*$, where $L$ is of codimension $d$, so
that the hypothesis that $X$ is cellularly trivial
proves that $[S^0, A\sm X]_*=0$ for any cellular spectrum
$A$. We note that $A=\siftyV{K}$ is cellular. Indeed, it is the
localization which inverts $e(\alpha)$ for those one dimensional
representations $\alpha $ with $\alpha^K=0$ (i.e., $K\not \subseteq
\ker(\alpha)$).  Accordingly, 
$\fibre (S^0\lra \siftyV{K})$ is the cellularization of $S^0$ using 
cells $G/L_+$ where $L$ is in the family of subgroups not containing
$K$. (More explicitly, it is the homotopy colimit of stable Koszul
complexes $\Kos^{\infty} (\alpha_1, \ldots , \alpha_s)$ where $\alpha_i^K=0$)
The point of considering $\siftyV{K}\sm M$ is the isomorphism
$$\siftyV{K}\sm  X \cong f_K(\phi^KX). $$

\begin{lemma}
For any torsion  DG-$\cOcFK$-module $M$, 
$$[S^0, f_K(M)]\cong H_*(M). $$
\end{lemma}

\begin{proof}
Since $f_K$ is right adjoint to evaluation at $K$, and since 
this is compatible with resolutions, 
the Adams spectral sequence for $[T,f_K(M)]^G$ takes the simple form
$$E_2^{s,t}=
\Ext_{\cOcFK}^{*,*}(\phi^KH_*(T), H_*(M))\Rightarrow [T,X]^G_*. $$
In particular, taking $T=S^0$ and $X=f_K(M)$,  we have $\phi^KH_*(S^0)=\cOcFK$ and
$$[S^0, f_K(M)]=\Hom_{\cOcFK}(\cOcFK,
H_*(M))=H_*(M).\qedhere $$
\end{proof}

Finally we see
$$0=[S^0,\siftyV{K}\sm  X]_*=[S^0, f_K(\phi^KX)=H_*\phi^K X$$
as required. 

This completes the inductive step and hence the proof of Theorem
\ref{prop:cellequivisequiv}. \qqed

\end{document}